\documentclass{amsart}
\usepackage{amssymb,amsmath,mathrsfs,amsfonts}
\usepackage{amscd, amssymb, amsmath, amsthm, graphics,graphicx}
\usepackage{amsmath,amsfonts,amsthm,amssymb}
\usepackage{slashbox}
\usepackage{fancyhdr}
\usepackage[all]{xy}
\usepackage[colorlinks=true]{hyperref}
\usepackage[small,nohug,heads=vee]{diagrams}
\usepackage{multirow}
\diagramstyle[labelstyle=\scriptstyle]

\newtheorem{theorem}{Theorem}[section]
\newtheorem{lemma}[theorem]{Lemma}
\newtheorem{proposition}[theorem]{Proposition}

\theoremstyle{definition}
\newtheorem{definition}[theorem]{Definition}

\newtheorem{remark}[theorem]{Remark}

\begin{document}
\sloppy

\title[A characterization on separable subgroups of $3$-manifold groups]{A characterization on separable subgroups of $3$-manifold groups}

\author{Hongbin Sun}
\address{Department of Mathematics, Rutgers University - New Brunswick, Hill Center, Busch Campus, Piscataway, NJ 08854, USA}
\email{hongbin.sun@rutgers.edu}



\thanks{The author is partially supported by NSF Grant No. DMS-1510383.}

\date{\today}
\begin{abstract}
In this paper, we give a complete characterization on which finitely generated subgroups of finitely generated $3$-manifold groups are separable. Our characterization generalizes Liu's spirality character on $\pi_1$-injective immersed surface subgroups of closed $3$-manifold groups. A consequence of our characterization is that, for any compact, orientable, irreducible and $\partial$-irreducible $3$-manifold $M$ with nontrivial torus decomposition, $\pi_1(M)$ is LERF if and only if for any two adjacent pieces in the torus decomposition of $M$, at least one of them has a boundary component with genus at least $2$.
\end{abstract}

\maketitle
\vspace{-.5cm}
\section{Introduction}

For a group $G$ and a subgroup $H < G$, we say that $H$ is {\it separable} in $G$ if for any $g\in G\setminus H$, there exists a finite index subgroup $G'<G$ such that $H<G'$ and $g\notin G'$. A group $G$ is {\it LERF} (locally extended residually finite) or {\it subgroup separable} if all finitely generated subgroups of $G$ are separable.

The topological meaning of subgroup separability is given by the following heuristic: for any $\pi_1$-injective immersed compact object in a topological space $S\looparrowright M$ (e.g. a $\pi_1$-injective immersed compact subsurface in a $3$-manifold), if $\pi_1(S)$ is a separable subgroup of $\pi_1(M)$, then $S$ lifts to be embedded in some finite cover of $M$.

Since the study of fundamental group is one of the most central topic in $3$-manifold topology, subgroup separability is very important for $3$-manifolds. More precisely, subgroup separability is closely related with the virtual Haken conjecture of hyperbolic $3$-manifolds, which was raised by Thurston in \cite{Thu2} and settled by Agol in \cite{Agol2}. It is one of the biggest breakthrough on $3$-manifold topology in the past fifteen years.

For fundamental groups of $1$-dim and $2$-dim manifolds/complexes, it is known that free groups (\cite{Hall}) and surface groups (\cite{Sco2}) are LERF. For $3$-manifold groups, Scott showed that Seifert manifold groups are LERF (\cite{Sco2}), while the celebrated works of Agol and Wise showed that (finitely generated) hyperbolic $3$-manifold groups are LERF (\cite{Wise} proved the cusped case, see also \cite{GM}, and \cite{Agol2} proved the closed manifold case).

Although all $1$-dim manifold groups, $2$-dim manifold groups and hyperbolic $3$-manifold groups are LERF, many $3$-manifold groups are not LERF. The first such example is given by \cite{BKS}. The example in \cite{BKS} is a graph manifold, and it is homeomorphic to the mapping torus of a Dehn-twist on the one-punctured torus. Then Niblo and Wise (\cite{NW}) proved that all graph manifold groups are not LERF, by showing that all graph manifold groups contain a subgroup that is isomorphic to a nonLERF group given in \cite{BKS}. In \cite{RW}, Rubinstein and Wang gave the first nonseparable subgroup of a graph manifold group that is carried by a $\pi_1$-injective properly immersed subsurface. Inspired by \cite{RW} and armed with modern tools on hyperbolic $3$-manifold groups, Liu proved that many closed mixed $3$-manifold groups have nonseparable subgroups (\cite{Liu}), and these nonseparable subgroups are carried by $\pi_1$-injective immersed subsurfaces.

By using results in \cite{Agol2} and \cite{Wise}, and generalizing the construction in \cite{RW} and \cite{Liu}, the author proved all mixed $3$-manifolds have nonLERF fundamental groups. Among with the ground breaking work of Agol and Wise (\cite{Agol2}, \cite{Wise}), Scott's work on Seifert manifolds (\cite{Sco2}), Niblo-Wise's work on graph manifolds (\cite{NW}) and the author's work on mixed manifolds, we obtain the following result on $3$-manifolds with empty or tori boundary in \cite{Sun}.
\begin{theorem}\label{toriboundary}
  Let $M$ be a compact, orientable, irreducible $3$-manifold with empty or tori boundary. Then $M$ supports one of Thurston's eight geometries if and only if $\pi_1(M)$ is LERF.
\end{theorem}

In this paper, we address the LERFness of fundamental groups of all other $3$-manifolds ($3$-manifolds with boundary of genus at least $2$). The following theorem gives a complete characterization on LERFness of fundamental group for all compact orientable $3$-manifolds.

\begin{theorem}\label{allmanifold}
  Let $M$ be a compact, orientable, irreducible and $\partial$-irreducible $3$-manifold with nontrivial torus decomposition and does not support the $\text{Sol}$ geometry. Then $\pi_1(M)$ is LERF if and only if for any two adjacent pieces of the torus decomposition of $M$, at least one of them has a boundary component with genus at least $2$.

  Moreover, let $M$ be a compact orientable $3$-manifold, $\pi_1(M)$ is LERF if and only if each piece of the sphere-disc decomposition of $M$ either supports one of Thurston's eight geometries or satisfies the above condition of LERF fundamental group.
\end{theorem}

Actually, Theorem \ref{allmanifold} is a corollary of a more technical theorem (Theorem \ref{main}), which characterizes (finitely generated) separable subgroups of (finitely generated) $3$-manifold groups.

Let $M$ be a compact, orientable, irreducible and $\partial$-irreducible $3$-manifold. Since $3$-manifolds with trivial torus decomposition always have LERF fundamental groups (\cite{Sco2}, \cite{Wise}, \cite{Agol2}), we always assume the torus decomposition of $M$ is nontrivial. We also assume $M$ does not support $\text{Sol}$ geometry, since they have polycyclic fundamental groups and these groups are LERF (\cite{Rob} 5.4.16). For any finitely generated infinite index subgroup $H<G\cong\pi_1(M)$, we use $\pi:M_H\to M$ to denote the covering space of $M$ corresponding to $H<\pi_1(M)$, then $M_H$ has an induced graph of space structure.

Similar to \cite{Liu}, the obstruction on separability of $H<\pi_1(M)$ lies in the {\it almost fibered surface} $\Phi(H)$ associated to $H$, and $H$ is separable in $\pi_1(M)$ if and only if the {\it generalized spirality character} of $H$ is trivial. We will give the precise definition of almost fibered surface and generalized spirality character in Section \ref{spiralitycharacter}. Currently, the readers only need to know that the almost fibered surface $\Phi(H)$ is a (possibly disconnected) embedded compact subsurface of $M_H$, and the generalized spirality character is a homomorphism $s:H_1(\Phi(H);\mathbb{Z})\to \mathbb{Q}_+^{\times}$ to the multiplicative group of positive rational numbers. If $s:H_1(\Phi(H);\mathbb{Z})\to \mathbb{Q}_+^{\times}$ is the trivial homomorphism, we say that $H$ is {\it aspiral} in $\pi_1(M)$.

The following theorem gives a complete characterization on separable subgroups of $3$-manifold groups, where the $3$-manifold is compact, orientable, irreducible and $\partial$-irreducible. The characterization on subgroups of more general $3$-manifold groups directly follows from this result.

\begin{theorem}\label{main}
Let $M$ be a compact, orientable, irreducible and $\partial$-irreducible $3$-manifold with nontrivial torus decomposition and does not support the $\text{Sol}$ geometry. Then a finitely generated infinite index subgroup $H<\pi_1(M)$ is separable if and only if $H$ is aspiral in $\pi_1(M)$.
\end{theorem}

\bigskip

Now we give a sketch of the proof of Theorem \ref{main}.

To prove that separability implies aspirality, we will prove the following result: there exists an intermediate finite cover $M'\to M$ of $M_H\to M$, and a codim-$0$ submanifold $\hat{M}\subset M'$, such that $\Phi(H)\subset M_H$ is mapped to a properly embedded subsurface in $\hat{M}$. Then the properness of $\Phi(H)$ in $\hat{M}$ implies the generalized spirality character of $H$ is trivial as in \cite{Liu}. To prove this statement, we use the separability of $H<\pi_1(M)$ to find an intermediate finite cover $M'\to M$ of $M_H\to M$, such that the induced map $\Phi(H)\to M'$ is injective and it is "as nice as possible". If the induced map of $\Phi(H)\to M'$ on their dual graphs is injective, and for each Seifert piece $V\subset M'$, the projection from $\Phi(H)\cap V$ to the base orbifold of $V$ is injective, we can construct the submanifold $\hat{M}\subset M'$ easily. In general, neither of the above two assumptions are satisfied, and we need to take the minimizer of some complexities of $\Phi(H)\to M'$ to construct $\hat{M}$.

To prove that aspirality implies separability, we reduce to the case that $M$ has empty or tori boundary, by pasting hyperbolic $3$-manifolds with totally geodesic boundary to higher genus boundary components of $M$. For any compact subset $K\subset M_H$ in the covering space of $M$ corresponding to $H$, we construct a finite semi-cover $N\to M$ (Definition \ref{semicoverdef}) such that $K$ embeds into $N$, then we use the separability of finite semi-cover subgroups (Lemma 6.2 of \cite{Liu}) to extend the finite semi-cover $N\to M$ to a finite cover $M'\to M$.

Let $G_K$ be the dual graph of $K$ (induced by $M_H$), then for the finite semi-cover $N$ we need to construct, its dual graph is isomorphic to $G_K$. We will construct pieces of $N$ individually and paste them together along boundary. Since geometric $3$-manifolds have LERF fundamental groups, for each piece $V_H\subset M_H$ with $V_H\cap K\ne \emptyset$ that covers a piece $V\subset M$, there exists an intermediate finite cover $\hat{V}\to V$ of $V_H\to V$ such that $V_H\cap K$ embeds into $\hat{V}$. To paste these pieces together and get $N$, we need to control the restriction of covering maps on boundary components of $\hat{V}$. This can be done by a case-by-case construction on various types of pieces of $M$ and various types of their subgroups. The basic idea is that, for geometrically finite subgroups of hyperbolic $3$-manifold groups and subgroups of Seifert manifold groups that contain a power of the regular fiber, the restriction of $\hat{V}\to V$ on different boundary components are independent of each other. In the other cases, the restriction of $\hat{V}\to V$ on different boundary components actually do depend on each other. Then we use the aspirality of $H$ to make sure the restrictions of these covering maps on boundary components can be made compatible, so that all these pieces can be pasted together to get the desired finite semi-cover $N\to M$.

The organization of this paper is as the following. In Section \ref{pre}, we review necessary backgrounds in group theory, $3$-manifold topology and hyperbolic geometry. In Section \ref{spiralitycharacter}, we define the almost fibered surface and generalized spirality character, and prove their basic properties. In Section \ref{elegant}, we prove that if a subgroup is separable, then its generalized spirality character is trivial. In Section \ref{casebycase}, for each type of subgroups of hyperbolic or Seifert $3$-manifold group $L<\pi_1(N)$, we prove a result on the existence of certain intermediate finite cover of $N_L\to N$, with some freedom on its boundary components. In Section \ref{combinatorics}, by assuming $H$ is aspiral, we analyse possible combinatoric patterns of two adjacent pieces of $M_H$, and construct a finite semi-cover $N\to M$ such that the pre-selected compact subset $K\subset M_H$ embeds into it, which implies the separability of $H$. In Section \ref{3manifold}, we use Theorem \ref{main} to prove Theorem \ref{allmanifold}.

\bigskip

{\bf Acknowledgement:} The author thanks Daniel Groves for confirming Theorem \ref{allmanifold} for $3$-manifolds such that all of its pieces have higher genus boundary, before the project in this paper was started. The author is grateful for Feng Luo for comments on a previous version of this paper.

\section{preliminary}\label{pre}

In this paper, we assume all $3$-manifolds are connected and all groups are finitely generated, unless otherwise stated. In this section, we review basic material in group theory, $3$-manifold topology and hyperbolic geometry.

\subsection{Subgroup separability}\label{group}

In this subsection, we review definitions and basic properties on subgroup separability. The main concepts we will study in this paper are the following ones.

\begin{definition}\label{LERF}
Let $G$ be a group, and $H<G$ be a subgroup, we say that $H$ is {\it separable} in $G$ if for any $g\in G\setminus H$, there exists a finite index subgroup $G'<G$ such that $H<G'$ and $g\notin G'$.

A group $G$ is {\it LERF} (locally extended residually finite) or {\it subgroup separable} if all finitely generated subgroups of $G$ are separable in $G$.
\end{definition}

The topological meaning of separable subgroups is given by Scott (\cite{Sco2} Lemma 1.4), and we reinterpret Scott's result as the following.

\begin{lemma}[\cite{Sco2}]\label{Scott}
Let $X$ be a Hausdorff space with fundamental group $\pi_1(X)\cong G$, and $\pi:\tilde{X}\to X$ be a covering space of $X$ corresponding to a subgroup $H<G$. Then $H$ is separable in $G$ if and only if for any compact subset $K\subset \tilde{X}$, there exists an intermediate finite cover $q:X'\to X$ of $\pi:\tilde{X}\to X$ as the following diagram, such that $p|_K:K\to X'$ is an embedding.
\begin{diagram}
\tilde{X} &\rTo^p & X'\\
&\rdTo_{\pi} &\dTo_{q}\\
& &X
\end{diagram}
\end{lemma}

By definition, all finite index subgroups are separable, so we will focus on finitely generated infinite index subgroups in this paper. Here we list a few elementary properties on separability of subgroups that will be used in this paper. All these statements are classical and well-known, and their proofs are all algebraic.

\begin{lemma}\label{basicLERF}
  \begin{enumerate}
    \item Let $G$ be a group, $H<G$ and $\Gamma<G$ be subgroups. If $H$ is separable in $G$, then $H\cap \Gamma$ is separable in $\Gamma$.
    \item Let $G$ be a group, $H<G$ be a subgroup and $\Gamma<G$ be a finite index subgroup, then $H$ is separable in $G$ if and only if $H\cap \Gamma$ is separable in $\Gamma$.
    \item Let $G$ be a group, $H<G$ be a subgroup and $H'<H$ be a finite index subgroup. If $H'$ is separable in $G$, then $H$ is separable in $G$.
    \item Let $G$ be a group, and $\Gamma<G$ be a subgroup. If $G$ is LERF, then $\Gamma$ is LERF. Moreover, if $\Gamma<G$ has finite index, then $G$ is LERF if and only if $\Gamma$ is LERF.
    \item Let $G_1*G_2$ be a free product of residually finite groups and $H<G_1*G_2$ be a finitely generated subgroup, then $H\cong (*_{i=1}^m K_i)*(*_{j=1}^n \mathbb{Z})$ and each $K_i$ equals $H\cap g_i G_{k_i}g_i^{-1}$ for some $k_i\in \{1,2\}$. Moreover, $H$ is separable in $G_1*G_2$ if and only if $K_i$ is separable in $g_i G_{k_i}g_i^{-1}$ for all $i$. In particular, $G_1*G_2$ is LERF if and only if both $G_1$ and $G_2$ are LERF.
  \end{enumerate}
\end{lemma}

\begin{proof}
  \begin{enumerate}
  \item For any $g\in \Gamma \setminus H\cap \Gamma$, $g\in G \setminus H$ holds. Then the separability of $H$ in $G$ implies that there is a finite index subgroup $G'<G$ containing $H$ and $g\notin G'$. Then $G'\cap \Gamma$ is a finite index subgroup of $\Gamma$ such that $H\cap \Gamma<G'\cap \Gamma$ and $g\notin G'\cap \Gamma$.
  \item Recall that $H$ is separable in $G$ if and only if $H$ is a closed subgroup under the profinite topology of $G$. Since $H$ is a finite union of left cosets of $H\cap \Gamma$, and the profinite topology on $\Gamma$ coincides with subspace topology on $\Gamma$ induced from the profinite topology on $G$, $H$ is closed in $G$ if and only if $H\cap \Gamma$ is closed in $\Gamma$.
  \item It also follows from the profinite topology argument in (2), since $H$ is a finite union of left cosets of $H'$.
  \item For any finitely generated subgroup $H<\Gamma$, since $G$ is LERF, $H$ is separable in $G$. Then part (1) implies that $H=H\cap \Gamma$ is separable in $\Gamma$. So LERFness of $G$ implies LERFness of $\Gamma$. Now we suppose that $\Gamma<G$ has finite index and $\Gamma$ is LERF. Let $H$ be any finitely generated subgroup of $G$, then $H\cap \Gamma$ is a finitely generated subgroup of $\Gamma$, so it is separable in $\Gamma$. Then part (2) implies that $H$ is separable in $G$. So in the finite index case, LERFness of $\Gamma$ implies LERFness of $G$.
  \item The first statement follows from the Kurosh subgroup theorem. The "if" part follows from \cite{Bur}, and the "only if" part follows from (1). $G_1*G_2$ is LERF if and only if both $G_1$ and $G_2$ are LERF follows from the fact that a free product (of finitely many groups) is finitely generated if and only if all of its free factors are finitely generated.
  \end{enumerate}
\end{proof}

\subsection{$3$-manifold topology}\label{3mfld}

In this subsection, we review basic topology of $3$-manifolds, and explain how to reduce the study on separability of subgroups from all finitely generated $3$-manifold groups to the case that the $3$-manifold is compact, orientable, irreducible and $\partial$-irreducible.

The goal of this paper is to study finitely generated $3$-manifold groups and their finitely generated subgroups. It follows from \cite{Sco1}, \cite{Hem2} and the geometrization of $3$-manifolds (\cite{Per1}, \cite{Per2}, see also \cite{BBMBP}, \cite{KL}, \cite{MT}) that all finitely generated $3$-manifold groups are residually finite (i.e. the trivial subgroup is separable). For a recent work on constructing infinitely generated $3$-manifold groups with non-residually finite fundamental groups, see \cite{CS}.

Since most tools in $3$-manifold topology work better for compact $3$-manifolds, we prefer to restrict ourselves to fundamental groups of compact $3$-manifolds (these groups are automatically finitely generated). This follows from the Scott core theorem in \cite{Sco1}, which says that any $3$-manifold $M$ with finitely generated fundamental group contains a compact codim-$0$ submanifold $C_M$, such that the inclusion $C_M\to M$ induces isomorphism on fundamental groups.

By taking the orientable double cover if necessary, we can further restrict to compact orientable 3-manifolds, since LERFness and separability of subgroups are well-behaved under taking finite index subgroups (Lemma \ref{basicLERF} (2) (4)). Then we apply the sphere and disc decompositions to $M$, and get irreducible and $\partial$-irreducible pieces of $M$. The general case follows from the irreducible, $\partial$-irreducible case since LERFness and separability of subgroups are well-behaved under taking free product (Lemma \ref{basicLERF} (5)).

So in the following, we assume all $3$-manifolds are compact, orientable, irreducible and $\partial$-irreducible, unless otherwise stated. The next decomposition we will do is the torus decomposition, which basically follows from the theory of characteristic submanifolds (\cite{JS}, \cite{Joh}). However, since we do not want to do the annulus decomposition, the precise statement we will use is the following one (see Theorem 1.9 of \cite{Hat}).

\begin{theorem}[\cite{Hat}]\label{torusdecomposition}
For a compact, orientable, irreducible $3$-manifold $M$, there exists a finite collection $\mathcal{T}\subset M$ of disjoint incompressible tori, such that each component of $M\setminus \mathcal{T}$ is either atoroidal or a Seifert manifold, and a minimal such collection $\mathcal{T}$ is unique up to isotopy.
\end{theorem}

We always assume the torus decomposition of $M$ is non-trivial, otherwise $M$ supports a geometric structure and $\pi_1(M)$ is LERF (\cite{Sco2}, \cite{Wise}, \cite{Agol2}). We also assume that $M$ does not support the $\text{Sol}$ geometry. In this case, $M\setminus \mathcal{T}$ consists of one copy of $T^2\times I$ or two copies of the twisted $I$-bundle over Klein bottle, and $\pi_1(M)$ is LERF. In the following, we will call each component of $M\setminus \mathcal{T}$ a {\it vertex space} or a {\it piece} of $M$, and call each component of $\mathcal{T}$ an {\it edge space} or a {\it decomposition torus} of $M$. Then no piece of $M$ is homeomorphic to $T^2\times I$.

For each atoroidal piece $V\subset M\setminus \mathcal{T}$ that is not the twisted $I$-bundle over Klein bottle, it is proved in \cite{Thu3} (see also \cite{Mor}) that the interior of $V$ supports a complete hyperbolic structure. The hyperbolic structure has finite volume if and only if all boundary components of $V$ are homeomorphic to the torus. When $V$ has a boundary component with higher genus (genus at least $2$), there exists a geometrically finite hyperbolic structure on $\text{int}(V)$ such any cusp ends of its convex core is a torus cusp (i.e. it has no accidental parabolic). We will always choose such a geometrically finite hyperbolic structure on $V$. Moreover, we can choose the hyperbolic structure such that its convex core has totally geodesic boundary if and only if $V$ is acylindrical.

For each Seifert piece $V\subset M$, it has a unique Seifert structure ($S^1$-bundle over $2$-orbifold structure), unless $V$ is homeomorphic to the twisted $I$-bundle over Klein bottle (Theorem 2.3 of \cite{Hat}). The twisted $I$-bundle over Klein bottle has exactly two Seifert structures. For one of them, the base orbifold is the M\"{o}bius band, and the base orbifold of the other Seifert structure is a disc with two index-$2$ singular points. When taking a finite cover of $M$, a twisted $I$-bundle over Klein bottle in $M$ may have an elevation (a component of its preimage) homeomorphic to $T^2\times I$, and it is the only case that the torus decomposition of $M$ does not lift to the torus decomposition of its finite cover.

For $3$-manifolds as in Theorem \ref{torusdecomposition}, those manifolds with empty or tori boundary attract more attention of $3$-manifold topologists. In this case, irreducible automatically implies $\partial$-irreducible, unless $M$ is the solid torus. For such an $M$, the above torus decomposition is the JSJ decomposition, and it is closely related to the geometric decomposition. The only difference  is that, if $M$ contains a twisted $I$-bundle over Klein bottle, the torus decomposition decomposes along the boundary torus of this $I$-bundle, and the geometric decomposition decomposes along the zero section of this bundle (a Klein bottle). We can always get rid of this difference by taking a double cover of $M$ such that the preimage of each twisted $I$-bundle over Klein bottle is homeomorphic to $T^2\times I$.

For compact, orientable, irreducible $3$-manifolds with empty or tori boundary and nontrivial torus decomposition, the concept of semi-cover was introduced in \cite{PW}, and it played important role in the study on subgroup separability of $3$-manifold groups.

\begin{definition}\label{semicoverdef}
Let $M$ be a compact, orientable, irreducible $3$-manifold with empty or tori boundary and nontrivial torus decomposition. A {\it semi-cover} of $M$ is a $3$-manifold $N$ and a local embedding $f:N\to M$, such that the restriction of $f$ on each boundary component of $N$ is a covering map to a decomposition torus or a boundary component of $M$.

Moreover, $f:N\to M$ is a {\it finite semi-cover} if $f:N\to M$ is a semi-cover and $N$ is compact.
\end{definition}

If $f:N\to M$ is a finite semi-cover, then all boundary components of $N$ are homeomorphic to tori. Let $\mathcal{T}_M$ be the set of decomposition tori of $M$, then $f^{-1}(\mathcal{T}_M)$ is a collection of $\pi_1$-injective tori in $N$, and each component of $N\setminus f^{-1}(\mathcal{T}_M)$ is a finite cover of a component of $M\setminus \mathcal{T}_M$. Two components of $f^{-1}(\mathcal{T}_M)$ are parallel to each other only if they are mapped to the same torus $T\in \mathcal{T}_M$ and $T$ bounds a twisted $I$-bundle over Klein bottle in $M$. If $M$ does not contain the twisted $I$-bundle over Klein bottle, then $f^{-1}(\mathcal{T}_M)\setminus \partial N$ gives the torus decomposition of $N$.

One important property of finite semi-covers is given by Lemma 6.2 of \cite{Liu}.
\begin{lemma}[\cite{Liu}]\label{semicover}
If $N$ is a connected finite semi-cover of a $3$-manifold $M$ with empty or tori boundary, then $N$ has an embedded lifting in a finite cover of $M$. In fact, the semi-covering map $N\to M$ is $\pi_1$-injective and $\pi_1(N)$ is separable in $\pi_1(M)$.
\end{lemma}

For a compact, orientable, irreducible, $\partial$-irreducible $3$-manifold $M$, its torus decomposition induces a graph of space structure on $M$. For any finitely generated subgroup $H<\pi_1(M)$, the covering space $\pi:M_H\to M$ corresponding to $H$ has an induced graph of space structure, and we use $G_H$ to denote its dual graph. We may sometimes implicitly thicken $\pi^{-1}(\mathcal{T}_M)$, so that the projection $M_H\to G_H$ maps each vertex space of $M_H$ to a vertex of $G_H$ and maps each edge space of $M_H$ to an edge of $G_H$.

Each edge space of $M_H$ is a covering space of a decomposition torus of $M$, so it is homeomorphic to a torus, a cylinder, or a plane. Since $H$ is finitely generated, there exists a finite subgraph $G_H^c\subset G_H$, such that the submanifold $M_H^c\subset M_H$ corresponding to $G_H^c$ carries $\pi_1(M_H)$. Moreover, since each edge space of $M_H$ has finitely generated fundamental group ($\{e\}$, $\mathbb{Z}$ or $\mathbb{Z}^2$), an argument as in Theorem 4 of \cite{KS} implies that all vertex spaces of $M_H$ have finitely generated fundamental groups.

\subsection{Subgroups of hyperbolic $3$-manifold groups and Seifert manifold groups}\label{subgroups}

By the end of previous subsection, we have seen that each piece of $M_H$ has finitely generated fundamental group. In this section, we review all possible types of finitely generated subgroups of hyperbolic $3$-manifold groups and Seifert manifold groups.

We first deal with the hyperbolic case. When a hyperbolic $2$- or $3$-manifold has cusps, we may not distinguish the cusped open manifold and the manifold obtained by truncating its cusps. Let $N$ be a finite volume hyperbolic $3$-manifold and $L<\pi_1(N)$ be a finitely generated subgroup, then one of the following hold.
\begin{enumerate}
  \item $L<\pi_1(N)$ is a finite index subgroup, and the covering space $N_L\to N$ corresponding to $L$ is a finite cover of $N$.
  \item $L<\pi_1(N)$ is a geometrically finite subgroup of $\pi_1(N)$, i.e. the convex core of $N_L$ has finite volume.
  \item $L<\pi_1(N)$ is a geometrically infinite subgroup of $\pi_1(N)$, which is equivalent to that $L$ is a virtually fibered subgroup (\cite{Thu1}, \cite{Can} and \cite{Agol1}, \cite{CG}).
\end{enumerate}

In the geometrically infinite case, the covering space $N_L$ of $N$ is homeomorphic to $\Sigma_L\times \mathbb{R}$ or $\Sigma_L\tilde{\times} \mathbb{R}$ (twisted $\mathbb{R}$-bundle). Here $\Sigma_L$ is an orientable surface in the first case and nonorientable in the second case.

If $N$ is a hyperbolic $3$-manifold with higher genus boundary, it supports a geometrically finite hyperbolic structure with infinite volume. Under this hyperbolic structure on $N$, any finitely generated subgroup of $\pi_1(N)$ is still geometrically finite.

\bigskip

Now we consider the case that $N$ is a Seifert manifold, then $N$ is an $S^1$-bundle over a $2$-orbifold $\mathcal{O}_N$. We first suppose that $N$ is not the twisted $I$-bundle over Klein bottle, so that the Seifert structure on $N$ is unique. For any finitely generated subgroup $L<\pi_1(N)$, we have the following possibilities, according to the induced bundle structure on $N_L$.

\textbf{Case I.} The induced bundle structure on $N_L$ is an $S^1$-bundle.

Case I.1. The base space of $N_L$ is a compact $2$-orbifold, then $N_L\to N$ is a finite cover and $L<\pi_1(N)$ is a finite index subgroup.

Case I.2. The base space of $N_L$ is a noncompact $2$-orbifold.

\textbf{Case II.} The induced bundle structure on $N_L$ is an $\mathbb{R}$-bundle, then the base space $\mathcal{O}_{N_L}$ of $N_L$ is a surface (with no singular points).

Case II.1. The base space of $N_L$ is a compact surface, then $N_L$ is homeomorphic to $\Sigma_L\times \mathbb{R}$ or $\Sigma_L\tilde{\times} \mathbb{R}$ for some compact surface $\Sigma_L$. In this case, $L$ is a virtually fibered subgroup of $\pi_1(N)$.

Case II.2. The base space of $N_L$ is a noncompact surface, and $N_L$ is homeomorphic to $\mathcal{O}_{N_L}\times \mathbb{R}$ or $\mathcal{O}_{N_L}\tilde{\times} \mathbb{R}$ for some noncompact surface $\mathcal{O}_{N_L}$. In this case, we call $L$ a {\it partially fibered subgroup} of $\pi_1(N)$.

In Case II.1, $\Sigma_L$ is naturally a compact subsurface of $N_L$. In Case II.2, if $L$ is nontrivial, we construct a compact subsurface $\Sigma_L \subset N_L$ in the following, such that the inclusion induces isomorphism on fundamental groups. Then $\Sigma_L\subset N_L$ is the {\it partially fibered subsurface} corresponding to the {\it partially fibered subgroup} $L<\pi_1(N)$.

Since we excluded the case that $N$ is the twisted $I$-bundle over Klein bottle, the Seifert structure on $N$ is unique. So the base orbifold $\mathcal{O}_N$ of $N$ is unique, and its interior supports a finite area cusped hyperbolic structure. We fix such a finite area hyperbolic structure on the interior of $\mathcal{O}_N$, and also identify $\mathcal{O}_N$ with a truncation of this cusped hyperbolic surface, then $\mathcal{O}_{N_L}$ has an induced hyperbolic structure. We take the convex core of $\mathcal{O}_{N_L}$, and denote it by $\Sigma_L$. Then $\Sigma_L$ is a compact subsurface of $\mathcal{O}_{N_L}$, and each component of $\mathcal{O}_{N_L}\setminus \Sigma_L$ is neither homeomorphic to the disc nor homeomorphic to the annulus. There always exists an embedded section $\Sigma_L\hookrightarrow N_L$ of the $\mathbb{R}$-bundle $N_L\to \mathcal{O}_{N_L}$, and this section is unique up to isotopy of sections. So we can consider $\Sigma_L$ as a subsurface of $N_L$.

If $L$ corresponds to a parabolic subgroup of $\pi_1(\mathcal{O}_N)$ isomorphic to $\mathbb{Z}$, the convex core is empty and we modify the definition of $\Sigma_L$ as the following. We take $\Sigma_L$ to be a closed $\epsilon$-neighborhood of the boundary component of $\mathcal{O}_{N_L}$ corresponding to $L$. When $\epsilon>0$ is small enough, $\Sigma_L$ is homeomorphic to the annulus and the projection from $\Sigma_L\subset \mathcal{O}_{N_L}$ to its image in $\mathcal{O}_N$ is a covering map.

If $N$ is a twisted $I$-bundle over Klein bottle, one of the following hold, and we only need to construct $\Sigma_L$ in the third case.
\begin{enumerate}
  \item $L<\pi_1(N)$ has finite index, i.e. $N_L\to N$ is a finite cover (as in Case I.1).
  \item $L<\pi_1(N)$ is trivial, then we do not need to construct $\Sigma_L$ (as in Case II.2).
  \item $L<\pi_1(N)$ is a nontrivial infinite index subgroup. Then $L\cong \mathbb{Z}$ and at least one of the Seifert structure on $N$ lifts to an $\mathbb{R}$-bundle structure on $N_L$ (whose base orbifold is homeomorphic to the annulus or M\"{o}bius band). Then we take $\Sigma_L$ to be a section of this $\mathbb{R}$-bundle. If $N_L$ has two induced $\mathbb{R}$-bundle structures, $\Sigma_L$ is transverse with both of them. (This corresponds to Case II.1.)
\end{enumerate}

\bigskip

In the partially fibered case, only a proper subset of boundary components of $\Sigma_L$ are mapped to the boundary of $N_L$, and $\Sigma_L$ do have some boundary components that are mapped to the interior of $N_L$. Those boundary components of $\Sigma_L$ that are mapped to $\partial N_L$ are more important for our definition of almost fibered surface and generalized spirality character. Actually, if no boundary component of $\Sigma_L$ is mapped to $\partial N_L$, then $\Sigma_L$ has no contribution to the almost fibered surface $\Phi(H)$.

The partially fibered case is the main difference between this work and the work of Liu (\cite{Liu}): the subgroup associated to an immersed $\pi_1$-injective subsurface in a closed $3$-manifold does not have any partially fibered piece.

\section{Generalized spirality character}\label{spiralitycharacter}

In this section, we define the main object in this paper: the generalized spirality character. It is a generalization of Liu's spirality character (\cite{Liu}) on $\pi_1$-injective immersed subsurfaces in closed, orientable, irreducible $3$-manifolds, and it is also rooted in the work of Rubinstein-Wang on graph manifolds (\cite{RW}).

The main difference between our generalized spirality character and the spirality character in \cite{Liu} is that we need to consider partially fibered subgroups in Seifert pieces. For example, the nonseparable subgroup of $3$-manifold group constructed in \cite{BKS} is a subgroup of a graph manifold group $H<\pi_1(M)$, such that the induced graph of group structure on $H$ has one vertex group and one edge group, and the vertex group is actually a partially fibered subgroup. Moreover, in \cite{Liu}, the almost fibered surface is naturally a subsurface of the given immersed surface. While in our case, we need to extract the "almost fibered surface" from the subgroup $H<\pi_1(M)$.

Recall that we always assume $M$ is a compact, orientable, irreducible, $\partial$-irreducible $3$-manifold with nontrivial torus decomposition, and does not support the $\text{Sol}$ geometry.

\subsection{The almost fibered surface $\Phi(H)$}\label{almostfiber}

In this subsection, for any finitely generated infinite index subgroup $H<\pi_1(M)$, we define the almost fibered surface $\Phi(H)$.

Since the fundamental group is defined with respect to a base point, we always implicitly take a base point $*\in M$, and identify $H$ with a subgroup of $\pi_1(M,*)$. In many cases, we need to change to another base point $*'\in M$. Then we take a path $\gamma$ in $M$ from $*'$ to $*$ and consider subgroup $\gamma H\gamma^{-1}< \pi_1(M,*')$. Then $H$ is separable in $\pi_1(M,*)$ if and only if $\gamma H\gamma^{-1}$ is separable in $\pi_1(M,*')$. So we will often abuse notation and still use $H$ to denote $\gamma H\gamma^{-1}$.

Let $\pi:M_H\to M$ be the covering space of $M$ corresponding to $H$, then the torus decomposition on $M$ induces a graph of space structure on $M_H$. Each elevation of a piece of $M$ in $M_H$ is called a {\it vertex space} or a {\it piece} of $M_H$, and each elevation of a decomposition torus of $M$ in $M_H$ is called an {\it edge space} of $M_H$. We denote the dual graph of $M_H$ by $G_H$. For each vertex $v$ of $G_H$, the corresponding piece $M_H^v$ in $M_H$ is a covering space of a piece of $M$ with finitely generated fundamental group. Since $H$ is finitely generated and $M$ does not support the $\text{Sol}$-geometry, only finitely many pieces of $M_H$ have nontrivial fundamental groups.

Let $M_H^c$ be the minimal connected submanifold of $M_H$, such that it is a union of pieces of $M_H$, it contains all pieces of $M_H$ with nontrivial fundamental groups, and the inclusion $M_H^c\hookrightarrow M_H$ induces isomorphism on fundamental groups. Since $H$ is finitely generated, $M_H^c$ has only finitely many pieces. Let $G_H^c$ be the subgraph of $G_H$ dual with $M_H^c$, then $G_H^c$ is a finite subgraph of $G_H$ and each component of $G_H\setminus G_H^c$ is a tree.

For a vertex $v\in G_H^c$, we use $M^v$ to denote the image of $M_H^v\subset M_H$ in $M$. The almost fibered surface $\Phi(H)$ is constructed by pasting a sub-collection of following subsurfaces living in pieces of $M_H^c$.
\begin{enumerate}
  \item If $M_H^v$ corresponds to a virtually fibered subgroup of $M^v$ ($M^v$ is either hyperbolic or Seifert), then $M_H^v$ is homeomorphic to $\Sigma^v\times \mathbb{R}$ or $\Sigma^v\tilde{\times} \mathbb{R}$ for some compact surface $\Sigma^v$. Then we take an embedded copy of $\Sigma^v$ in $M_H^v$.
  \item If $M^v$ is a Seifert space and $M_H^v$ corresponds to a partially fibered subgroup of $M^v$, then in Case II.2 of Section \ref{subgroups}, we constructed a partially fibered subsurface $\Sigma^v\subset M_H^v$. In this case, $\Sigma^v$ is a compact embedded subsurface in $M_H^v$, but it is not properly embedded.
\end{enumerate}

By our construction, for any two distinct components of $\partial \Sigma^v\cap \partial M_H^v$, they lie in distinct boundary components of $M_H^v$. In case (1), each boundary component of $M_H^v$ is homeomorphic to the cylinder ($C=S^1\times \mathbb{R}$), and each of them intersects with exactly one component of $\partial \Sigma^v$. In case (2), some boundary components of $M_H^v$ are homeomorphic to the cylinder and some are homeomorphic to the plane ($P=\mathbb{R}^2$). Since $\Sigma^v$ is given by the convex core construction, a boundary component of $M_H^v$ intersects with $\Sigma^v$ if and only if it is homeomorphic to the cylinder.

Let $\Sigma_1$ and $\Sigma_2$ be two fibered or partially fibered subsurfaces in $M_H^c$, such that their boundary components $c_1\subset \partial\Sigma_1$ and $c_2\subset \partial\Sigma_2$ lie in the same cylinder $C\subset M_H^c$. Since $c_1$ and $c_2$ are isotopic to each other in $C$, we can and will assume they are mapped to the same simple closed curve in $C$.

Let $J_H^c$ be the set of vertices $v\in G_H^c$ such that the following hold.
\begin{enumerate}
\item The vertex $v$ corresponds to a piece $M_H^v\subset M_H^c$ that is virtually fibered or partially fibered.
\item Moreover, if $M_H^v$ corresponds to a partially fibered subgroup of a Seifert piece in $M$, the partially fibered subsurface $\Sigma^v$ is non-empty and does not lie in the interior of $M_H^v$.
\end{enumerate}

In other words, the second condition excludes those partially fibered pieces $M_H^v$ with only plane boundary components. In particular, all vertices $v$ with $\pi_1(M_H^v)=\{e\}$ are excluded.

\begin{definition}\label{almostfiberdef}
Let $\{\Sigma^v\}_{v\in J_H^c}$ be the (finite) collection of virtually fibered and partially fibered subsurfaces defined for vertices $v\in J_H^c$. For each cylinder edge space $C\subset M_H$ that intersects with the surfaces in $\{\Sigma^v\}_{v\in J_H^c}$ along exactly two circles, we paste these surfaces along their intersections with $C$. After doing all these pasting, we get the {\it almost fibered surface} $\Phi(H)$, and it is naturally a subsurface of $M_H^c(\subset M_H)$.
\end{definition}

Note that $\Phi(H)$ might be disconnected. By the above pasting process, $\Phi(H)$ has a natural graph of space structure. We call each $\Sigma^v$ a {\it piece} or a {\it vertex space} of $\Phi(H)$, and call each circle component of $\Sigma^v\cap \Sigma^{v'}$ an {\it edge space} of $\Phi(H)$.  Let $G_{\Phi(H)}$ be the  dual graph of $\Phi(H)$, then the inclusions $\Phi(H)\to M_H^c$ and $\Phi(H)\to M_H$ induce injective maps on their dual graphs.

Let $M'\to M$ be a finite cover and let $H'=H\cap \pi_1(M')$, it is easy to check that the almost fibered surface $\Phi(H')$ (defined by $H'<\pi_1(M')$) is a finite cover of $\Phi(H)$. There are two graph of space structures on $\Phi(H')$. One of them is induced by the graph of space structure on $M'$, and the other one is induced by the covering map $\Phi(H')\to \Phi(H)$. If $M$ does not contain the twisted $I$-bundle over Klein bottle, then these two graph of space structures on $\Phi(H')$ are same with each other. Otherwise, in the graph of space structure on $\Phi(H')$ induced from $\Phi(H')\to \Phi(H)$, there might be some annulus vertex spaces that are mapped into some $T^2\times I\subset M'$. If we pinch all these annuli to circles, we obtain the graph of space structure on $\Phi(H')$ induced by $M'$.

\subsection{On oriented proper paths in $\Sigma^v\subset \Phi(H)$}\label{arcs}

In this subsection, we do some preparation for defining the generalized spirality character. For a virtually fibered or partially fibered subsurface $\Sigma^v\subset \Phi(H)\subset M_H$ and an oriented proper path $\delta:I\to \Sigma^v$ with $\delta(0),\delta(1)\in \partial \Sigma^v$, we will define a rational number $s_{\delta}\in \mathbb{Q}_+$ associated to $\delta$.

Let $M_H^v$ be the piece of $M_H$ containing $\Sigma^v$, and let $M^v$ be the corresponding piece of $M$. Let $c^{\text{ini}}$ be the boundary component of $\Sigma^v$ containing $\delta(0)$ and $c^{\text{ter}}$ be the boundary component of $\Sigma^v$ containing $\delta(1)$.

\textbf{Case I.} We first consider the case that $M^v$ is a Seifert manifold. If $M^v$ is not the twisted $I$-bundle over Klein bottle, then it has a unique Seifert structure and we define $$s_{\delta}=\frac{\langle c^{\text{ini}},h\rangle}{\langle c^{\text{ter}},h \rangle}.$$ Here $h$ is the regular fiber of the Seifert space $M^v$, and $\langle \cdot, \cdot\rangle$ is the (unsigned) algebraic intersection number in the corresponding boundary components of $M^v$. If $M^v$ is the twisted $I$-bundle over Klein bottle, and we apply the above formula to any Seifert structure of $M^v$ that transverses with $\Sigma_L$, we always have $s_{\delta}=1$.

Note that this definition of $s_{\delta}$ is similar to the definition in \cite{RW}.

\textbf{Case II.} If $\Sigma^v\subset \Phi(H)$ is a virtually fibered subsurface, the definition of $s_{\delta}$ is more complicated, and is is similar to the definition in \cite{Liu}. We first take a triple $(S^v,\phi^v,\{\hat{c}_i\}_{i \in I^v})$ such that the following conditions hold.
\begin{enumerate}
\item If $\Sigma^v$ is orientable, $S^v$ is homeomorphic to $\Sigma^v$. Otherwise, $S^v$ is the orientable double cover of $\Sigma^v$.
\item If $M^v$ is a Seifert space, then $\phi^v:S^v\to S^v$ is the identity. If $M^v$ is hyperbolic, then $\phi^v:S^v\to S^v$ is under the canonical form of pseudo-Anosov maps (i.e. it preserves two transverse measured singular foliations) and its restriction on $\partial S^v$ is identity.
\item The mapping torus $J^v=S^v\times I/(x,0)\sim (\phi^v(x),1)$ is equipped with a finite covering map $J^v\to M^v$, such that $S^v\times \{0\}$ is mapped to $\Sigma^v\subset M^v$ via a homeomorphism or a double cover. By (2), distinct boundary components of $S^v$ intersect with distinct boundary components of $J^v$.
\item Let $\{c_i\}_{i\in I^v}$ be the set of all boundary components of $\Sigma^v$, then each $\hat{c}_i$ is an elevation of $c_i$ in $S^v$.
\end{enumerate}

For an oriented proper path $\delta: I\to \Sigma^v$, we define $s_{\delta}$ by the following way.

Let $T^{\text{ini}}$ be the boundary torus of $M^v$ that contains $c^{\text{ini}}$, let $\hat{c}^{\text{ini}}$ be the boundary component of $S^v$ in $\{\hat{c}_i\}_{i \in I^v}$ corresponding to $c^{\text{ini}}$, and let $\hat{T}^{\text{ini}}$ be the boundary component of $J^v$ containing $\hat{c}^{\text{ini}}$. Similarly, we define $T^{\text{ter}}\subset \partial M^v$, $\hat{c}^{\text{ter}}\subset \partial S^v$ and $\hat{T}^{\text{ter}}\subset \partial J^v$ by $c^{\text{ter}}$. Then we define $$s_{\delta}=\frac{[\hat{T}^{\text{ini}}:T^{\text{ini}}]}{[\hat{T}^{\text{ter}}:T^{\text{ter}}]}.$$ Here $[\cdot : \cdot]$ denotes the covering degree.

In the definition of $s_{\delta}$, we made some choices, and we need to prove that $s_{\delta}$ is well-defined. The proof is similar to the argument in Section 4.2 of \cite{Liu}. When $\Sigma^v$ is orientable, the only ambiguity is on the surface automorphism $\phi^v$, since we can take powers of $\phi^v$. When we pass from $\phi^v$ to $(\phi^v)^n$, the covering degrees $[\hat{T}^{\text{ter}}:T^{\text{ter}}]$ and $[\hat{T}^{\text{ini}}:T^{\text{ini}}]$ are both multiplied by $n$, so $s_{\delta}$ is well defined.

When $\Sigma^v$ is nonorientable, taking a power of $\phi^v:S^v\to S^v$ still does not affect $s_{\delta}$, by the same argument as above. Moreover, since each $c_i\subset \Sigma^v$ has exactly two elevations in $S^v$, there are exactly two possible choices of $\hat{T}^{\text{ini}}$ (and $\hat{T}^{\text{ter}}$). Then up to taking a power of $\phi^v$, these two choices of $\hat{T}^{\text{ini}}$ differ by a deck transformation of $J^v\to M^v$, and the same argument works for $\hat{T}^{\text{ter}}$. So $s_{\delta}$ is well defined.

It is also easy to check that, if $M^v$ is a twisted $I$-bundle over Klein bottle and we apply the above definition, $s_{\delta}=1$ always holds.

\bigskip

In the case that $M^v$ is a Seifert space and $M_H^v$ corresponds to a virtually fibered subgroup of $M^v$, Case I and Case II give two definitions of $s_{\delta}$. The following lemma proves that these two definitions of $s_{\delta}$ give us the same number.

\begin{lemma}\label{welldefine}
Suppose $M^v$ is a Seifert space and $M_H^v$ corresponds to a virtually fibered subgroup of $M^v$. Let $\delta: I\to \Sigma^v$ be an oriented proper path, then the two definitions of $s_{\delta}$ in Case I and Case II are equal to each other.
\end{lemma}

\begin{proof}
  Let $J^v$ be a finite cover of $M^v$ homeomorphic to $S^v\times S^1$ as in Case II, and $q:J^v\to M^v$ be the covering map. Let $c^{\text{ini}}, c^{\text{ter}}\subset \partial \Sigma_v$, $\hat{c}^{\text{ini}}, \hat{c}^{\text{ter}}\subset \partial S_v$, $\hat{T}^{\text{ini}},\hat{T}^{\text{ter}}\subset \partial J^v$, $T^{\text{ini}},T^{\text{ter}}\subset \partial M^v$ be defined as in Case II. Let $\hat{h}$ and $h$ be the regular fibers of $J^v$ and $M^v$ respectively.

  Since $S^v$ is homeomorphic to either $\Sigma^v$ or the orientable double cover of $\Sigma^v$, the covering degrees of $\hat{c}^{\text{ini}}\to c^{\text{ini}}$ and $\hat{c}^{\text{ter}}\to c^{\text{ter}}$ are both $1$.  Since $\hat{h}$ and $\hat{c}^{\text{ini}}$ has intersection number $1$ in $\hat{T}^{\text{ini}}$, we have $[\hat{T}^{\text{ini}}:T^{\text{ini}}]=\langle q(\hat{c}^{\text{ini}}),q(\hat{h})\rangle=\langle c^{\text{ini}},q(\hat{h})\rangle$. Similarly, we have $[\hat{T}^{\text{ter}}:T^{\text{ter}}]=\langle q(\hat{c}^{\text{ter}}),q(\hat{h})\rangle=\langle c^{\text{ter}},q(\hat{h})\rangle$. Since the covering degree from the regular fiber of $J^v$ to the regular fiber of $M^v$ does not depend on which boundary component we choose, we have $$\frac{[\hat{T}^{\text{ini}}:T^{\text{ini}}]}{[\hat{T}^{\text{ter}}:T^{\text{ter}}]}=\frac{\langle c^{\text{ini}},q(\hat{h})\rangle}{\langle c^{\text{ter}},q(\hat{h})\rangle}=\frac{\langle c^{\text{ini}},h\rangle}{\langle c^{\text{ter}},h\rangle}.$$
\end{proof}

\subsection{The generalized spirality character}\label{spiralitycharacterdef}

In this section, we define the generalized spirality character $s:H_1(\Phi(H);\mathbb{Z})\to \mathbb{Q}_+^{\times}$. The spirality character in \cite{Liu} is original defined by partial dilations and a principal $\mathbb{Q}^{\times}$-bundle over $\Phi(S)$, then Liu showed that it can be computed by a combinatorial formula (formula 4.5 in \cite{Liu}). In this paper, we directly define the generalized spirality character by a combinatorial formula, which is closer to the flavor of \cite{RW}. Moreover, the definition in \cite{Liu} also takes care of orientations, so the image of Liu's spirality character may contain negative rational numbers. For simplicity, we will forget the sign and simply define it to be a homomorphism to $\mathbb{Q}_+^{\times}$.

We first define the generalized spirality character when $\Phi(H)$ is connected, then the definition for general case directly follows. When $\Phi(H)$ has no edge space (consists of one single vertex space), we define $s:H_1(\Phi(H);\mathbb{Z})\to \mathbb{Q}_+^{\times}$ to be the trivial homomorphism, i.e. $s(\alpha)=1$ for all $\alpha\in H_1(\Phi(H);\mathbb{Z})$.

In the case that $\Phi(H)$ has nonempty edge spaces, we take a base point $*\in \Phi(H)$ in an edge space, and first define a homomorphism $\hat{s}:\pi_1(\Phi(H),*)\to \mathbb{Q}_+^{\times}$.

\begin{definition}\label{spiralitydef}
Let $c:I\to \Phi(H)$ be an oriented closed path based at $*$ and transverse with edge spaces of $\Phi(H)$. Then $c$ is a concatenation of oriented paths $\delta_1,\cdots,\delta_n$ such that each $\delta_i$ is an oriented proper path in a vertex space $\Sigma_i\subset \Phi(H)$. Then we define $\hat{s}:\pi_1(\Phi(H),*)\to \mathbb{Q}_+^{\times}$ by
$$\hat{s}([c])=\prod_{i=1}^ns_{\delta_i}.$$
Then $\hat{s}$ induces a homomorphism $$s:H_1(\Phi(H);\mathbb{Z})\to \mathbb{Q}_+^{\times},$$ which is the {\it generalized spirality character}.
\end{definition}

It is easy to check that $\hat{s}$ is well-defined. Since all edge spaces of $\Phi(H)$ are essential simple closed curves, a homotopy between two oriented closed paths $c$ and $c'$ (up to reparameterization) consists of a finite sequence of following moves.
\begin{enumerate}
  \item Homotopy of paths with respect to their intersections with edge spaces of $\Phi(H)$, which clearly does not affect $\hat{s}([c])$.
  \item Do finger move of $c$ across one edge space of $\Phi(H)$. Those new terms in the formula created by the finger move cancel with each other, so $\hat{s}([c])$ does not change.
\end{enumerate}
From the definition, $\hat{s}:\pi_1(\Phi(H),*)\to \mathbb{Q}_+^{\times}$ is clearly a group homomorphism. Since $\mathbb{Q}_+^{\times}$ is abelian, $\hat{s}:\pi_1(\Phi(H),*)\to \mathbb{Q}_+^{\times}$ induces a homomorphism $s:H_1(\Phi(H);\mathbb{Z})\to \mathbb{Q}_+^{\times}$.

\begin{remark}\label{factorthrough}
Let $G_{\Phi(H)}$ be the dual graph of $\Phi(H)$. Then $s$ actually factors through a homomorphism $H_1(G_{\Phi(H)};\mathbb{Z})\to \mathbb{Q}_+^{\times}$, since only the initial and terminal points of $\{\delta_i\}$ affect the definition of $\hat{s}([c])$, and the topology of these paths in vertex spaces of $\Phi(H)$ do not matter. However, we prefer to use $H_1(\Phi(H);\mathbb{Z})$ in the definition of generalized spirality character, since it is more convenient for our proof, and this notation is compatible with the spirality character in \cite{Liu}.
\end{remark}

\subsection{Generalized spirality character under finite cover}\label{coverspirality}

For a finitely generated infinite index subgroup $H<\pi_1(M)$, we would like to study behavior of the generalized spirality character when taking finite covers of $M$ and finite index subgroups of $H$.

At first, we study the case that $M'\to M$ is a finite cover such that $H<\pi_1(M')$.

\begin{lemma}\label{manifoldcoverspirality}
  Let $M'\to M$ be a finite cover such that $H<\pi_1(M')$, then the almost fibered surfaces of $H$ associated to $M$ and $M'$ are homeomorphic to each other. Moreover, for the generalized spirality characters $s:H_1(\Phi(H);\mathbb{Z})\to \mathbb{Q}^{\times}$ and $s':H_1(\Phi(H);\mathbb{Z})\to \mathbb{Q}_+^{\times}$ defined by $M$ and $M'$ respectively, $s=s'$ holds.
\end{lemma}

\begin{proof}

The covering space of $M$ associated to $H<\pi_1(M)$ and the covering space of $M'$ associated to $H<\pi_1(M')$ are homeomorphic to each other, and we denote it by $M_H$. Then $M_H$ has two graph of space structures, which are induced by $M$ and $M'$ respectively. These two graph of space structures on $M_H$ are almost identical with each other, except that some trivial $I$-bundle vertex spaces are pinched to edge spaces. This pinching process only changes the graph of space structure on $\Phi(H)$, but does not change the topology of $\Phi(H)$.

For the graph of space structure on $\Phi(H)$ induced from $M$, the pinched pieces of $\Phi(H)$ live in those pieces of $M$ that are homeomorphic to the twisted $I$-bundle over Klein bottle, so $s_{\delta}=1$ holds for these pieces. Then we can ignore these terms in the definition of $\hat{s}$, and we can ignore the difference between these two graph of space structures on $\Phi(H)$.

For each oriented proper path $\delta$ in a piece $\Sigma^v\subset \Phi(H)$, $s'_{\delta}$ defined by $M'$ and $s_{\delta}$ defined by $M$ are equal to each other. In the Seifert case, it follows from the fact that the covering degree on regular fibers do not depend on the boundary component. In the fibered case, suppose $(S^v,\phi^v)$ is a pair that defines $s_{\delta}$ as in Case II, we can take a big enough power $(\phi^v)^n$ of $\phi^v$ such that the mapping torus $S^v\times I/(\phi^v)^n$ covers the corresponding vertex spaces in $M$ and $M'$. Then condition (2) in Case II of Section \ref{arcs} implies that the covering map $S^v\times I/(\phi^v)^n\to S^v\times I/\phi^v$ has the same covering degree on all boundary components. So $s_{\delta}=s_{\delta}'$ holds.

Then it follows that the two generalized spirality characters are equal to each other.
\end{proof}

Let $\bar{H}<H$ be a finite index subgroup, the following lemma compares the two generalized spirality characters $\bar{s}:H_1(\Phi(\bar{H});\mathbb{Z})\to \mathbb{Q}_+^{\times}$ and $s:H_1(\Phi(H);\mathbb{Z})\to \mathbb{Q}_+^{\times}$.

\begin{lemma}\label{subgroupspirality}
  Let $\bar{H}<H$ be a finite index subgroup, then the following commutative diagram hold. Here the horizontal map is induced by a finite covering map $\Phi(\bar{H})\to \Phi(H)$.
  \begin{diagram}
H_1(\Phi(\bar{H});\mathbb{Z}) &\rTo & H_1(\Phi(H);\mathbb{Z})\\
&\rdTo_{\bar{s}} &\dTo_s\\
& &\mathbb{Q}_+^{\times}
\end{diagram}
\end{lemma}

\begin{proof}
  Since $\bar{H}<H$ is a finite index subgroup, $M_{\bar{H}}$ is a finite cover of $M_H$. The preimage of virtually fibered and partially fibered pieces of $M_H$ are virtually fibered and partially fibered pieces of $M_{\bar{H}}$ respectively, and so do the virtually fibered and partially fibered pieces in $\Phi(H)$. So the almost fibered surface $\Phi(\bar{H})$ is a finite cover of $\Phi(H)$.

Let $p:\Phi(\bar{H})\to \Phi(H)$ be the covering map induced by $M_{\bar{H}}\to M_H$. For a virtually fibered or partially fibered subsurface $\bar{\Sigma}\subset \Phi(\bar{H})$ and an oriented proper path $\bar{\delta}:I\to \bar{\Sigma}$, let $\Sigma \subset \Phi(H)$ and $\delta=p\circ \bar{\delta}$ be their projections in $\Phi(H)$. Then the relation between $\bar{s}_{\bar{\delta}}$ and $s_{\delta}$ is as the following. Let $\bar{c}^{\text{ini}}$, $\bar{c}^{\text{ter}}$ be the boundary components of $\bar{\Sigma}$ containing $\bar{\delta}(0)$ and $\bar{\delta}(1)$ respectively, and let $c^{\text{ini}}$, $c^{\text{ter}}$ be the boundary components of $\Sigma$ containing $\delta(0)$ and $\delta(1)$ respectively. Then we have $$\bar{s}_{\bar{\delta}}=\frac{[\bar{c}^{\text{ini}}:c^{\text{ini}}]}{[\bar{c}^{\text{ter}}:c^{\text{ter}}]}\cdot s_{\delta}.$$ For the Seifert case, it follows directly from the definition. For the fibered case, it follows from a similar argument in the previous lemma.

Then the diagram commutes by the definition of generalized spirality character and the above equality. For an oriented closed path $c$ in $\Phi(\bar{H})$, when we compute $\bar{s}([c])$, the $\bar{c}^{\text{ter}}$ of $\bar{\delta}_i$ equals the $\bar{c}^{\text{ini}}$ of $\bar{\delta}_{i+1}$. So when we multiply $\bar{s}_{\bar{\delta}_1},\cdots,\bar{s}_{\bar{\delta}_n}$ together, the difference between $\bar{s}_{\bar{\delta}}$ and $s_{\delta}$ cancel with each other.

\end{proof}

\begin{remark}\label{samespirality}
Since $\Phi(\bar{H})$ is a finite cover of $\Phi(H)$, the image of $H_1(\Phi(\bar{H});\mathbb{Z})$ is a finite index subgroup of $H_1(\Phi(H);\mathbb{Z})$. Since $\mathbb{Q}^{\times}_{+}$ is torsion free, the generalized spirality character $s$ is trivial if and only if $\bar{s}$ is trivial.
\end{remark}

\section{Separability implies aspirality}\label{elegant}

In this section, we prove that if $H<\pi_1(M)$ is a separable subgroup, then the generalized spirality character of $H$ is trivial, i.e. the "only if" part of Theorem \ref{main}.

\subsection{The first reduction}

In this subsection, by using the separability of $H$, we do some simple reduction and find an intermediate finite cover $M'\to M$ of $M_H\to H$ such that it satisfies several normalization properties.

We first take a Scott core $C_H\subset M_H$ (\cite{Sco1}), then $C_H$ is a compact codim-$0$ submanifold of $M_H$, and it can be constructed by pasting Scott cores of (finitely many) pieces of $M_H^c$. The Scott core $C_H$ is path-connected and we can assume that $\Phi(H)\subset C_H$.

\begin{lemma}\label{embedding}
Suppose that $H<\pi_1(M)$ is separable, then there exists an intermediate finite cover $M'\to M$ of $M_H\to M$, such that $p|_{C_H}:C_H\to M'$ is an embedding. In particular, $p|_{\Phi(H)}:\Phi(H)\to M'$ is an embedding. Here $p$ is the covering map from $M_H$ to $M'$.
\end{lemma}
\begin{proof}
   This follows directly from Scott's lemma (Lemma \ref{Scott}) and separability of $H$, since $C_H\subset M_H$ is a compact subset.
\end{proof}

\begin{remark}\label{passtocover}
By Lemma \ref{basicLERF} (2), $H$ is separable in $\pi_1(M)$ implies that $H$ is separable in $\pi_1(M')$. By Lemma \ref{manifoldcoverspirality}, we need only to prove that $H$ is asprial in $M'$. So we abuse notation and still use $M$ to denote the finite cover given by Lemma \ref{embedding}, then we can assume that $\Phi(H)$ and $C_H$ are embedded in $M$.
\end{remark}

Currently, some vertex pieces of $\Phi(H)$ might be nonorientable surfaces, and some orientable pieces of $\Phi(H)$ may lie in a semi-bundle or lie in a Seifert manifold with nonorientable base orbifold. In the following lemma, we get rid of these situations.

\begin{lemma}\label{normalize}
  Suppose that $\Phi(H)$ is embedded in $M$, then there is a finite cover $M'$ of $M$ such that the following conditions hold for $H'=H\cap \pi_1(M')$.
  \begin{enumerate}
    \item For each Seifert piece of $M'$, its base space is orientable and has no singular points.
    \item Each vertex piece of $\Phi(H')$ is an orientable surface.
    \item For each virtually fibered subsurface $\Sigma\subset \Phi(H')$ that lies in a piece $V\subset M'$, $\Sigma$ is a fibered subsurface of $V$.
  \end{enumerate}
\end{lemma}

\begin{proof}
  Note that if one of the conditions in this lemma hold for some $M'$, then it holds for all further finite covers. So we can work on these three conditions separately.

  At first, Lemma 3.1 of \cite{PW} implies the existence of a finite cover $M'$ of $M$ satisfying condition (1). For this $M'$, any vertex piece $\Sigma\subset \Phi(H')$ that lies in a Seifert piece of $M'$ automatically satisfies condition (2). So we only need to work on virtually fibered pieces $\Sigma\subset \Phi(H')$.

  Let $\Sigma\subset \Phi(H')$ be an embedded virtually fibered subsurface in a piece $V\subset M'$. Then for a regular neighborhood $N(\Sigma)$ of $\Sigma$, $V\setminus N(\Sigma)$ has a finite cover that is homeomorphic to a (possibly disconnected) surface cross interval. Then Theorem 10.6 of \cite{Hem1} implies that $V\setminus N(\Sigma)$ is an $I$-bundle over a surface.

  If $\Sigma$ is nonorientable, let $S$ be the orientable double cover of $\Sigma$. Then $V$ is a union of two twisted $I$-bundles over $\Sigma$. Let $\sigma$ be the $\mathbb{Z}_2$-cohomology class dual with the union of cores of these two twisted $I$-bundles, and let $V''$ be the double cover of $V$ corresponding to $\sigma$. Then $V''$ is an $S$-bundle over $S^1$, and each boundary component of $V''$ is mapped to a boundary component of $V$ by homeomorphism. If $\Sigma$ is an orientable surface, then $V$ is either a $\Sigma$-bundle over $S^1$, or a semi-bundle as above and $\Sigma$ is the common boundary of two twisted $I$-bundles. In the second case, we still take the double cover $V''$ of $V$ as the previous case.

  Now for each piece $V\subset M'$, if we take the double cover $V''$ as above, we take one copy of $V''$, otherwise we take two copies of $V$. Since the restriction of the double cover on each boundary component of $V''$ is a homeomorphism, we can paste these $V''$ and $V$ together to get a double cover $M''$ of $M'$. Then it is easy to check that $M''$ and $H''=H'\cap \pi_1(M'')$ satisfy all desired conditions.

\end{proof}

\begin{remark}\label{noannulus}
  If $M$ and $H<\pi_1(M)$ satisfy the conditions in the above lemma, then no piece of $M$ is the twisted $I$-bundle over Klein bottle, and no vertex piece of $\Phi(H)$ is the annulus. Then for any intermediate finite cover $M'\to M$ of $M_H\to M$, the graph of space structure on $\Phi(H)$ induced by $M'$ is same with the graph of space structure on $\Phi(H)$ induced by $M$.
\end{remark}

Similar to Remark \ref{passtocover}, the separable subgroup $H'<\pi_1(M')$ constructed in Lemma \ref{normalize} is separable and we only need to prove that $H'$ is aspiral in $\pi_1(M')$. We can still assume the almost fibered surface $\Phi(H')$ and a Scott core $C_{H'}\subset M_{H'}$ are embedded into $M'$. For simplicity, we still denote this manifold by $M$ and denote this subgroup by $H$.

The following lemma is the main reduction in this subsection.

\begin{lemma}\label{normalizer}
Suppose that $H<\pi_1(M)$ is separable and satisfies Lemma \ref{embedding} and \ref{normalize}, then there is an intermediate finite cover $M'\to M$ of $M_H\to M$, such that the following holds.

For any piece $\Sigma \subset \Phi(H)$ and any boundary component $c\subset \partial \Sigma$, let $M_{\Sigma}'\subset M'$ be the vertex piece of $M'$ containing $\Sigma$, and $T$ be the boundary component of $M_{\Sigma}'$ containing $c$. Then for any base point $*\in c$, $\pi_1(T,*)$ lies in the normalizer of $\pi_1(\Sigma,*)$.
\end{lemma}

\begin{proof}
  Note that once the normalizer condition holds for one piece $\Sigma \subset \Phi(H)$ and one boundary component $c\subset \partial \Sigma$ in some $M'$, then it holds for all further intermediate finite covers of $M_H\to M'$. So it suffices to prove the lemma for one piece $\Sigma\subset \Phi(H)$ and one of its boundary component $c$.

  If $\Sigma$ is a virtually fibered subsurface, then Lemma \ref{normalize} (3) implies that $\Sigma$ is a fibered subsurface in the corresponding piece $M_{\Sigma}\subset M$. Then $\pi_1(\Sigma,*)$ is a normal subgroup of $\pi_1(M_{\Sigma},*)$, so $\pi_1(T,*)$ lies in the normalizer of $\pi_1(\Sigma,*)$.

  Now it suffices to consider the case that $\Sigma$ is a partially fibered subsurface in a Seifert piece $M_{\Sigma}\subset M$. Since the base orbifold of $M_{\Sigma}$ is orientable (Lemma \ref{normalize} (1)), the regular fiber $h$ of $M_{\Sigma}$ lies in the center of $\pi_1(M_{\Sigma},*)$, so it lies in the normalizer of $\pi_1(\Sigma,*)$. Since both $h$ and $c$ lie in the normalizer $N(\pi_1(\Sigma,*))$ of $\pi_1(\Sigma,*)$ and they are two nonparallel slopes on $T$, $N_T=N(\pi_1(\Sigma,*))\cap \pi_1(T,*)$ is a finite index subgroup of $\pi_1(T,*)$.

  We take the left coset decomposition $\pi_1(T,*)=N_T\cup (\cup_{i=1}^kg_iN_T)$. Then for each $i$, there exists $\gamma_i\in \pi_1(\Sigma,*)$ such that $g_i\gamma_ig_i^{-1}\notin \pi_1(\Sigma,*)$ (or $g_i^{-1}\gamma_ig_i\notin \pi_1(\Sigma,*)$). Since $g_i\gamma_ig_i^{-1}\in \pi_1(M_{\Sigma},*)$ and $\pi_1(M_{\Sigma},*)\cap H=\pi_1(\Sigma,*)$, we have $g_i\gamma_ig_i^{-1}\notin H$. By separability of $H$, there is an intermediate finite cover $M'\to M$ of $M_H\to M$ such that $g_i\gamma_ig_i^{-1}\notin \pi_1(M')$ for all $i=1,2,\cdots,k$.

  Then for the elevation $T'$ of $T$ in $M'$ that contains $c$, we must have $\pi_1(T')<N_T$, thus $\pi_1(T')$ lies in the normalizer of $\pi_1(\Sigma)$. Otherwise, there exist $g_in\in \pi_1(T')$ for some $i\in \{1,2,\cdots,k\}$ and $n\in N_T$. Since $n^{-1}\gamma_in\in \pi_1(\Sigma)<H$, we have $$g_i\gamma_ig_i^{-1}=(g_in)(n^{-1}\gamma_in)(g_in)^{-1}\in \pi_1(T')\pi_1(\Sigma)\pi_1(T')\subset \pi_1(M').$$ It contradicts with the assumption that $g_i\gamma_ig_i^{-1}\notin \pi_1(M')$, so the proof is done.
\end{proof}

\subsection{Properly embedding $\Phi(H)$ into a virtual submanifold}\label{reduction}

In this subsection, we always assume that $M$ and a separable subgroup $H<\pi_1(M)$ satisfy Lemma \ref{embedding}, \ref{normalize} and \ref{normalizer}. Note that any further intermediate finite cover $M'\to M$ of $M_H\to M$ still satisfies these lemmas.

The following proposition is the main result in this subsection. This result enables us to lift the embedded almost fibered surface $\Phi(H)\subset M$ to some finite cover $M'$, such that $\Phi(H)$ lies in a codim-$0$ submanifold $\hat{M}\subset M'$ as a properly embedded subsurface. Then the properness implies aspirality as in \cite{Liu}.

\begin{proposition}\label{cut}
  If $H<\pi_1(M)$ is separable, then there exists an intermediate finite cover $M'\to M$ of $M_H\to M$ and a (possibly disconnected) codimension-$0$ submanifold $\hat{M}\subset M'$, such that the following hold.
  \begin{enumerate}
    \item For each component $T$ of $\partial \hat{M}\setminus \partial M'$, $T$ is either a decomposition torus of $M'$ or a vertical torus in a Seifert piece of $M'$.
    \item $\Phi(H)$ is contained in $\hat{M}$ as a properly embedded subsurface.
  \end{enumerate}
\end{proposition}

To prove Proposition \ref{cut}, we need a few steps to find the intermediate finite cover $M'\to M$ of $M_H\to M$, such that $\Phi(H)$ lives in $M'$ in better and better position.

We first define a set $$\mathcal{I}=\{M'\ |\ M'\ \text{is\ an\ intermediate\ finite\ cover\ of\ } M_H\to M\}.$$ Then for any $M'\in \mathcal{I}$, $\Phi(H)\subset M_H$ projects into $M'$ by an embedding. In the following lemmas, the key point is to take the minimizer of some complexity in $\mathcal{I}$.

We first prove that there exists some $M'\in \mathcal{I}$, such that for any two decomposition circles or boundary components of $\Phi(H)$ that are mapped to the same decomposition torus or boundary torus of $M'$, they should have "same" adjacent pieces. To make the statement of following lemma simpler, we count boundary components of $M'$ as its decomposition tori, and count boundary components of $\Phi(H)$ as its decomposition circles.

\begin{lemma}\label{parallel}
   There exists an intermediate finite cover $M'\to M$ of $M_H\to M$, such that the following hold.
\begin{enumerate}
\item For any two decomposition circles $c,c'\subset \Phi(H)$, if they are mapped to the same decomposition torus $T'\subset \partial M'$, then either both $c$ and $c'$ lie in $\partial (\Phi(H))$ or both of them lie in $\text{int}(\Phi(H))$.
\item  Moreover, for $c$ and $c'$ as above, let $\Sigma,\Sigma'$ be two vertex spaces of $\Phi(H)$ adjacent to $c,c'$ respectively, such that they lie in the same vertex piece $V'\subset M'\setminus T'$ (as in Figure 1). Then $\Sigma$ and $\Sigma'$ are parallel with each other in $V'$. More precisely, after choosing a path $\gamma$ in $T'$ from $*\in c$ to $*'\in c'$, we have $\pi_1(\Sigma,*)=\gamma\pi_1(\Sigma',*')\bar{\gamma}<\pi_1(V',*)$.
\end{enumerate}
\end{lemma}

\begin{proof}

Let $E(\Phi(H))$ be the set of all decomposition circles of $\Phi(H)$. For each $M'\in \mathcal{I}$, let $E_{\Phi(H)}(M')$ be the set of all decomposition tori of $M'$ that contain some $c\in E(\Phi(H))$. Then we define the complexity of $M'$ by $$C(M')=|E(\Phi(H))|-|E_{\Phi(H)}(M')|.$$
Here $|\cdot|$ denotes the number of elements in a (finite) set. $C(M')$ measures how far is $E(\Phi(H))\to E_{\Phi(H)}(M')$ from being injective. Since the complexity $C(M')$ is a natural number, the minimizer of this complexity exits. When we take a finite cover $M''\to M'$ with $M',M''\in \mathcal{I}$, the complexity does not increase. We will prove that a minimizer $M'\in \mathcal{I}$ satisfies desired properties.

We take a base point $*\in c$ and denote the covering space corresponding to $H<\pi_1(M')$ by $(M_H,*)\to (M',*)$.

\bigskip

We first prove (1) by proof by contradiction. Without loss of generality, we suppose that $c$ lies in $\text{int}(\Phi(H))$ and $c'$ lies in $\partial \Phi(H)$. Then there is a unique piece $\Sigma'\subset\Phi(H)$ containing $c'$, and we denote the piece of $M'$ that contains $\Sigma'$ by $V'$. Since $c\subset \text{int}(\Phi(H))$, there is another piece $V\subset M'$ that is adjacent to $V'$ along $T'$. For the two pieces of $\Phi(H)$ adjacent to $c$, we denote the one lying in $V'$ by $\Sigma$, and denote the one lying in $V$ by $\Xi$. All these constructions are summarized by Figure 1.

\begin{center}
\includegraphics[width=2in]{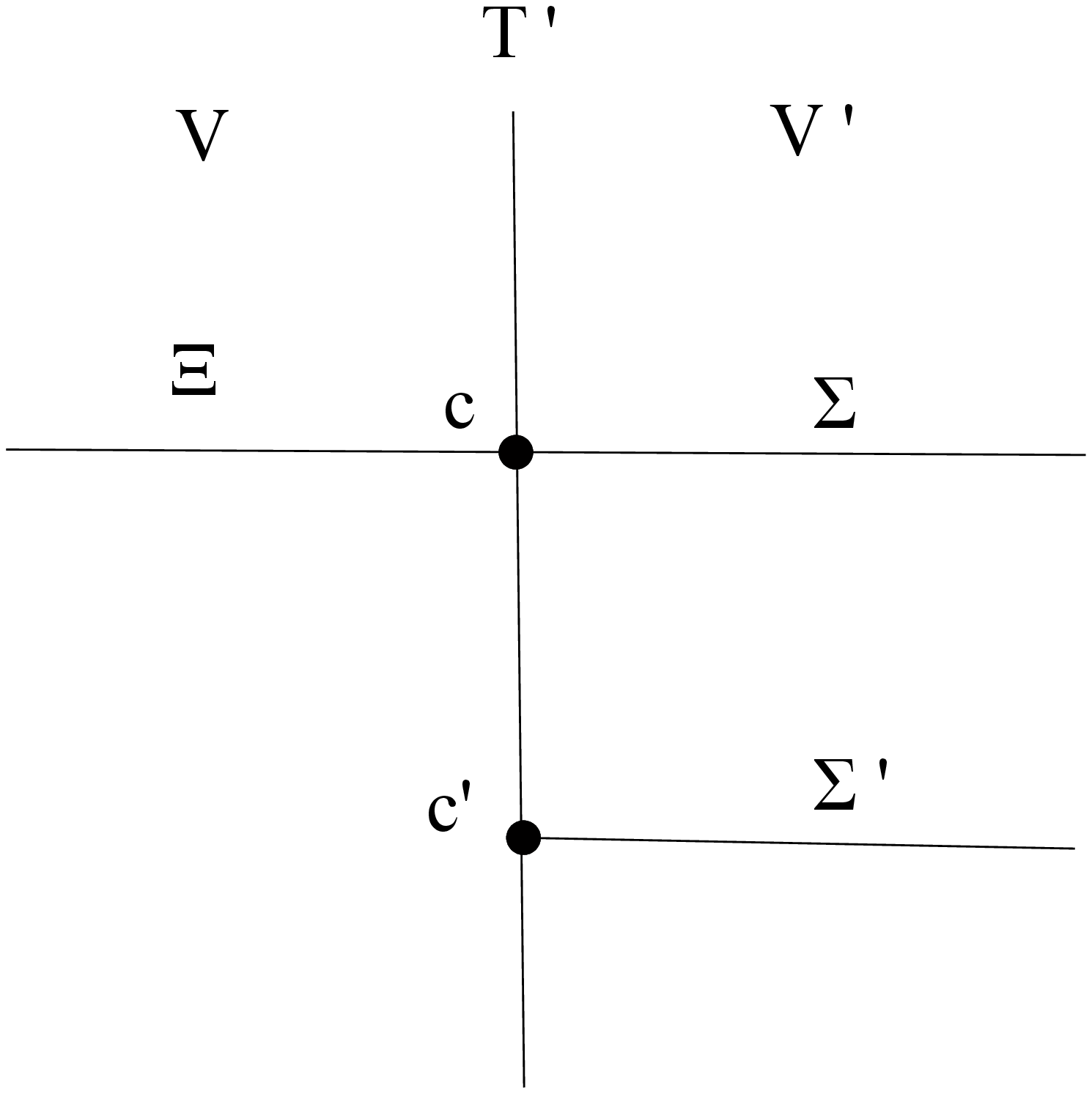}
 \centerline{Figure 1}
\end{center}

We take an auxiliary point $*'\in c'$, a path $\alpha$ in the Scott core $C_H$ from $*$ to $*'$, and a path $\gamma$ in $T'$ from $*$ to $*'$. Then $\alpha\bar{\gamma}$ gives an element $a\in \pi_1(M',*)$.

 We denote the vertex space of $M_H$ that contains $c'$ but does not contain $\Sigma'$ by $V_H$. Then $(V_H,*')$ is a covering space of $(V,*')$ that does not intersect with $\text{int}(\Phi(H))$. $V$ can not be a Seifert manifold, since $\pi_1(V_H)$ is nontrivial and it does not contain any power of the regular fiber of $V$. So $V$ must be hyperbolic and $V_H$ corresponds to a geometrically finite subgroup of $\pi_1(V)$. For $\pi_1(\Xi)<\pi_1(V)$, it is a fibered subgroup of $\pi_1(V)$, so any subgroup of $\pi_1(V)$ containing $\pi_1(\Xi)$ as a proper subgroup is a finite index subgroup of $\pi_1(V)$. So there exists an element $g\in \pi_1(\Xi)-\pi_1(V_H)$. More precisely, there exists $g\in \pi_1(V,*)$ such that $g\in \pi_1(\Xi,*)$ but $g\notin \gamma \pi_1(V_H,*')\bar{\gamma}$.

  \bigskip

  \textbf{Claim 1:} We must have $aga^{-1}\notin H$.

  Otherwise, $aga^{-1}=\alpha\bar{\gamma}g\gamma\bar{\alpha}$ lifts to a loop in $M_H$ base at $*$, and it implies that $\bar{\gamma} g\gamma$ lifts to a loop in $M_H$ based at $*'$. Since both $g$ and $\gamma$ are paths in $V$, $\bar{\gamma} g\gamma$ lifts to a loop in $V_H\subset M_H$ based at $*'$. So we have $\bar{\gamma}g\gamma\in \pi_1(V_H,*')$, which contradicts with $g\notin \gamma \pi_1(V_H,*')\bar{\gamma}$. This implies Claim 1 holds.

  \bigskip

 By the separability of $H<\pi_1(M')$, there is an intermediate finite cover $(M'',*)\to (M',*)$ of $(M_H,*)\to (M',*)$ such that $aga^{-1}\notin \pi_1(M'')$.
  \bigskip

  \textbf{Claim 2:} For $M''$ as above, $*\in c$ and $*'\in c'$ are not mapped to the same decomposition torus of $M''$.

  Otherwise, suppose both $*$ and $*'$ are mapped to the same decomposition torus $T''\subset M''$, then there is a path $\delta$ in $T''$ from $*$ to $*'$. Recall that $\alpha$ is a path in the Scott core $C_H$ from $*$ to $*'$. Then $\alpha\bar{\delta}$ gives an element $b\in \pi_1(M'',*)$ and $\delta\bar{\gamma}$ gives an element $c\in \pi_1(T',*)$. So we have $$aga^{-1}=(\alpha\bar{\gamma})g(\gamma\bar{\alpha})=(\alpha\bar{\delta})\big((\delta\bar{\gamma})g(\delta\bar{\gamma})^{-1}\big)(\alpha\bar{\delta})^{-1}=b(cgc^{-1})b^{-1}.$$
  By Lemma \ref{normalizer}, $\pi_1(T')$ lies in the normalizer of $\pi_1(\Xi)$. Since $g\in \pi_1(\Xi)$, $cgc^{-1}\in \pi_1(\Xi)<H<\pi_1(M'')$ holds. Since $b\in \pi_1(M'')$, we have $aga^{-1}=b(cgc^{-1})b^{-1}\in \pi_1(M'')$, which contradicts with our assumption that $aga^{-1}\notin \pi_1(M'')$. So Claim 2 holds.

\bigskip

  So $*$ and $*'$ are mapped to different decomposition tori of $M''$, thus $c$ and $c'$ are mapped to different decomposition tori in $M''$. This implies that $C(M'')<C(M')$, which contradicts with the assumption that $M'$ is a minimizer of $C(\cdot)$. So $c$ and $c'$ must simultaneously lie in $\partial \Phi(H)$ or lie in $\text{int}(\Phi(H))$.

  \bigskip

  For the moreover part, the proof is almost the same. If $\pi_1(\Sigma,*)\ne \gamma \pi_1(\Sigma',*')\bar{\gamma}$, there exists $g$ such that either $g\in \pi_1(\Sigma,*)\setminus \gamma \pi_1(\Sigma',*')\bar{\gamma}$ or $g\in \gamma \pi_1(\Sigma',*')\bar{\gamma}\setminus \pi_1(\Sigma,*)$. Then a similar proof as above (invoking Claim 1 and Claim 2) gives us a contradiction.

 Since Lemma \ref{normalizer} implies that $\pi_1(T',*)$ lies in the normalizer of $\pi_1(\Sigma,*)$, the choice of path $\gamma$ from $*$ to $*'$ does not matter. Then $\pi_1(\Sigma,*)=\gamma \pi_1(\Sigma',*')\bar{\gamma}$ implies that the covering spaces of $V'$ corresponding to $\pi_1(\Sigma)$ and $\pi_1(\Sigma')$ are same with each other.
\end{proof}

Note that the condition of this lemma still holds when passing to a further finite cover $M''\to M$ of $M'\to M$ with $M''\in \mathcal{I}$.

\begin{remark}\label{sameboundary}
     By the construction of $\Phi(H)$, each piece of $\Phi(H)$ is either a fibered surface, or the convex core with respect to a fixed hyperbolic structure on the base orbifold, or an annulus neighborhood of a boundary component. So $\pi_1(\Sigma,*)=\gamma \pi_1(\Sigma',*')\bar{\gamma}$ implies that $\Sigma$ and $\Sigma'$ intersect with the same set of boundary components of $V'$, and their projections to the base orbifold of $V'$ have the same image when $V'$ is a Seifert manifold.
\end{remark}

\begin{remark}\label{alwayssametori}
  By the proof of Lemma \ref{parallel}, in the remaining part of this section, we will always assume that $M$ is a minimizer of $C(\cdot)$ in $\mathcal{I}$. So for any intermediate finite cover $M'\to M$ of $M_H\to M$, if $c$ and $c'$ are two decomposition circles in $\Phi(H)$ that are mapped to the same decomposition torus in $M$, then they are mapped to the same decomposition torus in $M'$.
\end{remark}

The following lemma is the crucial one to guarantee the properness of $\Phi(H)\subset \hat{M}$ in Proposition \ref{cut}.

\begin{lemma}\label{proper}
 Suppose that $H< \pi_1(M)$ and $\Phi(H)\to M$ satisfy Lemma \ref{parallel}. Then there exists an intermediate finite cover $M'\to M$ of $M_H\to M$, such that the following hold.

 For any partially fibered subsurface $\Sigma\subset \Phi(H)$ that is contained in a Seifert piece $V\subset M'$, let $r:V\to \mathcal{O}_V$ be the projection to its base orbifold. Then $\Sigma$ is contained in $\hat{V}=r^{-1}(r(\Sigma))\subset V$ as an properly embedded subsurface (i.e. $\Sigma\cap \partial \hat{V}=\partial \Sigma$).
\end{lemma}

\begin{proof}
The conclusion of this lemma is equivalent with that $r|_{\Sigma}:\Sigma\to \mathcal{O}_V$ is a covering map to the image of $\Sigma$. If the conclusion holds for one partially fibered subsurface $\Sigma\subset \Phi(H)$ in some $M'\in \mathcal{I}$, then it still holds for $\Sigma$ and any finite intermediate cover $M''\to M'$ of $M_H\to M'$. So we only need to work on one single partially fibered piece $\Sigma\subset \Phi(H)$.

Now we define a complexity on $\mathcal{I}$ with respect to $\Sigma$. Let $V$ be the Seifert piece of $M'$ containing $\Sigma$, $\mathcal{O}_V$ be the base orbifold of $V$, and $r|_{\Sigma}:\Sigma\to \mathcal{O}_V$ be the restriction of the projection map from $V$ to $\mathcal{O}_V$. Then we define a (discontinuous) function $\rho:\mathcal{O}_V\to \mathbb{N}$ by $\rho(x)=|r^{-1}(x)\cap\Sigma|$, and define the complexity of $M'$ by $$C_{\Sigma}(M')=\max{\rho}.$$

Recall that (Case II.2 of Section \ref{subgroups}) $\Sigma$ is the convex core of a covering space $\mathcal{O}_{V_L}\to \mathcal{O}_V$ with finitely generated fundamental group (or a closed $\epsilon$-neighborhood of a boundary component). Since $\Sigma$ is transverse with the Seifert fibration, the compactness of $\Sigma$ implies that $C_{\Sigma}(M')=\max{\rho}$ is a finite number. Moveover, if $M''\to M'$ is an intermediate finite cover of $M_H\to M'$, then $C_{\Sigma}(M'')\leq C_{\Sigma}(M')$ holds. Since $C_{\Sigma}(\cdot)$ is a natural number, the minimizer of $C_{\Sigma}(\cdot)$ exists.

Now we take a minimizer $M'$ of the complexity $C_{\Sigma}(\cdot)$ in $\mathcal{I}$. Then we will prove that for the minimizer $M'$, $\Sigma$ is a properly embedded subsurface of $\hat{V}=r^{-1}(r(\Sigma))\subset V$. We only need to prove it in the case that $\Sigma$ is obtained by the convex core construction, since this lemma clearly holds for the closed $\epsilon$-neighborhood case, if $\epsilon>0$ is small enough.

\bigskip

Since $r$ maps each boundary component of $\Sigma$ to a closed geodesic in $\mathcal{O}_V$, for two boundary components of $\Sigma$, either they are mapped to the same closed geodesic in $\mathcal{O}_V$, or their images are transverse with each other. Let $X$ be the set of all self-intersection points of $r(\partial \Sigma)$, then all of $\mathcal{O}_V \setminus r(\partial \Sigma)$, $r(\partial \Sigma) \setminus X$ and $X$ have finitely many components. For the function $\rho:\mathcal{O}_V\to \mathbb{N}$, it is constant on each of these components.

By the connectedness of $\Sigma$, if $\Sigma$ is an properly embedded subsurface of $\hat{V}=r^{-1}(r(\Sigma))$, then $\text{im}(\rho)=\{0,n\}$, where $n$ is the intersection number between $\Sigma$ and a regular fiber of $\hat{V}$. Actually, it is a sufficient and necessary condition. If $\Sigma$ is not proper in $\hat{V}$, then $r$ maps some $x\in \partial \Sigma$ into $\text{int}(r(\Sigma))$. Then there is a small open disc $U\subset r(\Sigma)$ containing $r(x)$, such that $(r|_{\Sigma})^{-1}(U)$ is a disjoint union of open subsets $\cup V_{\alpha}$ in $\Sigma$. Moreover, for each $\alpha$, $r|_{V_{\alpha}}$ is an embedding from $V_{\alpha}$ to $U$ whose image either equals $U$ or is a half disc centered at $r(x)$. Since the component $V_{\alpha_0}$ that contains $x$ is a half disc, $\rho$ is positive but not a constant function near $x$.

If $\Sigma$ is not proper in $\hat{V}$, then $\text{im}(\rho)$ consists of more than two integers, and let $L$ be the maximum of $\rho$. We take all components of $\mathcal{O}_V \setminus r(\partial \Sigma)$, $r(\partial \Sigma) \setminus X$ and $X$ such that the value of $\rho$ on them equals $L$, and denote the collection of their closures by $\{A'_1,\cdots,A_l'\}$. Then we throw away all those $A'_i$ such that $A_i'\subsetneqq A'_j$ for some $j$, and obtain a sub-collection $\{A_1,\cdots,A_k\}$. For any $A_i$, there exist $x_i\in A_i$ and $y_i,y_i'\in \Sigma$ such that $r(y_i)=r(y_i')=x_i$ and one of the following hold.
\begin{enumerate}
  \item Either $y_i\in \partial \Sigma$ and $y_i'\in \text{int}(\Sigma)$;
  \item or $y_i, y_i'\in \partial \Sigma$, and there are half discs $D_i,D_i'\subset \Sigma$ of equal radius, such that $D_i$ and $D_i'$ contain $y_i$ and $y_i'$ respectively, but $r(D_i)\ne r(D_i')$.
\end{enumerate}
Figure 2 shows these two possibilities. The left half corresponds to case (1) and the right half corresponds to case (2). The dotted points in the picture correspond to $x_i\in \mathcal{O}_V$. For the function $\rho$, its value on the lighter shaded region is smaller than its value on the darker shaded region.

  \begin{center}
\includegraphics[width=3in]{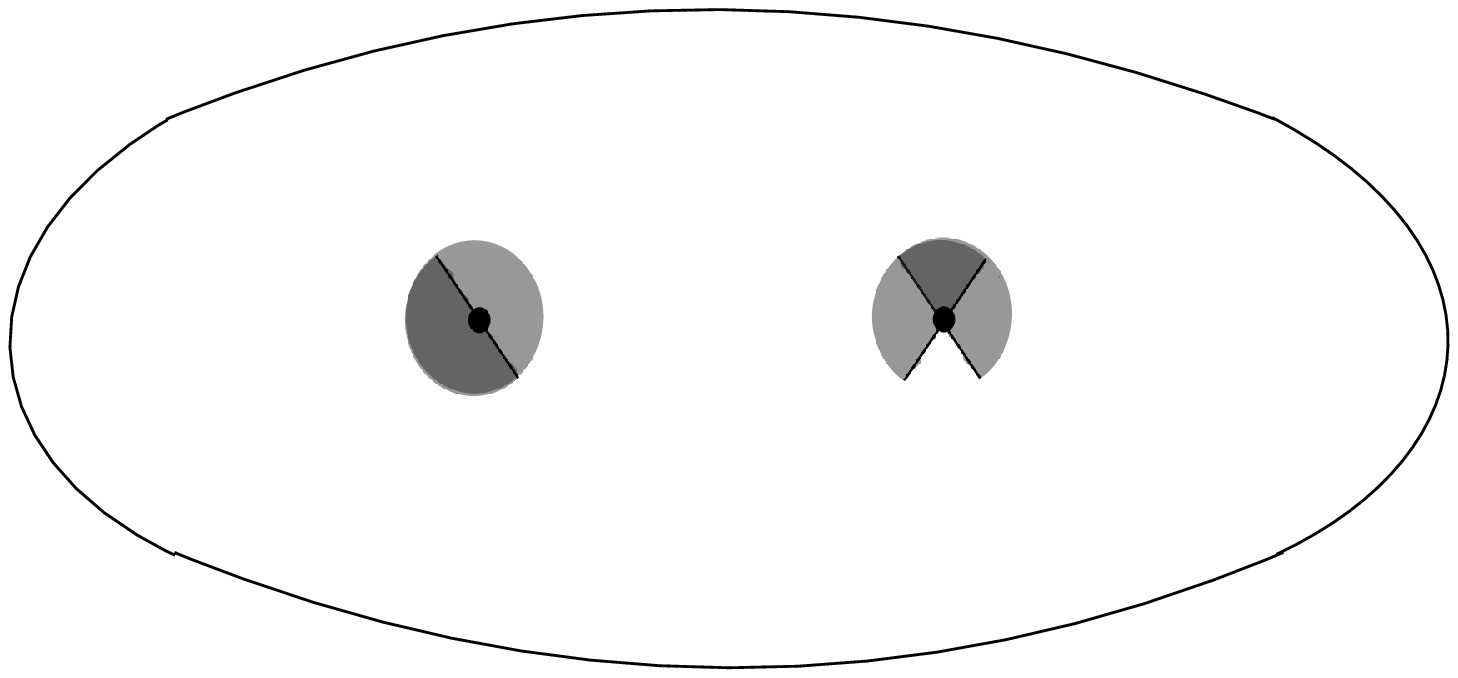}
 \centerline{Figure 2}
\end{center}

We take a base point $*\in \partial\Sigma\cap \partial V$. For each $i$, we take an oriented path $\alpha_i$ in $\Sigma$ from $*$ to $y_i$, an oriented path $\beta_i$ that lies in a regular fiber of $V$ and goes from $y_i$ to $y_i'$, and an oriented path $\gamma_i$ in $\Sigma$ from $y_i'$ to $*$. Let $f_i\in \pi_1(V,*)$ be the element given by closed path $\alpha_i\beta_i\gamma_i$. Then we first prove the following claim.

\bigskip

\textbf{Claim.} For each $i$, $f_i$ does not lie in the normalizer of $\pi_1(\Sigma,*)$ in $\pi_1(V,*)$.

Recall that $\Sigma$ is the convex core of a covering space $\mathcal{O}_{V_H}\to \mathcal{O}_V$. Let $N$ be the normalizer of $\pi_1(\Sigma,*)$ in $\pi_1(\mathcal{O}_V,*)$, then the deck transformation group of $\mathcal{O}_{V_H}\to \mathcal{O}_V$ is isomorphic to $N/\pi_1(\Sigma,*)$, which also acts on $\Sigma$ as a covering map action. Since $\Sigma$ is compact, $N/\pi_1(\Sigma,*)$ must be a finite group, so $N$ is finitely generated. Let $\Sigma_N$ be the convex core of the covering space $\mathcal{O}_N\to \mathcal{O}_V$ corresponding to $N<\pi_1(\mathcal{O}_N)$ (after truncating cusps), then $\Sigma_N$ is compact. Recall that Lemma \ref{normalize} (1) implies that $V$ is an $S^1$-bundle over surface $\mathcal{O}_V$. Let $V_N$ be the pull back bundle of $V\to \mathcal{O}_V$ via covering map $\mathcal{O}_N\to \mathcal{O}_V$, and we take a base point $*\in V_N$ that is mapped to $*\in \Sigma\subset V$.

The spaces we constructed above are summarized in the following diagram.

\begin{diagram}
& &V_H & \rTo & V_N & \rTo & V\\
& \ruInto&\dTo & &\dTo_{r_N} & &\dTo_{r}\\
 \Sigma &\subset &\mathcal{O}_{V_H} & \rTo & \mathcal{O}_N & \rTo & \mathcal{O}_V\\
& \rdTo(4,2)&  &  & \cup &  & \\
& & & & \Sigma_N &  & \\
\end{diagram}

Suppose that $f_i$ lies in the normalizer of $\pi_1(\Sigma,*)$ in $\pi_1(V,*)$, then $r_*(f_i)$ lies in the normalizer of $\pi_1(\Sigma,*)$ in $\pi_1(\mathcal{O}_V,*)$. Since $\pi_1(\Sigma)<\pi_1(\mathcal{O}_N)$ and $V_N$ is the pull-back bundle of $\mathcal{O}_N\to \mathcal{O}_V$, $\Sigma$ lifts to an embedded surface in $V_N$. Since $r_*(f_i)\in N$, $\alpha_i\beta_i\gamma_i$ lifts to a loop in $V_N$ based at $*$. We use $r_N:V_N\to \mathcal{O}_N$ to denote the projection from $V_N$ to its base orbifold. Since $\beta_i$ is a path in the regular fiber, while $y_i$ and $y_i'$ are the initial and terminal points of $\beta_i$ respectively, after lifting everything to $V_N$, we have $r_N(y_i)=r_N(y_i')\in \Sigma_N\subset \mathcal{O}_N$.

By considering $\Sigma$ as a subsurface of $V_N$, $r_N|_{\Sigma}:\Sigma\to \Sigma_N\subset \mathcal{O}_N$ is the finite cover induced by $\mathcal{O}_{V_H}\to \mathcal{O}_N$. By also considering $\Sigma$ as a subsurface of $V$, we have the following commutative diagram.

\begin{diagram}
\Sigma &\rTo^{r_N|_{\Sigma}} & \Sigma_N\\
&\rdTo_{r|_{\Sigma}} &\dTo\\
& &\mathcal{O}_V
\end{diagram}

By our assumption, there are open discs or half discs $D_i,D_i'\subset \Sigma$ of the same radius and centered at $y_i,y_i'$ respectively, such that $r(D_i)\ne r(D_i')$. However, since $r_N|_{\Sigma}:\Sigma\to \Sigma_N$ is a covering map and $r_N(y_i)=r_N(y_i')$ in $\Sigma_N$, $r_N$ must map $D_i$ and $D_i'$ to the same half disc in $\Sigma_N$, i.e. $r_N(D_i)=r_N(D_i')$. It contradicts with $r(D_i)\ne r(D_i')$, so $f_i$ does not lie in the normalizer of $\pi_1(\Sigma,*)$ in $\pi_1(V,*)$. The proof of this claim is done.

\bigskip

As subgroups of $\pi_1(V,*)$, we have $f_i\pi_1(\Sigma,*)f_i^{-1}\ne \pi_1(\Sigma,*)$, so there exists $g_i\in \pi_1(\Sigma,*)$ such that $f_ig_if_i^{-1}\notin \pi_1(\Sigma,*)$ (or replace $f_i$ by $f_i^{-1}$ if necessary). Then $H\cap \pi_1(V,*)=\pi_1(\Sigma,*)$ implies that $f_ig_if_i^{-1}\notin H$, and separability of $H$ guarantees the existence of an $M''\in \mathcal{I}$ such that $f_ig_if_i^{-1}\notin \pi_1(M'')$ for all $i$.

Let $V''$ be the Seifert piece of $M''$ that contains $\Sigma$, then $y_i$ and $y_i'$ do not lie in the same regular fiber of $V''$. Otherwise, we must have $$\alpha_ih^{n_i}\beta_i\gamma_i=(\alpha_ih^{n_i}\alpha_i^{-1})f_i\in \pi_1(V'',*),$$
here $h$ is the regular fiber of $V\subset M'$ based at $y_i$ and $n_i\in \mathbb{Z}$. Since $\alpha_ih\alpha_i^{-1}$ lies in the center of $\pi_1(V,*)$, $g_i\in \pi_1(\Sigma,*)\subset \pi_1(V'',*)$ and $(\alpha_ih^{n_i}\alpha_i^{-1})f_i\in \pi_1(V'',*)$ imply that $$f_ig_if_i^{-1}=(\alpha_ih^{n_i}\alpha_i^{-1})f_ig_if_i^{-1}(\alpha_ih^{n_i}\alpha_i^{-1})^{-1}=(\alpha_ih^{n_i}\alpha_i^{-1}f_i)g_i(\alpha_ih^{n_i}\alpha_i^{-1}f_i)^{-1}\in \pi_1(V'')<\pi_1(M''),$$ which contradicts with the choice of $M''$. So for $M''$, $y_i$ and $y_i'$ do not lie in the same regular fiber of $V''$.

Let $r'':V''\to \mathcal{O}''$ be the projection to the base orbifold of $V''$, then we have the following commutative diagram.

\begin{diagram}
\Sigma &\rInto & V'' & \rTo & V\\
&\rdTo &\dTo_{r''} & &\dTo_{r}\\
&      & \mathcal{O}'' & \rTo^{s} &\mathcal{O}_V\\
\end{diagram}

Suppose that $C_{\Sigma}(M'')=L$, then there exists $x''\in \mathcal{O}''$ such that $|(r'')^{-1}(x'')\cap \Sigma|=L$. Then we must have $|r^{-1}(s(x''))\cap \Sigma|=L$. By our choice of $\{A_1,\cdots,A_k\}$, there exists $i$ such that $s(x'')\in A_i$. So there is a path $\gamma$ in $A_i$ from $s(x'')$ to $x_i$ such that $\gamma$ maps $(0,1)\subset I$ to the interior of $A_i$. Then we lift $\gamma$ to a path $\gamma''$ in $\mathcal{O}''$ from $x''$ to some $x_i''\in \mathcal{O}''$ with $s(x_i'')=x_i$. Since $\gamma''$ maps $(0,1)$ to a component of $\mathcal{O}''\setminus r''(\partial \Sigma)$, $r''(\partial \Sigma)\setminus X''$ or $X''$ ($X''$ is the set of all self-intersections of $r''(\partial \Sigma)$), we have
$$|(r'')^{-1}(x_i'')\cap \Sigma |=|(r'')^{-1}(x'')\cap \Sigma|=L.$$ However, we must have $|(r'')^{-1}(x_i'')\cap \Sigma| < |r^{-1}(x_i)\cap \Sigma|=L$, since one of $y_i,y_i'\in r^{-1}(x_i)\cap \Sigma$ is not in $(r'')^{-1}(x_i'')\cap \Sigma$.

We have proved that  $C_{\Sigma}(M'')<C_{\Sigma}(M')$, but it contradicts with that $M'$ is a minimizer of $C_{\Sigma}(\cdot)$. So for the minimizer $M'$ and $V\subset M'$, the image of $\rho:\mathcal{O}_V\to \mathbb{N}$ consists of only two integers, thus $\Sigma$ is a properly embedded subsurface of $\hat{V}=r^{-1}(r(\Sigma))\subset V$.

\end{proof}

The following lemma implies the existence of an $M'\in \mathcal{I}$, such that for any two pieces $\Sigma_1,\Sigma_2\subset \Phi(H)$ lying in the same Seifert piece $V\subset M'$, either $\Sigma_1$ and $\Sigma_2$ are parallel in $V$ (as in Lemma \ref{parallel}), or they have disjoint projection in the base orbifold of $V$. Actually, this lemma is not absolutely necessary for the proof of "separability implies aspirality" part of Theorem \ref{main}. Without this lemma, the $\hat{M}$ in Proposition \ref{cut} will be a manifold with a $\pi_1$-injective immersion into $M'$, instead of simply a submanifold of $M'$.

\begin{lemma}\label{disjoint}
  Suppose that $H< \pi_1(M)$ and $\Phi(H)\to M$ satisfy Lemma \ref{parallel} and \ref{proper}. Then there exists an intermediate finite cover $M'\to M$ of $M_H\to M$, such that for any two distinct pieces $\Sigma_1,\Sigma_2$ of $\Phi(H)$, one of the following hold.
  \begin{enumerate}
    \item $\Sigma_1$ and $\Sigma_2$ lie in two distinct pieces of $M'$.
    \item $\Sigma_1$ and $\Sigma_2$ lie in the same piece $V\subset M'$ and they are parallel to each other (as in Lemma \ref{parallel}).
    \item $\Sigma_1$ and $\Sigma_2$ are partially fibered subsurfaces lying in the same Seifert piece $V\subset M'$, and $r(\Sigma_1)$ is disjoint from $r(\Sigma_2)$ in $\mathcal{O}_V$. Here $r:V\to \mathcal{O}_V$ is the projection from $V$ to its base orbifold.
  \end{enumerate}
\end{lemma}

\begin{proof}
  At first, for a pair of pieces $\Sigma_1, \Sigma_2\subset \Phi(H)$, if they satisfy one of these three conditions for some $M'$, then they still satisfy one of them in any intermediate finite cover $M''\to M'$ of $M_H\to M'$ (condition (2) is preserved by Remark \ref{alwayssametori}). So we only need to work on one pair of pieces $\Sigma_1,\Sigma_2\subset \Phi(H)$ and assume they lie in the same piece $V\subset M$.

  At first, there is nothing to prove if $V$ is hyperbolic. In the hyperbolic case, $V\cap \Phi(H)$ consists of a disjoint union of fibered subsurfaces (Lemma \ref{normalize} (3)), so they must be parallel to each other. Similarly, if $V$ is a Seifert manifold and one component of $V\cap \Phi(H)$ is a fibered subsurface, it intersects with all boundary components of $V$. Then Lemma \ref{parallel} (2) implies that all components of $V\cap \Phi(H)$ are parallel fibered subsurfaces.

  The remaining case is that $V$ is a Seifert manifold and all components of $V\cap \Phi(H)$ are partially fibered subsurfaces. By Lemma \ref{parallel} (2), if $\Sigma_1$ is not parallel to $\Sigma_2$, each boundary component of $V$ intersects with at most one of $\Sigma_1$ and $\Sigma_2$. Take $*\in \Sigma_1\cap \partial V$ that lies in a boundary component $c\subset \partial \Sigma_1$, $*'\in \Sigma_2\cap \partial V$ and an oriented path $\alpha$ in the Scott core $C_H$ from $*$ to $*'$.

  Let $T_1=r(\Sigma_1)$ and $T_2=r(\Sigma_2)$ be the projections of $\Sigma_1$ and $\Sigma_2$ in $\mathcal{O}_V$ respectively, then Lemma \ref{proper} implies that $r|_{\Sigma_1}:\Sigma_1\to T_1$ and $r|_{\Sigma_2}:\Sigma_2\to T_2$ are both covering maps. So boundary components of $T_1$ and $T_2$ are closed geodesics in $\mathcal{O}_V$, and $T_1\cap T_2$ has only finitely many connected components.

  Let $A_1,\cdots,A_k$ be all components of $T_1\cap T_2$. For each $i$, take a point $x_i\in A_i$. Let $r^{-1}(x_i)\cap \Sigma_1=\{y_{i,1},\cdots,y_{i,n}\}$ and let $r^{-1}(x_i)\cap \Sigma_2=\{z_{i,1},\cdots,z_{i,m}\}$. For each pair $s,t$ with $s\in \{1,\cdots,n\}$ and $t\in \{1,\cdots,m\}$, we take an oriented path $\beta_{i,s}$ in $\Sigma_1$ from $*$ to $y_{i,s}$, an oriented path $\gamma_{i,t}$ in $\Sigma_2$ from $*'$ to $z_{i,t}$ and an oriented path $\delta_{i,t,s}$ from $z_{i,t}$ to $y_{i,s}$ that lies in the regular fiber that projects to $x_i$. Then $f_{i,t,s}=\alpha\gamma_{i,t}\delta_{i,t,s}\beta_{i,s}^{-1}$ is an element in $\pi_1(M',*)$ (as in Figure 3).

  \begin{center}
\includegraphics[width=2in]{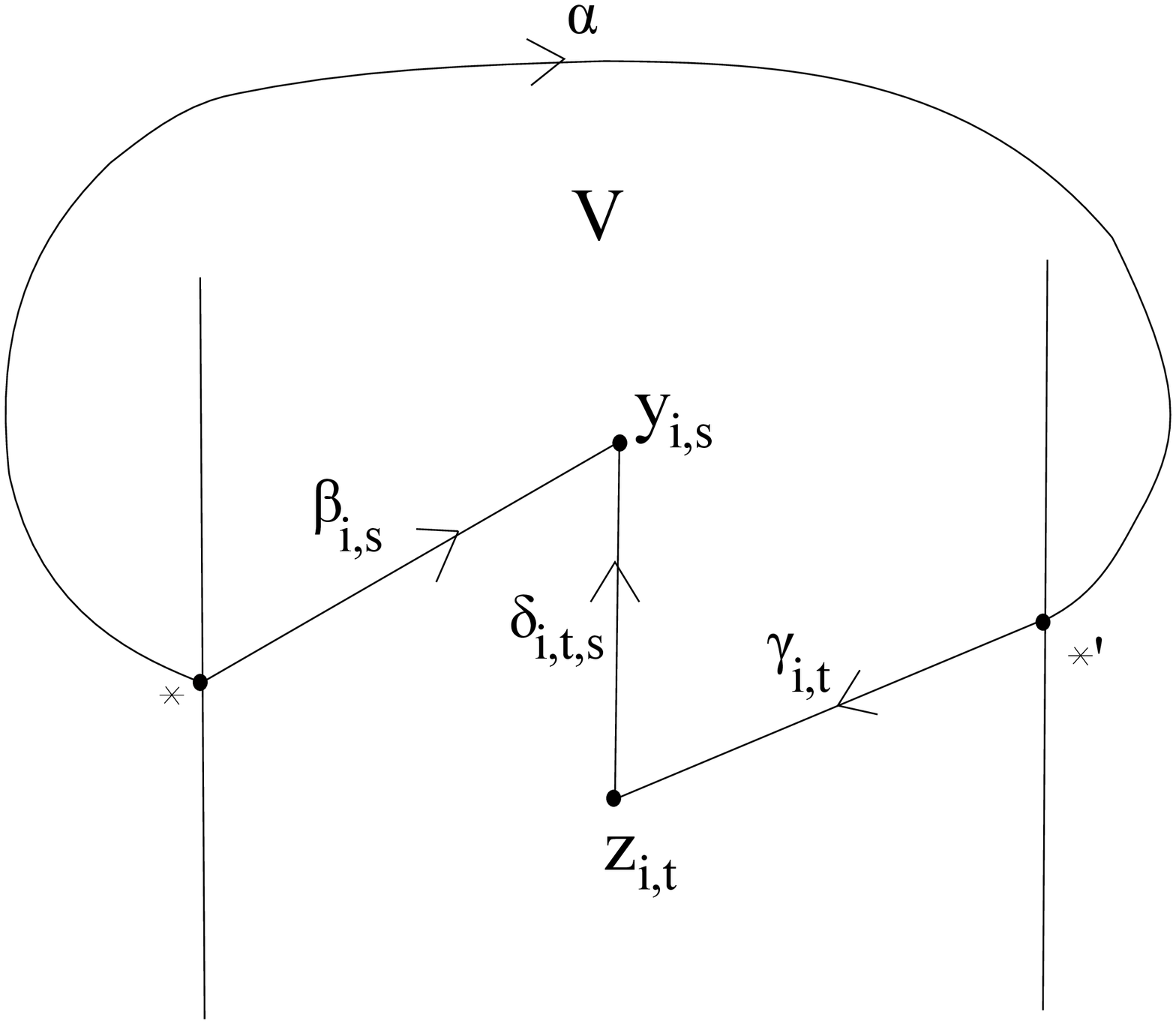}
 \centerline{Figure 3}
\end{center}

  Since $\Sigma_2$ does not intersect with the boundary component of $V$ containing $c$, any closed path in $\Sigma_2$ is not homotopy to $c$ in $V$. Then an argument as in Claim 1 of Lemma \ref{parallel} implies that $f_{i,t,s}cf_{i,t,s}^{-1}\notin H$. By separability of $H$, there exists an intermediate finite cover $M'\to M$ of $M_H\to M$ such that $f_{i,t,s}cf_{i,t,s}^{-1}\notin \pi_1(M')$ for all $i\in \{1,\cdots,k\}$, $s\in \{1,\cdots,n\}$ and $t\in \{1,\cdots,m\}$. If $\Sigma_1$ and $\Sigma_2$ do not lie in the same piece of $M'$, then the proof is done. Otherwise, suppose that both $\Sigma_1$ and $\Sigma_2$ lie in a Seifert piece $V'\subset M'$. Then for any triple $(i,t,s)$ as above, $y_{i,s}$ and $z_{i,t}$ do not lie in the same regular fiber of $V'$, by the argument in Lemma \ref{proper}.

So for $M'$, the projections of $\Sigma_1$ and $\Sigma_2$ are disjoint in $\mathcal{O}_{V'}$, and the proof is done.

\end{proof}

Now we are ready to prove Proposition \ref{cut}.

\begin{proof}

We take a finite cover $M'\in \mathcal{I}$ of $M$ that satisfies Lemma \ref{parallel}, \ref{proper} and \ref{disjoint}. Then $\hat{M}$ will be constructed by pasting submanifolds of pieces of $M'$, instructed by a quotient graph of $G_{\Phi(H)}$ (the dual graph of $\Phi(H)$).

On $G_{\Phi(H)}$, we define two vertices $v,v'$ to be equivalent if the corresponding pieces $\Sigma^v,\Sigma^{v'}\subset \Phi(H)$ lie in the same piece $V\subset M'$, and $V$ has a boundary component $T$ that intersects with both $\Sigma^v$ and $\Sigma^{v'}$. We define two edges $e,e'\subset G_{\Phi(H)}$ to be equivalent if their corresponding decomposition circles $c^e,c^{e'}\subset \Phi(H)$ lie in the same decomposition torus $T\subset M'$.

If two vertices $v,v'$ of $G_{\Phi(H)}$ are in the same equivalence class, then Lemma \ref{parallel} (2) implies that they have the same valence. Moreover, each edge $e$ adjacent to $v$ is equivalent to an edge $e'$ adjacent to $v'$ and vice versa. If two edges $e,e'$ of $G_{\Phi(H)}$ are in the same equivalence class, then the two vertices adjacent to $e$ are equivalent to the two vertices adjacent to $e'$.

The quotient of $G_{\Phi(H)}$ under this equivalence relation gives rise to a graph $\hat{G}$, and the manifold $\hat{M}$ we will construct has dual graph isomorphic to $\hat{G}$.

\bigskip

For each vertex $[v]\in \hat{G}$, take a representative vertex $v\in G_{\Phi(H)}$ of $[v]$, and let $\Sigma^v\subset \Phi(H)$ be the piece of $\Phi(H)$ corresponding to $v$. If $\Sigma^v$ is a fibered subsurface, then we take $M^{[v]}$ to be the piece of $M'$ containing $\Sigma^v$, and it is clear that $M^{[v]}$ is independent of the choice of $v\in [v]$.

If $\Sigma^v$ is a partially fibered subsurface, let $M^v$ be the Seifert piece of $M'$ that contains $\Sigma^v$, $\mathcal{O}_{M^v}$ be the base orbifold of $M^v$ and $r:M^v\to \mathcal{O}_{M^v}$ be the projection. Then we take $M^{[v]}=r^{-1}(r(\Sigma^v))$, and Lemma \ref{proper} implies that $\Sigma^v\subset M^{[v]}$ is a properly embedded subsurface. For another vertex $v'\in [v]$, Lemma \ref{parallel} (2) implies that $\Sigma^v$ is parallel to $\Sigma^{v'}$ in $M^v$. Since both $\Sigma^v$ and $\Sigma^{v'}$ are obtained by the convex core construction (or they are both the closed $\epsilon$-neighborhood of the same boundary component), we must have $r(\Sigma^v)=r(\Sigma^{v'})$. So $M^{[v]}$ is a well-defined submanifold of $M^v$, and $\partial M^{[v]}\setminus \partial M^v$ is a union of vertical tori in $M^v$. If we have two partially fibered pieces $\Sigma^v,\Sigma^w\subset \Phi(H)$ lying in the same Seifert piece $V\subset M'$ such that no boundary component of $V$ intersects with both $\Sigma^v$ and $\Sigma^w$, then Lemma \ref{disjoint} (3) implies that the submanifolds $M^{[v]}$ and $M^{[w]}$ are disjoint from each other.

For each edge $[e]\in \hat{G}$, it is represented by some edge $e\in G_{\Phi(H)}$. Let $c^e\subset \Phi(H)$ be the decomposition circle corresponding to $e$, and let $T^{[e]}\subset M'$ be the decomposition torus containing $c^e$, then $T^{[e]}$ is clearly well defined.

\bigskip

We define $\hat{M}$ to be the graph of space with dual graph $\hat{G}$, such that its vertex and edge spaces are the $M^{[v]}$ and $T^{[e]}$ constructed as above. It is naturally a codim-$0$ submanifold of $M'$, and condition (1) holds by the construction of vertex spaces $M^{[v]}$.

By our construction of $\hat{M}$, $\Phi(H)$ is a subsurface of $\hat{M}$. So it suffices to check that $\partial (\Phi(H))\subset \partial \hat{M}$, which implies that $\Phi(H)\subset \hat{M}$ is proper. Let $c$ be a boundary component of $\Phi(H)$, let $\Sigma$ be the piece of $\Phi(H)$ that contains $c$, and let $V$ be the vertex space of $M'$ that contains $\Sigma$. If $c$ lies in a decomposition torus $T\subset M'$, then Lemma \ref{parallel} (1) implies that $T$ is a boundary component of $\hat{M}$. If $c$ lies in the interior of a Seifert piece $V\subset M'$, then $\Sigma$ is a partially fibered subsurface. Then the construction of $M^{[v]}$ and Lemma \ref{proper} imply that $c$ lies in a boundary component of $\hat{M}$. So $\Phi(H)\subset \hat{M}$ is a properly embedded subsurface, and the proof is done.

\end{proof}

\subsection{Proof of separability implies aspirality}

Now we are ready to prove the "only if" part of Theorem \ref{main}.

\begin{proof}
We first apply Lemma \ref{embedding} and \ref{normalize} to $M$ and $H$, to get a finite cover $M'\to M$ and $H'=H\cap \pi_1(M')$ satisfying these two lemmas.

Let $M'_{H'}\to M'$ be the covering space of $M'$ corresponding to $H'<\pi_1(M')$. Then Proposition \ref{cut} implies that there exists an intermediate finite cover $M''\to M'$ of $M'_{H'}\to M'$ and a codim-$0$ submanifiold $\hat{M}\subset M''$, such that $\partial \hat{M}$ is a union of incompressible tori in $M''$, and $\Phi(H')\subset \hat{M}$ is a properly embedded subsurface.

Since we only need to prove the generalized spirality character is trivial on each component of $\Phi(H')$, we can assume that both $\Phi(H')$ and $\hat{M}$ are connected. Since $\Phi(H')\subset \hat{M}$ is a properly embedded subsurface and each piece of $\Phi(H')$ is a fibered subsurface of the corresponding piece of $\hat{M}$, $\hat{M}$ has a $\Phi(H')$-bundle over $S^1$ structure. Then we take a finite cyclic cover $\bar{M}$ of $\hat{M}$ along $\Phi(H')$, such that the monodromy of the fibering structure on $\bar{M}$ fixes all pieces of $\Phi(H')$, and the pseudo-Anosov and periodic canonical forms on pieces of $\Phi(H')$ fix the boundary pointwisely.

It follows from Definition \ref{spiralitycharacterdef} that the generalized spirality character of $\Phi(H')$ in $\bar{M}$ is trivial. This is because that each piece $\Sigma\subset \Phi(H')$ is a fibered subsurface in the corresponding piece $\bar{V}\subset \bar{M}$, and $\bar{V}$ satisfies the requested conditions in Case II of Section \ref{arcs}. So in the formula of generalized spirality character in Section \ref{spiralitycharacterdef}, all $s_{\delta}$ for $H'<\pi_1(\bar{M})$ are equal to $1$, thus the generalized spirality character of $\Phi(H')$ in $\bar{M}$ is trivial.

Lemma \ref{manifoldcoverspirality} and the definition of generalized spirality character together imply the generalized spirality character of $\Phi(H')$ in $M'$ is trivial. Then Lemma \ref{manifoldcoverspirality} and \ref{subgroupspirality} imply that the generalized spirality character of $\Phi(H)$ in $M$ is trivial.
\end{proof}

\section{Constructing finite covers of hyperbolic and Seifert pieces}\label{casebycase}

In the following two sections, we will prove the "aspirality implies separability" part of Theorem \ref{main}.

In this section, we do some preparation for this proof. For various types of finitely generated infinite index subgroup $L$ of an orientable finite volume hyperbolic or Seifert $3$-manifold group $\pi_1(N)$, we construct an intermediate finite cover $\hat{N}\to N$ of $N_L\to N$ whose restriction to $\partial \hat{N}$ has some pre-selected pattern. In the next section, we will paste these finite covers of pieces of $M$ together, to get a finite semi-cover of $M$.

The possibilities of a finitely generated subgroup of a finite volume hyperbolic or Seifert $3$-manifold group was studied in Section \ref{subgroups}, and they are summarized in Table 1.

\begin{table}\label{covers}
\begin{center}
\begin{tabular}{ |c |c |c |}
\hline
  f.g. subgroup & cover of hyperbolic $N$  & cover of Seifert $N$ \\
  \hline
 finite index & finite cover & finite cover \\
 \hline
 no contribution  &\multirow{2}{*} {geometrically finite}  & $S^1$-bundle over \\
 to $\Phi(H)$&   &  a non-compact $2$-orbifold\\
 \hline
may & & $\mathbb{R}$-bundle over a compact surface\\
contribute & geometrically infinite & i.e. virtually fibered\\
 \cline{3-3}
 to $\Phi(H)$& i.e. virtually fibered & $\mathbb{R}$-bundle over a non-compact surface \\
 & & i.e. partially fibered\\
 \hline
\end{tabular}
\caption{Covering spaces of finite volume hyperbolic and Seifert pieces.}
\end{center}
\end{table}

For each finite volume hyperbolic or Seifert $N$ as above, let $L$ be a finitely generated subgroup of $\pi_1(N)$ and let $N_L$ be the covering space of $N$ corresponding to $L<\pi_1(N)$. Then each boundary component of $N_L$ is a covering space of a torus, so it is a torus, a cylinder or a plane. Table 2 summarizes all possible types of boundary components of $N_L$, and whether $N_L$ has finitely many or infinitely many boundary components of a given type.

\begin{table}\label{boundaries}
\begin{center}
\begin{tabular}{ | c |c |c |c | c|}
\hline
  type of $N$ & subgroup & torus & cylinder  & plane \\
  \hline
 hyp. or Seifert & finite index & finite & no & no \\
 \hline
 hyperbolic & geometrically finite &finite& finite & possibly infinite\\
 \hline
  hyperbolic & geometrically infinite  &  no & finite& no\\
  \hline
  Seifert & contains a power of fiber & finite & possibly infinite & no \\
  \hline
  Seifert & virtually fibered &  no & finite & no\\
  \hline
  Seifert & partially fibered  & no & finite & possibly infinite\\
  \hline
\end{tabular}
\caption{Boundary components of covering spaces of finite volume hyperbolic and Seifert pieces.}
\end{center}
\end{table}

\bigskip

Now we set up some notations that will be used throughout this section. We assume $L<\pi_1(N)$ is a finitely generated infinite index subgroup, and use $\pi:N_L\to N$ to denote the covering space of $N$ corresponding to $L<\pi_1(N)$. The boundary components of $N_L$ consists of tori $T_1,\cdots,T_l$, cylinders $C_1,\cdots,C_m, \cdots$, and planes $P_1,\cdots,P_n,\cdots$. Note that $N_L$ may not have all three types of boundary components (see Table 2).

For each cylinder $C_j\subset \partial N_L$, it contains a unique essential simple closed curve $c_j$ (up to isotopy), which is called the {\it core} of $C_j$. Then $\pi(c_j)=x_js_j$ for some $x_j\in \mathbb{Z}_+$ and some essential simple closed curve (slope) $s_j$ in the torus $\pi(C_j)\subset \partial N$. Then we take another slope $t_j\subset \pi(C_j)$ not parallel to $s_j$. For each plane $P_k\subset \partial N_L$, we take two nonparallel slopes $u_k,v_k$ in the torus $\pi(P_k)\subset \partial N$.

\subsection{Finite volume hyperbolic case}\label{hyperbolic}

We first work on finitely generated infinite index subgroups of finite volume hyperbolic $3$-manifold groups.

If $L<\pi_1(N)$ is a geometrically infinite subgroup, then all boundary components of $N_L$ are cylinders. It turns out to be the simplest case.

\begin{proposition}\label{hypfiber}
Let $N$ be a finite volume hyperbolic $3$-manifold, $L<\pi_1(N)$ be a geometrically infinite subgroup that corresponds to an orientable fibered surface, and $K\subset N_L$ be any compact subset. Then there exist $A_1,\cdots,A_l\in \mathbb{Z}_+$, such that for any $\alpha\in \mathbb{Z}_+$, there exists an intermediate finite cover $q:\hat{N}\to N$ of $\pi:N_L\to N$, such that the following conditions hold for commutative diagram:\begin{diagram}
N_L &\rTo^p & \hat{N}\\
&\rdTo_{\pi} &\dTo_{q}\\
& &N.
\end{diagram}
\begin{enumerate}
  \item $p|_{K}:K\to \hat{N}$ is an embedding.
  \item $p$ maps any two distinct cylinders in $\partial N_L$ to distinct boundary components of $\hat{N}$.
  \item For each cylinder $C_j\subset \partial N_L$, $p$ maps its core $c_j$ to a simple closed curve in $p(C_j)$, and $q|_{p(C_j)}:p(C_j)\to \pi(C_j)$ is the finite cover corresponding to subgroup $\mathbb{Z}[c_j]\oplus \mathbb{Z}[\alpha A_jt_j]<\pi_1(\pi(C_j))$.
\end{enumerate}
(Here $\mathbb{Z}[c_j]\oplus \mathbb{Z}[\alpha A_jt_j]<\pi_1(\pi(C_j))$ denotes the subgroup of $\pi_1(\pi(C_j))$ that is the direct sum of two infinite cyclic subgroups generated by $c_j$ and $\alpha A_jt_j$ respectively.)
\end{proposition}

\begin{proof}
Since $L$ is a geometrically infinite subgroup corresponding to an orientable fibered surface, $N_L$ is homeomorphic to $\Sigma_L\times \mathbb{R}$ with $\Sigma_L\cap C_j=c_j$ for each cylinder $C_j\subset \partial N_L$. Then $N_L\to N$ has an intermediate finite cover $N'=\Sigma_L\times I/\phi$ that is a mapping torus. By taking a power of $\phi$, we can assume that the projection from $K\subset N_L$ to $N'$ is an embedding, $\phi$ fixes $\partial \Sigma_L$ pointwisely, and each elevation of $t_j$ in $N'$ intersects with $\Sigma_L$ exactly once.

We still use $\phi$ to denote this power of $\phi$. Let $t_j'$ be an elevation of $t_j$ in $N'$ that lies in the image of $C_j\subset \partial N_L$, and let $A_j$ be the covering degree of $t_j'\to t_j$. Then for any $\alpha\in \mathbb{Z}_+$, the manifold $\hat{N}=\Sigma_L\times I/\phi^{\alpha}$ satisfies all desired conditions.

\end{proof}

\begin{remark}
The manifold $\hat{N}$ constructed in Proposition \ref{hypfiber} satisfies the conditions of $J^v$ in Case II of Section \ref{arcs}, so it can be used to define $s_{\delta}$ for an oriented proper path $\delta$ in a fibered subsurface.
\end{remark}

Now we deal with the geometrically finite case. In this case, boundary components of $N_L$ can be torus, cylinder and plane, and there might be infinitely many plane boundary components. So the statement becomes more complicated.

\begin{proposition}\label{geofinite}
Let $N$ be a finite volume hyperbolic $3$-manifold, $L<\pi_1(N)$ be a geometrically finite subgroup, and $K\subset N_L$ be any compact subset. Let $\mathcal{P}=\{P_1,\cdots,P_n\}$ be a finite collection of plane boundary components of $\partial N_L$. Then there exists $A\in \mathbb{Z}_+$, such that for any positive integers $\alpha_1,\cdots,\alpha_m, \beta_1,\cdots,\beta_n,\delta_1,\cdots,\delta_n\in \mathbb{Z}_+$, there exists an intermediate finite cover $q:\hat{N}\to N$ of $\pi:N_L\to N$, such that the following conditions hold for commutative diagram:\begin{diagram}
N_L &\rTo^p & \hat{N}\\
&\rdTo_{\pi} &\dTo_{q}\\
& &N.
\end{diagram}
\begin{enumerate}
  \item $p|_{K}:K\to \hat{N}$ is an embedding.
  \item For any two distinct boundary components of $N_L$ in $\{T_1,\cdots,T_l,C_1,\cdots,C_m,P_1,\cdots,P_n\}$, $p$ maps them to distinct boundary components of $\hat{N}$.
  \item For each torus boundary component $T_i\subset \partial N_L$, $p|_{T_i}:T_i\to p(T_i)$ is a homeomorphism.
  \item For each cylinder $C_j\subset \partial N_L$, $p$ maps $c_j$ to a simple closed curve in $p(C_j)$, and $q|_{p(C_j)}:p(C_j)\to \pi(C_j)$ is the finite cover corresponding to $\mathbb{Z}[c_j]\oplus \mathbb{Z}[\alpha_j At_j]<\pi_1(\pi(C_j))$.
  \item For each plane $P_k\in \mathcal{P}$, $q|_{p(P_k)}:p(P_k)\to \pi(P_k)$ is the finite cover corresponding to $\mathbb{Z}[\beta_kAu_k]\oplus \mathbb{Z}[\delta_kAv_k]<\pi_1(\pi(P_k))$.
\end{enumerate}
\end{proposition}

\begin{proof}
We take a base point $*\in K\subset N_L$ and $*$ will be the base point of all fundamental groups in this proof. By enlarging $K$, we can assume that $K$ is path-connected, it contains all torus boundary components $T_1,\cdots,T_l$ of $N_L$, and it intersects with all of $C_1,\cdots,C_m$ and $P_1,\cdots,P_n$.

Since $L<\pi_1(N)$ is separable, there exists an intermediate finite cover of $N_L\to N$ such that $K$ is mapped into this finite cover by embedding. By abusing notation, we still denote this finite cover by $N$. Then $N$ satisfies condition (1) and (2) in this proposition.

For each boundary component $T_i$, $C_j$ or $P_k$ of $N_L$, we take a path in $K$ from $*$ to $T_i$, $C_j$ or $P_k$ respectively, and consider $\pi_1(\pi(T_i)), \pi_1(\pi(C_j))$ and $\pi_1(\pi(P_k))$ as subgroups of $\pi_1(N,*)$ by conjugation them by these paths.

Then we have: $L\cap \pi_1(\pi(T_i))=\pi_1(\pi(T_i))$, $L\cap \pi_1(\pi(C_j))=\langle c_j \rangle$ and $L\cap \pi_1(\pi(P_k))=\{e\}$. Note that $\pi_1(N)$ is a relative hyperbolic group, $L<\pi_1(N)$ is a relatively quasi-convex subgroup, and $\pi_1(\pi(C_j)), \pi_1(\pi(P_k))<\pi_1(N)$ are maximal parabolic subgroups. By applying Theorem 1.1 of \cite{MP} inductively, there exist $A_1,\cdots,A_m, B_1,\cdots,B_n, D_1,\cdots,D_n\in \mathbb{Z}_+$, such that the subgroup $L''<\pi_1(N)$ generated by
$$L,\langle c_j, A_jt_j\rangle<\pi_1(\pi(C_j))\ j=1,\cdots,m,\  \text{and}\ \langle B_ku_k, D_kv_k\rangle<\pi_1(\pi(P_k))\ k=1,\cdots,n$$
is isomorphic to an iterated $\mathbb{Z}$-amalgamation and free product:
$$\Big(L*_{\{\langle c_j\rangle\}_{j=1}^m}(*_{j=1}^m\langle c_j, A_jt_j\rangle)\Big)*(*_{k=1}^n \langle B_ku_k, C_kv_k\rangle).$$ Moreover, each parabolic subgroup of $L''$ is conjugate to a subgroup of $L$, $\langle c_j, A_jt_j\rangle$ or $\langle B_ku_k, C_kv_k\rangle$.

Let $A$ be the least common multiple of above positive integers $A_1,\cdots,A_m, B_1,\cdots,B_n, D_1,\cdots,D_n$. Then for any positive integers $\alpha_1,\cdots,\alpha_m, \beta_1,\cdots,\beta_n,\delta_1,\cdots,\delta_n\in \mathbb{Z}_+$, we take the subgroup $L'<L''<\pi_1(N)$ generated by
$$L,\langle c_j, \frac{\alpha_jA}{A_j}A_jt_j\rangle<\pi_1(\pi(C_j))\ \text{and}\ \langle \frac{\beta_kA}{B_k}B_ku_k, \frac{\delta_kA}{D_k}D_kv_k\rangle<\pi_1(\pi(P_k))$$ for $j=1,\cdots,m$ and $k=1,\cdots,n$.
 Then $L'$ has a similar group of graph structure as $L''$, and each parabolic subgroup of $L'$ is conjugate to a subgroup of one of these factors.

Let $N_{L'}$ be the covering space of $N$ corresponding to $L'<\pi_1(N)$ (which is an intermediate cover of $N_L\to N$). Then the following boundary components of $N_{L'}$ are distinct from each other: images of $T_1,\cdots,T_l$, the tori boundary components of $N_{L'}$ corresponding to $\langle c_j, \alpha_jAt_j\rangle$ ($j=1,\cdots,m$) and $\langle \beta_kAu_k, \delta_kAv_k\rangle$ ($k=1,\cdots,n$). This is because their corresponding subgroups are not conjugate to each other in $L'$.

The covering spaces we constructed are summarized by the following tower of covering spaces: $$N_L\to N_{L'}\to N_{L''}\to N,$$ and we will apply LERFness of $\pi_1(N)$ to $N_{L'}\to N$ to get $\hat{N}$.

We take a compact subset $K'\subset N_{L'}$ that contains $K$ and all tori boundary components of $N_{L'}$. Since $L'$ is finitely generated, it is separable in $\pi_1(N)$. So there exists an intermediate finite cover $\hat{N}\to N$ of $N_{L'}\to N$, such that $K'$ is projected into $\hat{N}$ by embedding. Then all desired conditions hold for $\hat{N}$. $K$ embeds in $\hat{N}$ since $K'$ embeds in $\hat{N}$. The boundary components of $N_L$ under consideration (i.e. $T_1,\cdots,T_l,C_1,\cdots,C_m,P_1,\cdots,P_n$) are mapped to distinct boundary components of $\hat{N}$, since their projections to $N_{L'}$ are distinct tori and they are all contained in $K'$. The covering maps from boundary components of $\hat{N}$ to boundary components of $N$ are the desired ones, since the restriction of $N_{L'}\to \hat{N}$ on these boundary components are homeomorphisms.
\end{proof}

\subsection{Seifert case}\label{Seifert}

Now we work on the case that $N$ is a Seifert space. We always assume that $N$ is not the twisted $I$-bundle over Klein bottle, so the base orbifold $\mathcal{O}_N$ of $N$ has negative Euler characteristic. Some proofs in this section are borrowed from proofs in Section \ref{hyperbolic}, since we will first apply techniques in hyperbolic geometry to the base orbifold.

If $L<\pi_1(N)$ is a virtually fibered subgroup, then the corresponding statement and its proof are same with the hyperbolic case.

\begin{proposition}\label{Seifertfiber}
Let $N$ be a Seifert $3$-manifold with orientable base orbifold, $L<\pi_1(N)$ be a virtually fibered subgroup, and $K\subset N_L$ be any compact subset. Then there exist $A_1,\cdots,A_l\in \mathbb{Z}_+$, such that for any $\alpha\in \mathbb{Z}_+$, there exists an intermediate finite cover $q:\hat{N}\to N$ of $\pi:N_L\to N$, such that the following conditions hold for commutative diagram:
\begin{diagram}
N_L &\rTo^p & \hat{N}\\
&\rdTo_{\pi} &\dTo_{q}\\
& &N.
\end{diagram}
\begin{enumerate}
  \item $p|_{K}:K\to \hat{N}$ is an embedding.
  \item $p$ maps any two distinct cylinder boundary components of $N_L$ to distinct boundary components of $\hat{N}$.
  \item For each cylinder $C_j\subset \partial N_L$, $p$ maps its core $c_j$ to a simple closed curve in $p(C_j)$, and $q|_{p(C_j)}:p(C_j)\to \pi(C_j)$ is the finite cover corresponding to $\mathbb{Z}[c_j]\oplus \mathbb{Z}[\alpha A_jt_j]<\pi_1(\pi(C_j))$.
\end{enumerate}
\end{proposition}

\begin{remark}\label{iffiber}
If we take all $t_j$ to be the regular fiber $h$ in $\pi(C_j) \subset \partial N$, then all $A_j$s must be equal to each other. In this case, the covering degree from $p(C_j)$ to $\pi(C_j)$ equals $\alpha A_j\langle c_j,t_j\rangle=\alpha A_j \langle c_j,h\rangle$. Then the quotient of two such covering degrees is $\frac{\langle c_i,h\rangle}{\langle c_j,h\rangle}$, which equals the $s_{\delta}$ defined in Case I of Section \ref{arcs}.
\end{remark}

To prove the corresponding results for partially fibered subgroups and subgroups corresponding to $S^1$-bundles, we need the following two results on covering spaces of surfaces.

\begin{lemma}\label{infinitecoversurface}
  Let $S$ be a compact orientable hyperbolic surface with boundary, $A<\pi_1(S)$ be a finitely generated infinite index subgroup, and $S_A$ be the covering space of $S$ corresponding to $A$ with convex core $\Sigma_A\subset S_A$. Let $\mathcal{L}=\{L_1,\cdots,L_n\}$ be a finite collection of boundary components of $S_A$ homeomorphic to $\mathbb{R}$, then there exists an intermediate finite cover $\hat{S}\to S$ of $S_A\to S$ such that the following hold.
        \begin{enumerate}
          \item $\Sigma_A\subset S_A$ is mapped to $\hat{S}$ by embedding, and $\hat{S}\setminus \Sigma_A$ is connected.
          \item Let $\mathcal{C}$ be the collection of all boundary components of $S_A$ homeomorphic to $S^1$, then any two distinct components of $\mathcal{C}\cup \mathcal{L}$ are mapped to distinct boundary components of $\hat{S}$.
          \item There is a boundary component of $\hat{S}$ that is not in the image of $\mathcal{C}\cup \mathcal{L}$.
        \end{enumerate}
\end{lemma}

Note that when $A$ is an infinite cyclic parabolic subgroup, our convention is that, $\Sigma_A$ is a closed $\epsilon$-neighborhood of the corresponding circle boundary component of $S_A$. When $A$ is the trivial subgroup, its convex core is the empty set.

\begin{proof}
 Since $A$ is finitely generated, the convex core $\Sigma_A$ has finite area. So $\Sigma_A$ has finitely many boundary components, and it contains all circle boundary components of $S_A$. Since $\pi_1(S)$ is LERF, there exists an intermediate finite cover of $S_A\to S$ such that $\Sigma_A$ embeds into this finite cover, and we still denote this finite cover by $S$. Then the first half of condition (1) holds for $S$.

 We first prove that $S_A$ has infinitely many boundary components. Now we consider $\pi_1(S)$ as a discrete subgroup of $\text{Isom}_+(\mathbb{H}^2)$ such that $\mathbb{H}^2/\pi_1(S)$ is a finite area hyperbolic surface with cusp ends. Then it suffices to prove that, in the domain of discontinuity $\Omega(A)\subset \partial \mathbb{H}_{\infty}^2$, there are infinitely many $A$-orbits of parabolic fixed points of $\pi_1(S)$. If $A$ is the trivial subgroup, we take any parabolic fixed point $x\in \partial \mathbb{H}^2$ of $\pi_1(S)$, and any hyperbolic element $\gamma\in \pi_1(S)$ such that $\gamma(x)\ne x$. Then $\{\gamma^n(x)\}_{n\in \mathbb{Z}}$ gives us infinitely many distinct parabolic fixed points of $\pi_1(S)$. The case that $A$ is an infinite cyclic parabolic subgroup can be proved similarly.

 It remains to consider the case that $\Sigma_A$ has nonempty geodesic boundary in the interior of $S_A$. For any such geodesic boundary component $c\subset \partial \Sigma_A$, let $\delta$ be one elevation of $c$ in $\mathbb{H}^2$. Then the universal cover of $\Sigma_A$ lies on one side of $\delta\subset \mathbb{H}^2$, while the other side is bounded by $\delta$ and a path $\alpha\subset \partial \mathbb{H}^2$. We take any parabolic fixed point $x\in \alpha$ and any hyperbolic element $\gamma\in \pi_1(S)\setminus A$ such that $\gamma$ has a contracting fixed point in $\text{int}(\alpha) \setminus \{x\}$. Then for $N$ large enough, $\{\gamma^n(x)\}_{n>N}$ gives us infinitely many distinct $\langle c\rangle$-orbits of parabolic fixed points of $\pi_1(S)$ in $\alpha$, and they are also inequivalent under the $A$-action .

\bigskip

 Since $S_A$ has infinitely many boundary components, we take a line boundary component $L_{n+1}\subset \partial S_A$ not in $\mathcal{L}$, and enlarge $\mathcal{L}$ to $\mathcal{L}'=\mathcal{L}\cup \{L_{n+1}\}$. Let $l_k$ be a generator of the parabolic subgroup of $\pi_1(S)$ corresponding to $L_k$. Similar to the argument in Proposition \ref{geofinite}, we can apply Theorem 1.1 of \cite{MP} to get positive integers $E_1,\cdots,E_{n+1}$, such that the subgroup $A'<\pi_1(S)$ generated by $A$ and $E_1l_1,\cdots,E_{n+1}l_{n+1}$ is isomorphic to $A*(*_{k=1}^{n+1}\langle E_kl_k\rangle)$, and each parabolic element in $A'$ either conjugates into $A$ or conjugates into $\langle E_kl_k\rangle$ for some $k$.

 The covering space $S_{A'}\to S$ corresponding to $A'<\pi_1(S)$ is an intermediate cover of $S_A\to A$, and we have the following tower of covering spaces: $$S_A\to S_{A'}\to S.$$
By our construction, distinct components of $\mathcal{C}\cup \mathcal{L}'$  in $\partial S_A$ are mapped to distinct boundary components of $S_{A'}$, and all these boundary component of $S_{A'}$ are homeomorphic to $S^1$. Then the convex core $\Sigma_{A'}\subset S_{A'}$ contains all compact boundary components of $S_{A'}$. Now we apply LERFness of $\pi_1(S)$ to $A'<\pi_1(S)$ to get an intermediate finite cover $S'\to S$ of $S_{A'}\to S$ such that $\Sigma_{A'}$ embeds into $S'$. Then distinct components of $\mathcal{C}\cup \mathcal{L}'$ are mapped to distinct boundary components of $S'$. Since the image of $L_{n+1}$ is a boundary component of $S'$ that does not lie in the image of $\mathcal{C}\cup \mathcal{L}$, condition (2) and condition (3) hold.

Note that conditions (2) and (3) still hold when taking any intermediate finite cover of $S_A\to S'$. Then we apply Theorem 1.3 of \cite{BC} to get an intermediate finite cover $\hat{S}\to S'$ of $S_A\to S'$ such that $\hat{S}\setminus \Sigma_A$ is connected, which satisfies the second half of condition (1).
\end{proof}

Let $S$ be a compact hyperbolic surface, we say that a compact connected subsurface $\Sigma\subset S$ is {\it nice} if each boundary component of $\Sigma$ either equals a boundary component of $S$ or lies in the interior of $S$, and no component of $S\setminus \Sigma$ is homeomorphic to the disc or annulus.

In particular, suppose that $S$ is a compact hyperbolic surface, and $\Sigma$ is the convex core of an infinite cover of $S$ with finitely generated fundamental group. If the covering map embeds $\Sigma$ into $S$, then $\Sigma$ is a nice subsurface of $S$.

\begin{lemma}\label{orbifoldcover}
Let $S$ be a compact orientable surface with positive genus and nontrivial boundary, and $\Sigma\subset S$ be a nice subsurface. We suppose that $\partial S\setminus \partial \Sigma\ne \emptyset $ and let $\{c_1,\cdots,c_n\}$ be a collection of components of $\partial S\setminus \partial \Sigma$. Then for any integers $E_1,\cdots,E_n\geq 2$, there exists a finite cover $\hat{S}\to S$, such that $\Sigma$ lifts into $\hat{S}$ and each elevation of $c_k$ in $\hat{S}$ covers $c_k$ by degree $E_k$.
\end{lemma}

\begin{proof}
  For each boundary component $c_k \subset \partial S$ in the collection, we take a disc with one singular point of cone angle $\frac{2\pi}{E_k}$, and attach it to $S$ along $c_k$. We denote the resulting $2$-orbifold by $\mathcal{O}$, and consider $S$ as a subset of $\mathcal{O}$. Since $\Sigma \subset S$ is a nice subsurface, the inclusion $\Sigma\to \mathcal{O}$ is $\pi_1$-injective. This is because that, no component of $\mathcal{O}\setminus \Sigma$ is a disc with at most one singular point.

  Since the underlying space of $\mathcal{O}$ is a surface with positive genus, $\pi_1(\mathcal{O})$ is a virtually free group or a virtually surface group. Then Proposition 7.1 of \cite{KRS} implies that $\pi_1(\mathcal{O})$ has a torsion-free finite index subgroup $K$ that contains $\pi_1(\Sigma)$. Although Proposition 7.1 of \cite{KRS} assumes that $\Lambda$ (which is $\pi_1(\mathcal{O})$ in our case) is an arithmetic lattice of simplest type, its proof only uses that $\Lambda$ is good and LERF, and these properties hold for virtually free and virtually surface groups.

  The torsion-free finite index subgroup $K<\pi_1(\mathcal{O})$ corresponds to a finite-sheet orbifold cover $\hat{\mathcal{O}}\to \mathcal{O}$, such that $\hat{\mathcal{O}}$ is a surface (with no singular points) and $\Sigma$ lifts to $\hat{\mathcal{O}}$. Let $\hat{S}$ be the preimage of $S$ in $\hat{\mathcal{O}}$, then for each boundary component $c_k$ of $S$ in the collection $\{c_1,\cdots,c_n\}$, any elevation of $c_k$ in $\hat{S}$ covers $c_k$ by degree $E_k$.

The construction is summarized by the following diagram:
\begin{diagram}
& & \hat{S} &\subset & \hat{\mathcal{O}}\\
&\ruInto &\dTo & & \dTo \\
\Sigma &\subset & S &\subset & \mathcal{O}.
\end{diagram}

\end{proof}

Now we work on the partially fibered case. In this case, we take slopes $t_j\subset \pi(C_j)$ and $v_k\subset \pi(P_k)$ to be the regular fiber of $N$.

\begin{proposition}\label{Seifertpartialfiber}
Let $N$ be a Seifert $3$-manifold whose base orbifold is an orientable surface with positive genus and nontrivial boundary, $L<\pi_1(N)$ be a partially fibered subgroup, and $K\subset N_L$ be any compact subset. Let $\mathcal{P}=\{P_1,\cdots,P_n\}$ be a finite collection of plane boundary components of $\partial N_L$. Then there exists $A\in \mathbb{Z}_+$, such that for any $\alpha, \beta_1\cdots,\beta_n\in \mathbb{Z}_+$, there exists an intermediate finite cover $q:\hat{N}\to N$ of $\pi:N_L\to N$, such that the following conditions hold for commutative diagram:
\begin{diagram}
N_L &\rTo^p & \hat{N}\\
&\rdTo_{\pi} &\dTo_{q}\\
& &N.
\end{diagram}
\begin{enumerate}
  \item $p|_{K}:K\to \hat{N}$ is an embedding.
  \item For any two distinct boundary components of $N_L$ such that each of them is either a cylinder or a plane in $\mathcal{P}$, $p$ maps them to distinct boundary components of $\hat{N}$.
  \item For each cylinder $C_j\subset \partial N_L$, $p$ maps its core $c_j$ to a simple closed curve in $p(C_j)$, and $q|_{p(C_j)}:p(C_j)\to \pi(C_j)$ is the finite cover corresponding to $\mathbb{Z}[c_j]\oplus \mathbb{Z}[\alpha At_j]<\pi_1(\pi(C_j))$.
  \item For each plane $P_k\in \mathcal{P}$, $q|_{p(P_k)}:p(P_k)\to \pi(P_k)$ is the finite cover corresponding to $\mathbb{Z}[\beta_kAu_k]\oplus\mathbb{Z}[\alpha A v_k]<\pi_1(\pi(P_k))$.
\end{enumerate}
\end{proposition}

\begin{proof}

  At first, we enlarge $K\subset N_L$ so that the partially fibered subsurface $\Sigma_L \subset N_L$ is contained in $K$. Since $\pi_1(N)$ is LERF, there exists an intermediate finite cover of $N_L\to N$ such that $K$ embeds into this finite cover. By abusing notation, we still use $N$ to denote this finite cover and it satisfies condition (1).

  Let $\mathcal{O}_{N_L}$ be the base orbifold of $N_L$, then distinct boundary components of $N_L$ project to distinct boundary components of $\mathcal{O}_{N_L}$. Let $\mathcal{P}'=\{P'_1,\cdots,P'_n\}$ be the line boundary components of $\mathcal{O}_{N_L}$ corresponding to $\mathcal{P}$, and $\mathcal{C}'$ be the finite collection of circle boundary components of $\mathcal{O}_{N_L}$. By Lemma \ref{infinitecoversurface}, there is an intermediate finite cover $\mathcal{O}'\to \mathcal{O}_N$ of $\mathcal{O}_{N_L}\to \mathcal{O}_N$, such that the following holds.
\begin{enumerate}
\item $\Sigma_L$ embeds into $\mathcal{O}'$ and $\mathcal{O}'\setminus \Sigma_L$ is connected.
\item Any two distinct components of $\mathcal{C}'\cup \mathcal{P}'$ are mapped to distinct boundary components of $\mathcal{O}'$.
\item $\mathcal{O}'$ has a boundary component that is not in the image of $\mathcal{C}'\cup \mathcal{P}'$.
\end{enumerate}

  We take the pull-back bundle of $N\to \mathcal{O}_N$ via covering map $\mathcal{O}'\to \mathcal{O}_N$ and denote the resulting manifold by $N'$. Then $N'$ satisfies condition (2) in this proposition, and so does any further intermediate finite cover of $N_L\to N'$. Moreover, each core $c_j\subset C_j$ is mapped to a simple closed curve in a boundary component of $N'$.

  \bigskip

  Since $K$ contains $\Sigma_L$ and $K$ embeds into $N'$, we can consider $\Sigma_L$ as a subsurface of $N'$. Let $r:N'\to \mathcal{O}'$ be the projection to the base orbifold of $N'$, then $r|_{\Sigma_L}:\Sigma_L\to \mathcal{O}'$ is an embedding. We decompose $N'$ as $N'=N'_+\cup N'_-$, by taking $N'_+=r^{-1}(r(\Sigma_L))$ and $N'_-=\overline{N'\setminus N'_+}$. Then both $N'_+$ and $N'_-$ have induced Seifert structures. Since $\Sigma_L$ is a section of $S^1$-bundle $N'_+$, $N'_+$ is homeomorphic to $\Sigma_L\times S^1$. Since $\mathcal{O}'\setminus \Sigma_L$ is connected, $N'_-=r^{-1}(\overline{\mathcal{O}'\setminus \Sigma_L})$ is connected.

  For each component $T'$ of $\partial N'_-\cap \partial N'_+$, it is equipped with a slope $\Sigma_L\cap T'$. For each component $T'_k\subset \partial N'_-$ that is the image of $P_k\in \mathcal{P}$, it is equipped with a slope $u_k'$, which is an elevation of $u_k\subset \pi(P_k)\subset \partial N$ in $N'$. The above condition (3) on $\mathcal{O}'$ implies that $N'_-$ has a boundary component that lies neither in $\partial N'_+$ nor in the image of $\mathcal{P}$. Since all these slopes on $\partial N'_-$ are not parallel to the regular fiber of $N'_-$, by the theory of horizontal subsurfaces in Seifert spaces, there exists a connected properly embedded horizontal subsurface $\Sigma_-\subset N'_-$, such that its intersection with each component of $(\partial N'_+\cap \partial N'_-)\cup (\cup_{k=1}^nT'_k)$ consists of parallel copies of the above chosen slope. Let $i$ be the algebraic intersection number between $\Sigma_-$ and the regular fiber of $N'_-$. Then for each component $T'\subset \partial N'_+\cap \partial N'_-$, $\Sigma_-\cap T'$ consists of $i$ parallel copies of $\Sigma_L\cap T'$. Then we paste $\Sigma_-$ and $i$ parallel copies of $\Sigma_L$ in $N'_+$ to get a properly embedded horizontal subsurface $\Sigma'\subset N'$, such that each $u_k'$ is parallel to a boundary component of $\Sigma'$.

    So $N'$ has a $\Sigma'$-bundle over $S^1$ structure, i.e. $N'=\Sigma'\times I/\phi$, such that the period of $\phi$ is $i$. For $N''=\Sigma'\times I/\phi^i$, it is homeomorphic to $\Sigma'\times S^1$, and $\Sigma'$ intersects with its regular fiber exactly once. Since $\Sigma_L$ is a subsurface of $\Sigma'$, $N''\to N$ is an intermediate finite cover of $N_L\to N$, and we denote the covering map from $N_L$ to $N''$ by $s:N_L\to N''$. By abusing notation, we use $u_k'$ to denote the boundary component of $\Sigma'$ that lies in the image of $P_k\subset \partial N_L$.

    Let $A_0$ be the induced covering degree of $N''\to N$ on the regular fiber. For each slope $u_k\subset \pi(P_k)$, let $B_k$ be the covering degree of $u_k'\to u_k$. Let $A'$ be the least common multiple of $A_0, B_1,\cdots,B_n$, and let $A=2A'$. For any positive integers $\alpha, \beta_1,\cdots,\beta_n\in \mathbb{Z}_+$, we will construct a finite cover $\hat{N}\to N$ satisfying all desired conditions.

    For each boundary component $u_k'\subset \partial \Sigma'$ corresponding to $P_k\subset \partial N_L$, we associate it with an integer $\beta_k\frac{A}{B_k}\geq 2$. By Lemma \ref{orbifoldcover}, there exists a finite cover $s_1:\hat{\Sigma}\to \Sigma'$ such that $\Sigma_L$ lifts into $\hat{\Sigma}$ and each elevation of $u_k'$ in $\hat{\Sigma}$ has covering degree $\beta_k\frac{A}{B_k}$ to $u_k'$.

  Let $s_2:S^1\to S^1$ be the covering map of degree $\alpha\frac{A}{A_0}$. Then the covering space $\hat{s}=s_1\times s_2:\hat{N}=\hat{\Sigma}\times S^1\to N''=\Sigma'\times S^1$ is an intermediate cover of $N_L\to N''$. Since $\Sigma'$ intersects with the regular fiber of $N''$ exactly once, for each cylinder $C_j\subset \partial N_L$, $s(C_j)\subset \partial N''$ is the finite cover of $\pi(C_j)\subset \partial N$ corresponding to $\mathbb{Z}[c_j]\oplus \mathbb{Z}[A_0t_j]<\pi_1(\pi(C_j))$. For each plane $P_k\subset \partial N_L$, $s(P_k)\subset \partial N''$ is the finite cover of $\pi(P_k)\subset \partial N$ corresponding to $\mathbb{Z}[B_ku_k]\oplus \mathbb{Z}[A_0v_k]<\pi_1(\pi(P_k))$. Since the covering map $\hat{N}\to N''$ preserves the product structure, it is easy to check that $\hat{N}$ satisfies all desired conditions.

The constructions in this proof is summarized by the following diagram.
\begin{diagram}
                & & & &\hat{\Sigma}\times \{1\} & & \subset & &  \hat{\Sigma} \times S^1 & \rTo & \Sigma'\times S^1   & \rTo &  \Sigma'\times  I/\phi & &\\
                & &\ruInto(3,2)  & &  & & &  \ldTo_{\cong}       &         &                 & \dTo_{\cong} &    & \dTo_{\cong}& &\\
\Sigma_L & \rInto & N_L & \rTo(3,1) & & & \hat{N} & & \rTo(3,1)^{\hat{s}} & & N'' & \rTo & N' & \rTo &N \\
               & \rdInto& \dTo &      &              &          & &   & &            &        &        & \dTo_{r}& & \dTo \\
               &           &  \mathcal{O}_{N_L} & \rTo(7,1) & & & & & & & & & \mathcal{O}' & \rTo & \mathcal{O}_N
\end{diagram}
\end{proof}

It remains to deal with the case that the induced bundle structure on $N_L$ is an $S^1$-bundle over a noncompact $2$-orbifold. In this case, for each cylinder boundary component $C_j\subset \partial N_L$, its core $c_j$ is mapped to a power of the regular fiber of $N$. Let the regular fibered of $N$ be denoted by $h$, then the image of $c_j$ equals $xh$ for some $x\in \mathbb{Z}_+$, which is independent of boundary component $C_j$.

\begin{proposition}\label{S1bundle}
Let $N$ be a Seifert $3$-manifold whose base orbifold is an orientable surface with positive genus and nontrivial boundary, $L<\pi_1(N)$ be a finitely generated infinitely index subgroup that contains a power of the regular fiber, and $K\subset N_L$ be any compact subset. Let $\mathcal{C}=\{C_1,\cdots,C_m\}$ be a finite collection of cylinder boundary components of $N_L$. Then there exists $A\in \mathbb{Z}_+$, such that for any $\alpha_1\cdots,\alpha_m\in \mathbb{Z}_+$, there exists an intermediate finite cover $q:\hat{N}\to N$ of $\pi:N_L\to N$, such that the following conditions hold for commutative diagram:\begin{diagram}
N_L &\rTo^p & \hat{N}\\
&\rdTo_{\pi} &\dTo_{q}\\
& &N.
\end{diagram}
\begin{enumerate}
  \item $p|_{K}:K\to \hat{N}$ is an embedding.
  \item For any two distinct boundary components of $N_L$ such that each of them is either a torus or a cylinder in $\mathcal{C}$, $p$ maps them to distinct boundary components of $\hat{N}$.
  \item For each torus boundary component $T_i\subset \partial N_L$, $p|_{T_i}:T_i\to p(T_i)$ is a homeomorphism.
  \item For each cylinder $C_j\in \mathcal{C}$, $p$ maps $c_j$ to a simple closed curve in $p(C_j)$, and $q|_{p(C_j)}:p(C_j)\to \pi(C_j)$ is the finite cover corresponding to $\mathbb{Z}[c_j]\oplus \mathbb{Z}[\alpha_j At_j]<\pi_1(\pi(C_j))$.
\end{enumerate}
\end{proposition}

\begin{proof}

Since $\pi_1(N)$ is LERF, there exists an intermediate finite cover $N'\to N$ of $N_L\to N$, such that $K$ embeds into $N'$. By LERFness again, we can assume that the restriction of $N_L\to N'$ on regular fibers has degree $1$,  and each torus boundary component of $N_L$ is mapped to $N'$ by a homeomorphism. This $N'$ satisfies condition (1), (3) and the first half of (4) in this proposition. We abuse notation and still denote this manifold by $N$.

Let $\mathcal{O}_N$ be the base orbifold of $N$, and $\mathcal{O}_{N_L}$ be the base orbifold of $N_L$. Then we apply Lemma \ref{infinitecoversurface} to get an intermediate finite cover $\mathcal{O}'\to \mathcal{O}_N$ of $\mathcal{O}_{N_L}\to \mathcal{O}_N$, such that the following holds.
\begin{enumerate}
\item $\Sigma_L\subset \mathcal{O}_{N_L}$ embeds into $\mathcal{O}'$.
\item All circle boundary components of $\mathcal{O}_{N_L}$ and all of its line boundary components corresponding to $\mathcal{C}$ are mapped to distinct boundary components of $\mathcal{O}'$.
\end{enumerate}
We take the pull-back bundle of $N\to \mathcal{O}_N$ via covering map $\mathcal{O}'\to \mathcal{O}_N$, and denote it by $N'$. Since $\Sigma_L$ embeds into $\mathcal{O'}$, $N'\to N$ is an intermediate finite cover of $N_L\to N$. It is easy to see that $N'\to N$ satisfies the first three conditions in this proposition.

Let $s:N_L\to N'$ be the covering map. Then for any two distinct $C_i,C_j\in \mathcal{C}$, $s(C_i)$ and $s(C_j)$ are distinct boundary components of $N'$. Let $s_j$ be the boundary component of $\mathcal{O}'$ corresponding to $s(C_j)$, and let $t'_j$ be an elevation of $t_j\subset \pi(C_j)\subset \partial N$ in $s(C_j)\subset \partial N'$.

Let $A_j$ be the covering degree of $t'_j\to t_j$, and let $\gamma_j$ be the algebraic intersection number between $t'_j$ and the regular fiber of $N'$. Let $A'$ be the least common multiple of $A_1,\cdots,A_m$, and let $A=2A'$. For any positive integers $\alpha_1,\cdots,\alpha_m\in \mathbb{Z}_+$, and any component $s_j\subset \partial \mathcal{O}'$ corresponding to $s(C_j)\subset \partial N'$, we associate $s_j$ with a positive integer $\alpha_j\gamma_j\frac{A}{A_j}\geq 2$.

Then Lemma \ref{orbifoldcover} provides us an intermediate finite cover $\hat{\mathcal{O}}\to \mathcal{O}'$ of $\mathcal{O}_{N_L}\to \mathcal{O}'$, such that $\Sigma_L$ lifts into $\hat{\mathcal{O}}$ and each elevation of $s_j\subset \partial \mathcal{O}'$ in $\hat{\mathcal{O}}$ has covering degree $\alpha_j\gamma_j\frac{A}{A_j}$ to $s_j$. Let $\hat{N}$ be the pull-back bundle of $N'\to \mathcal{O}'$ via covering map $\hat{\mathcal{O}}\to \mathcal{O}'$, and we use $p:N_L\to \hat{N}$ to denote the covering map.

For each $j$, $p(C_j)\subset \partial \hat{N}$ is a covering space of $s(C_j)\subset \partial N'$, such that its restriction on regular fiber has degree $1$ and the covering degree on base space is $\alpha_j\gamma_j\frac{A}{A_j}$. Since the algebraic intersection number between $t'_j$ and the regular fiber in $s(C_j)\subset \partial N'$ is $\gamma_j$, and the covering degree on base space is a multiple of $\gamma_j$, each elevation of $t'_j$ in $p(C_j)\subset \partial \hat{N}$ (denoted by $\hat{t}_j$) has algebraic intersection number $1$ with the regular fiber of $\hat{N}$, and the covering degree from $\hat{t}_j$ to $t'_j$ is $\alpha_j\frac{A}{A_j}$.

Since $\hat{t}_j$ has algebraic intersection number $1$ with the regular fiber of $\hat{N}$, the fundamental group of $p(C_j)\subset \partial \hat{N}$ is generated by the regular fiber and $\hat{t}_j$. Since the covering degree from $t'_j$ to $t_j$ is $A_j$, the covering degree from $\hat{t}_j$ to $t_j$ is $A_j\cdot (\alpha_j\frac{A}{A_j})=\alpha_jA$. So condition (4) holds for $\hat{N}$.

The constructions in this proof is summarized by the following diagram:
\begin{diagram}
\Sigma_L & \rInto & N_L                            & \rTo^p        & \hat{N} & \rTo & N'                   & \rTo &N \\
               & \rdInto& \dTo                          &             &        \dTo      &         & \dTo       &         & \dTo \\
               &            &  \mathcal{O}_{N_L} & \rTo         & \hat{\mathcal{O}}     & \rTo   & \mathcal{O}' & \rTo  & \mathcal{O}_N
\end{diagram}

\end{proof}

For the readers' convenience, we summary the results in Proposition \ref{hypfiber}, \ref{geofinite}, \ref{Seifertfiber}, \ref{Seifertpartialfiber} and \ref{S1bundle} in Table 3.

In Table 3, $L<\pi_1(N)$ is a finitely generated infinite index subgroup, $\pi:N_L\to N$ is the covering space corresponding to $L$, $C\subset \partial N_L$ is a cylinder boundary component and $P\subset \partial N_L$ is a plane boundary component. The second (or third) column of the table should be read as: for any compact subset $K\subset N_L$, there exists a positive integer $A$ (or $A_C$) such that for any positive integer $\alpha$ (or $\alpha_C,\beta_P, \delta_P$), there exists an intermediate finite cover $\hat{N}\to N$ of $N_L\to N$ with covering map $p:N_L\to \hat{N}$, such that $K$ is embedded into $\hat{N}$, and $p(C)$ (or $p(P)$) is the finite cover of $\pi(C)$ (or $\pi(P)$) corresponding to the subgroup given in the table. In the table, when we write $A$ or $\alpha$, we mean it is a positive integer only depends on $N$; and when we write $A_C$ or $\alpha_C$ ($\beta_P,\delta_P$), we mean it is a positive integer that also depends on the boundary component $C$ (or $P$). In the following, we will call $A$ and $A_C$ "constants" of $(N,L,K)$, and call $\alpha$, $\alpha_C, \beta_P,\delta_P$ "parameters" of $(N,L,K)$.

\begin{table}\label{summary}
\begin{center}
\begin{tabular}{ |c |c |c |}
\hline
   \multirow{3}{*} {\backslashbox{$L<\pi_1(N)$}{$\partial$ comp. of $N_L$}}& cylinder $C$ with core $c$ & plane $P$ with two  \\
   & and slope $t\subset \pi(C)$ & nonparallel slopes\\
   & not parallel to $\pi(c)$ &   $u,v\subset \pi(P)$ \\
  \hline
 \multirow{2}{*} {$L$ is virtually fibered} & $\exists A_C$ such that for $\forall \alpha$ &  \multirow{2}{*} {$\times$}\\
  & $\mathbb{Z}[c]\oplus \mathbb{Z}[\alpha A_C t]$ & \\
 \hline
 $N$ is hyperbolic and  &$\exists A$ such that for $\forall \alpha_C$ & $\exists A$ such that for $\forall \beta_P,\delta_P$\\
 $L$ is geom. finite & $\mathbb{Z}[c]\oplus \mathbb{Z}[\alpha_CA t]$& $\mathbb{Z}[\beta_PA u]\oplus \mathbb{Z}[\delta_PA v]$\\
 \hline
 $N$ is Seifert and  & $t$ is the regular fiber & $v$ is the regular fiber\\
 $L$ is partially fibered & $\exists A$ such that for $\forall \alpha$ & $\exists A$ such that for $\forall \beta_P,\alpha$\\
  & $\mathbb{Z}[c]\oplus \mathbb{Z}[\alpha A t]$& $\mathbb{Z}[\beta_PA u]\oplus \mathbb{Z}[\alpha A v]$\\
  \hline
  $N$ is Seifert and  & $\exists A$ such that for $\forall \alpha_C$ & \multirow{2}{*} {$\times$} \\
  $N_L$ is an $S^1$-bundle & $\mathbb{Z}[c]\oplus \mathbb{Z}[\alpha_CA t]$& \\
  \hline
\end{tabular}
\caption{Summary on constructions of finite covers in various cases.
In the fourth row, the two $\alpha$'s in the second and third columns are same with each other.}
\end{center}
\end{table}

\section{Aspirality implies separability}\label{combinatorics}

For the proof of "aspirality implies separability" part in Theorem \ref{main}, we first assume the $3$-manifold $M$ has empty or tori boundary, then the general case follows from this special case by a filling argument in Section \ref{filling}.

Let $M$ be a compact, orientable, irreducible $3$-manifold with empty or tori boundary that does not support a geometric structure, and $H<\pi_1(M)$ be a finitely generated infinite index subgroup. Let $\pi:M_H\to M$ be the covering space of $M$ corresponding to $H<\pi_1(M)$, and $K\subset M_H$ be an arbitrary compact subset. If $H$ is aspiral in $\pi_1(M)$, to prove $H$ is separable in $\pi_1(M)$, it suffices to prove that $K$ embeds in an intermediate finite cover of $M_H\to M$.

Let $M_K$ be the union of all pieces of $M_H$ that intersect with $K$, and we denote its dual graph by $G_K$. Then our strategy to prove the separability of $H$ consists of following steps.
\begin{enumerate}
  \item Enlarge $K$ such that $K$ is path-connected and contains a Scott core of $M_H$.
  \item Construct a finite semi-cover $N\to M$ such that $K$ embeds into $N$ and the dual graph of $N$ is isomorphic to $G_K$.
  \item Take an intermediate finite cover $M'\to M$ of $M_H\to M$, such that $N$ is a submanifold of $M'$ and the restriction of the covering map $M'\to M$ on $N$ is the semi-covering map in (2).
\end{enumerate}

The first step is obvious, and the third step follows from that subgroups of finite semi-covers are separable (Lemma \ref{semicover}). So we will focus on the second step. For each piece $V_K\subset M_K$, we denote its image in $M$ by $V$, which is a piece of $M$. Then the propositions in Section \ref{casebycase} construct many intermediate finite covers $\hat{V}\to V$ of $V_K\to V$ such that $K\cap V_K$ is mapped into $\hat{V}$ via embedding. For each $V_K$, we need to choose one of its intermediate finite cover, such that these finite covers of pieces of $M$ are compatible with each other on the boundary. Then we can paste them together to get the desired finite semi-cover $N\to M$, such that $K$ embeds into $N$.

\subsection{All possible cases of an edge space in $M_K$ and its two adjacent pieces}

For each edge space $E_K\subset M_K$, there are two vertex spaces $V_K,V_K'\subset M_K$ adjacent to $E_K$. The edge space $E_K$ can be a torus, a cylinder, or a plane. Let $\pi:M_H\to M$ be the covering map, then $\pi(E_K)$ is a decomposition torus in $M$.

Now we consider the pair $(V_K,E_K)$. When $E_K=C$ is a cylinder, in Proposition \ref{hypfiber}, \ref{geofinite}, \ref{Seifertfiber}, \ref{Seifertpartialfiber} and \ref{S1bundle}, we take a slope $t\subset \pi(E_K)$ not parallel to $\pi(c)$ (here $c$ is the core of $C$). In Proposition \ref{Seifertpartialfiber}, we require that $t$ is the regular fiber of $\pi(V_K)=V\subset M$, otherwise $t$ can be any slope not parallel to $\pi(c)$. When $E_K=P$ is a plane, in Proposition \ref{geofinite} and \ref{Seifertpartialfiber}, we take two nonparallel slopes $u,v\subset \pi(E_K)$. In Proposition \ref{Seifertpartialfiber}, we require that $v$ is the regular fiber of $\pi(V_K)$ and $u$ can be any slope not parallel to $v$; otherwise $u,v$ can be any two non-parallel slopes.

When we prove "aspirality implies separability", for each vertex piece $V_K\subset M_K$ and its image $V\subset M$, we first need to choose slopes on $\partial V$, so that we can apply the propositions in Section \ref{casebycase} to construct an intermediate finite cover $\hat{V}\to V$ of $V_K\to V$. In this construction, we need to use the information of $V'_K$ to choose slopes on $\pi(E_K)$: one slope $t\subset \pi(E_K)$ when $E_K$ is a cylinder, and two nonparallel slopes $u,v\subset \pi(E_K)$ when $E_K$ is a plane.

In Table 4 and Table 5, we specify our choice of these slopes $t$ and $u,v$ on $\pi(E_K)\subset \pi(V_K)$ (from $V_K$ viewpoint). In priori, we may choose different slopes on $\pi(E_K)\subset \pi(V_K')$ (from $V_K'$ viewpoint). Table 4 deals with the case that $E_K$ is a cylinder, and Table 5 deals with the case that $E_K$ is a plane.

\begin{table}\label{slopeoncylinder}
\begin{center}
\begin{tabular}{ |c |c |c | c| c| c|}
\hline
   \multirow{2}{*} {\backslashbox{$V_K$}{$V'_K$}}& hyperbolic  & hyperbolic & Seifert  & Seifert  & Seifert \\
    & virtual fiber & geom. finite & virtual fiber & partial fiber & $S^1$ bundle \\
  \hline
 hyperbolic &  degeneracy  &  degeneracy & degeneracy & degeneracy& degeneracy \\
 virtual fiber & slope of $\pi(V_K)$ & slope of $\pi(V_K)$ & slope of $\pi(V_K)$& slope of $\pi(V_K)$ & slope of $\pi(V_K)$ \\
 \hline
 hyperbolic  & degeneracy & \multirow{2}{*} {any slope}& regular fiber& regular fiber&\multirow{2}{*} {any slope}\\
 geom. finite & slope of $\pi(V'_K)$& & of $\pi(V'_K)$& of $\pi(V'_K)$ &\\
 \hline
 Seifert  & regular fiber& regular fiber& regular fiber& regular fiber& regular fiber\\
 virtual fiber & of $\pi(V_K)$& of $\pi(V_K)$& of $\pi(V_K)$& of $\pi(V_K)$& of $\pi(V_K)$\\
   \hline
  Seifert   & regular fiber& regular fiber& regular fiber& regular fiber& regular fiber\\
  partial fiber & of $\pi(V_K)$& of $\pi(V_K)$& of $\pi(V_K)$& of $\pi(V_K)$& of $\pi(V_K)$\\
  \hline
  Seifert  & degeneracy& \multirow{2}{*} {any slope}& regular fiber& regular fiber&\multirow{2}{*} {$\times$}\\
  $S^1$ bundle & slope of $\pi(V'_K)$& & of $\pi(V'_K)$& of $\pi(V'_K)$&\\
  \hline
\end{tabular}
\caption{When $E_K$ is a cylinder and $c$ is the core of $E_K$, the choice of a slope $t\subset \pi(E_K)\subset \pi(V_K)$ not parallel to $\pi(c)$ (from $V_K$ viewpoint).}
\end{center}
\end{table}

\begin{table}\label{slopesonplane}
\begin{center}
\begin{tabular}{ |c |c |c |}
\hline
   \multirow{2}{*} {\backslashbox{$V_K$}{$V'_K$}}&  hyperbolic & Seifert  \\
    &  geom. finite &  partial fiber \\
 \hline
 hyperbolic  & any two& $u$ is regular fiber of $\pi(V'_K)$,\\
 geom. finite & nonparallel slopes&  $v$ is any other slope\\
   \hline
  Seifert   & $v$ is regular fiber of $\pi(V_K)$,& $u$ is regular fiber of $\pi(V'_K)$,\\
  partial fiber & $u$ is any other slope& $v$ is regular fiber of $\pi(V_K)$\\
  \hline
\end{tabular}
\caption{When $E_K$ is a plane, the choice of two nonparallel slopes $u,v\subset \pi(E_K)\subset \pi(V_K)$ (from $V_K$ viewpoint).}
\end{center}
\end{table}

In some cases in Table 4, we take the degeneracy slope for hyperbolic virtually fibered subgroups, which is defined in \cite{GK}. For any (isotopy class of) pseudo-Anosov map, we take a representative $\phi:S\to S$ in this isotopy class such that $\phi$ preserves a pair of measured singular foliations on $S$. For the mapping torus $M_{\phi}=S\times I/\phi$, the suspensions of the above singular foliations give a slope on each boundary component $T\subset \partial M_{\phi}$, and it is the {\it degeneracy slope} of $M_{\phi}$ on $T$.

In the case that $E_K$ is a cylinder, when we switch the role of $V_K$ and $V'_K$, Table 4 gives us another slope $t'\subset \pi(E_K)\subset \pi(V_K')$. By checking Table 4, if both $V_K$ and $V'_K$ are "hyperbolic virtually fibered", "Seifert virtually fibered" or "Seifert partially fibered", it is possible that $t\ne t'$. Otherwise, we can and will choose these slopes such that $t=t'$.

In the case that $E_K$ is a plane, when we switch the role of $V_K$ and $V'_K$, Table 5 gives us another pair of slopes $u',v'\subset \pi(E_K)\subset \pi(V_K')$. By checking Table 5, we can and will choose these two slopes such that $u=v'$ and $v=u'$ in all cases.

So only in the case that $E_K$ is a cylinder and both $V_K,V'_K$ are "hyperbolic virtually fibered", "Seifert virtually fibered" or "Seifert partially fibered", the two choices of slopes on $\pi(E_K)$ are not compatible with each other. The following simple lemma takes care of this non-compatible issue.

\begin{lemma}\label{Z2}
Let $c,t,t'\in \mathbb{Z}^2$ be three nontrivial elements, such that they generate finite index subgroups $\mathbb{Z}[c]\oplus \mathbb{Z}[t],\mathbb{Z}[c]\oplus \mathbb{Z}[t']<\mathbb{Z}^2$. Then there exist $B,B'\in \mathbb{Z}_+$, such that for any $\alpha\in \mathbb{Z}_+$, $\mathbb{Z}[c]\oplus \mathbb{Z}[\alpha Bt]$ and $\mathbb{Z}[c]\oplus \mathbb{Z}[\alpha B't']$ are the same subgroup of $\mathbb{Z}^2$.
\end{lemma}

\begin{proof}
  Up to an automorphism of $\mathbb{Z}^2$, we can assume that $c=(x,0)\in \mathbb{Z}^2$ for some $x\in \mathbb{Z}_+$. Then we have $t=(y,z)$ and $t'=(y',z')$ for $y,y'\in \mathbb{Z}$ and $z,z'\in \mathbb{Z}\setminus \{0\}$.

  We take $B=x|z'|$ and $B'=x|z|$. Then we have
  $$\mathbb{Z}[c]\oplus \mathbb{Z}[\alpha Bt]=\mathbb{Z}[(x,0)]\oplus \mathbb{Z}[(\alpha x|z'|y,\alpha x|z'|z)]=\mathbb{Z}[(x,0)]\oplus \mathbb{Z}[(0,\alpha x|z'|z)]$$
  and $$\mathbb{Z}[c]\oplus \mathbb{Z}[\alpha B't']=\mathbb{Z}[(x,0)]\oplus\mathbb{Z}[(\alpha x|z|y',\alpha x|z|z')]=\mathbb{Z}[(x,0)]\oplus \mathbb{Z}[(0,\alpha x|z|z')].$$
  So $\mathbb{Z}[c]\oplus \mathbb{Z}[\alpha Bt]$ and $\mathbb{Z}[c]\oplus \mathbb{Z}[\alpha B't']$ are the same subgroup of $\mathbb{Z}^2$.
\end{proof}

\subsection{Proof of "aspirality implies separability" for $3$-manifolds with tori or empty boundary}

In this section, we prove the "aspirality implies separability" part of Theorem \ref{main}, in the case that $M$ is a compact, orientable, irreducible $3$-manifold with empty or tori boundary.

\begin{proposition}\label{toriboundaryseparability}
  Let $M$ be a compact, orientable, irreducible $3$-manifold with empty or tori boundary, such that it has nontrivial torus decomposition and does not support the $\text{Sol}$-geometry. Let $H<\pi_1(M)$ be a finitely generated infinite index subgroup. If $H$ is aspiral, then $H$ is separable in $\pi_1(M)$.
\end{proposition}

To prove Proposition \ref{toriboundaryseparability}, we need the following technical result.

\begin{proposition}\label{embedsemicover}
  Let $M$ be a compact, orientable, irreducible $3$-manifold with empty or tori boundary, such that it has nontrivial torus decomposition and does not support the $\text{Sol}$-geometry. Let $H<\pi_1(M)$ be an aspiral finitely generated infinite index subgroup, and $\pi:M_H\to M$ be the covering space of $M$ corresponding to $H<\pi_1(M)$. We further suppose the base orbifolds of all Seifert pieces of $M$ are orientable surfaces (with no singular points) with positive genus, and the almost fibered surface $\Phi(H)$ is orientable.

  Then for any compact subset $K\subset M_H$, there exist a finite semi-cover $f:N\to M$ and an embedding $i:K\hookrightarrow N$, such that $f\circ i=\pi|_{K}$, i.e. the following diagram commute.
  \begin{diagram}
K &\rTo^i & N\\
&\rdTo_{\pi|_K} &\dTo_{f}\\
& &M
\end{diagram}
\end{proposition}

The proof of Proposition \ref{embedsemicover} is quite long, so we break it into a few steps.

  At first, for the compact subset $K\subset M_H$, we enlarge $K$ such that the following hold:
  \begin{enumerate}
    \item $K$ is path-connected,
    \item $K$ intersects with all pieces of $M_H^c\subset M_H$ ($M_H^c$ is defined in Section \ref{almostfiber}),
    \item $K$ contains a Scott core of $M_H$ and the almost fibered surface $\Phi(H)\subset M_H$.
  \end{enumerate}
  We define $M_K$ to be the submanifold of $M_H$ that is the union of all vertex spaces $V_H\subset M_H$ such that $V_H\cap K\ne \emptyset$. Then $M_H^c$ is a submanifold of $M_K$, and $\Phi(H)$ is a subsurface of $M_K$. We denote the dual graph of $M_K$ by $G_K$. By the assumption on $K$, $G_K$ is a connected finite subgraph of $G_H$ and each component of $G_H\setminus G_K$ is a tree. For the manifold $N$ that we will construct in this proof, it has a graph of space structure with dual graph isomorphic to $G_K$.

\bigskip

\textbf{Step I. Construct a maximal spanning tree $T\subset G_K$ satisfying the edge condition.}

  The almost fibered surface $\Phi(H)$ is naturally a subsurface of $M_K$, and any piece of $M_K$ contains at most one piece of $\Phi(H)$. So the dual graph $G_{\Phi(H)}$ of $\Phi(H)$ is naturally a subgraph of $G_K$.

  Note that for each partially fibered piece $V_K\subset M_K$, it has no contribution to $\Phi(H)$ if and only if all boundary components of $V_K$ are planes. We start with a maximal spanning forest of $G_{\Phi(H)}$, and extend it to a maximal spanning tree $T\subset G_K$. For each edge $e$ in $G_K\setminus T$ with corresponding edge space $E^e\subset M_K$, one of the following \textbf{edge conditions} hold.
  \begin{enumerate}
    \item At least one of the adjacent vertex space of $E^e$ in $M_K$ is not a virtually fibered or partially fibered piece.
    \item Both adjacent vertex spaces of $E^e$ are virtually fibered or partially fibered pieces, and $E^e$ is homeomorphic to the plane.
    \item If both adjacent vertex spaces of $E^e$ are virtually fibered or partially fibered pieces and $E^e$ is homeomorphic to the cylinder, then $e$ must lie in a simple cycle $c\subset G_{\Phi(H)}\subset G_K$.
  \end{enumerate}

  We take $M_T$ to be the submanifold of $M_K$ dual with $T$, i.e. $M_T$ is obtained by cutting $M_K$ along edge spaces of $M_K$ corresponding to edges in $G_K\setminus T$. Let $n$ be the number of vertices in $T$, and let $T_1\subset \cdots\subset T_n=T$ be a sequence of increasing connected subtrees of $T$, such that $T_j$ has exactly $j$ vertices. Let $M_{T_j}$ be the submanifold of $M_T$ dual with $T_j$.

  \bigskip

  \textbf{Step II. Inductively construct a finite semi-cover $N_{T_j}\to M$ with dual graph isomorphic to $T_j$, such that $K\cap M_{T_j}$ embeds into $N_{T_j}$.}

  Since $G_K$ is a finite graph, for each vertex space $V_K\subset M_K$, there are only finitely many edge spaces of $M_K$ adjacent to $V_K$ (there might be infinitely many such edge spaces in $M_H$). If $V_K$ is not a finite cover of the corresponding vertex space of $M$, there are finitely many cylinder and plane edge spaces $E_1,\cdots,E_m\subset M_K$ adjacent to $V_K$.

For each vertex space $V_K\subset M_K$, we apply Proposition \ref{hypfiber}, \ref{geofinite}, \ref{Seifertfiber}, \ref{Seifertpartialfiber} or \ref{S1bundle} to $V_K, K\cap V_K$, its finitely many adjacent edge spaces $\{E_1,\cdots,E_m\}$ in $M_K$ and the chosen slopes on $\{\pi(E_1),\cdots,\pi(E_m)\}$, to get a positive integer $A_{V_K}$ (only depends on $V_K$), or a finite collection of positive integers $\{A_{(V_K,E_j)}\}_{j=1}^m$ (also depend on $E_j\subset \partial V_K$). For each cylinder edge space $E\subset M_K$ such that the two vertex spaces adjacent to it are both virtually fibered or partially fibered, the two choices of slopes on $\pi(E)$ may not be compatible, then Lemma \ref{Z2} gives us two positive integers $B_E$ and $B'_E$. We take $\mathfrak{A}$ to be the product of all positive integers obtained in the above.

Our construction in step II is given by the following lemma.

  \begin{lemma}\label{constructionontree}
  For each $j\in \{1,2,\cdots,n\}$, there exists a finite semi-cover $s_j:N_{T_j}\to M$, and a semi-cover $p_j:M_{T_j}\to N_{T_j}$ that induces isomorphism on dual graphs, such that following conditions hold for the following commutative diagram.
  \begin{enumerate}
  \item $s_j\circ p_j=\pi|_{M_{T_j}}$.
   \item $p_j|_{K\cap M_{T_j}}:K\cap M_{T_j}\to N_{T_j}$ is an embedding.
    \item For any two distinct boundary components of $M_{T_j}$ that lie in the interior of $M_K$,  $p_j$ maps them to distinct boundary components of $N_{T_j}$.
   \end{enumerate}

  \begin{diagram}
K\cap M_{T_j} &\subset & M_{T_j}\\
& \rdTo_{p_j|_{K\cap M_{T_j}}} & \dTo_{p_j}\\
 & & N_{T_j}\\
 & & \dTo_{s_j}\\
 & & M
\end{diagram}

  Moreover, for any boundary component $E\subset \partial M_{T_j}$ that lies in the interior of $M_K$, it also satisfies the following \textbf{boundary condition}.
  \begin{enumerate}
  \item If $E$ is a torus, $p_j|_{E}:E\to p_j(E)$ is a homeomorphism.
  \item If $E$ is a cylinder, $s_j|_{p_j(E)}:p_j(E)\to \pi(E)$ is the finite cover corresponding to $\mathbb{Z}[c]\oplus \mathbb{Z}[\tau\mathfrak{A}^{n-j+1}t]<\pi_1(\pi(E))$ for some $\tau\in \mathbb{Z}_+$ (Here $c$ is the core of $E$, and $t$ is a slope on $\pi(E)$ chosen by Table 4, from $M_{T_j}$ viewpoint).
  \item If $E$ is a plane, $s_j|_{p_j(E)}:p_j(E)\to \pi(E)$ is the finite cover corresponding to following subgroups of $\pi_1(\pi(E))$:
   \begin{enumerate}
     \item  $\mathbb{Z}[\mu\mathfrak{A}^{n-j+1}u]\oplus \mathbb{Z}[\nu\mathfrak{A}^{n-j+1}v]$ for some $\mu,\nu\in \mathbb{Z}_+$, if $E$ lies in the interior of $M_{T}$;
     \item $\mathbb{Z}[\mathfrak{A}u]\oplus \mathbb{Z}[\mathfrak{A}v]$, if $E$ does not lie in the interior of $M_{T}$ and its adjacent piece in $M_T$ is a geometrically finite piece;
     \item $\mathbb{Z}[\mathfrak{A}u]\oplus \mathbb{Z}[\nu\mathfrak{A}^{n-j+1}v]$ for some $\nu\in \mathbb{Z}_+$, if $E$ does not lie in the interior of $M_{T}$ and its adjacent piece in $M_T$ is a partially fibered piece.
   \end{enumerate} (Here $u,v$ are two nonparallel slopes on $\pi(E)$ chosen by Table 5, from $M_{T_j}$ viewpoint. In case (c), $v$ is the regular fiber in the vertex space of $M_{T_j}$ adjacent to $E$.)
  \end{enumerate}
\end{lemma}

\begin{proof}

  We prove this lemma by induction. The frame of this construction is shown in the following diagram.

\begin{diagram}
M_{T_n}        &  \supset & M_{T_{n-1}}    & \supset        & \cdots  &  \supset    & M_{T_2} &  \supset   & M_{T_1} \\
   \dTo_{p_n}   &           & \dTo_{p_{n-1}}                       &             &             &                    & \dTo_{p_2}  &         & \dTo_{p_1} \\
     N_{T_n}      & \supset   & N_{T_{n-1}} & \supset        & \cdots  & \supset     & N_{T_2}  & \supset     & N_{T_1} \\
\dTo                 &             & \dTo                  &                    &            &                  &    \dTo    &                &\dTo \\
    T_n                &  \supset  &  T_{n-1}      & \supset         & \cdots  & \supset    & T_2          & \supset   &  T_1
\end{diagram}

Since $T_1$ consists of a single vertex, $M_{T_1}=V_K$ consists of a single vertex space of $M_K$, and it covers a vertex space $V\subset M$. If $V_K\to V$ is a finite cover, we simply take $N_{T_1}=V_K$. Otherwise, we apply Proposition \ref{hypfiber}, \ref{geofinite}, \ref{Seifertfiber}, \ref{Seifertpartialfiber} or \ref{S1bundle} to $V_K\to V$, $K\cap V_K\subset V_K$ and boundary components of $V_K$ that lie in the interior of $M_K$, to get an intermediate finite cover $\hat{V}\to V$ of $V_K\to V$. Since all constants $A_{V_K}$ and $A_{(V_K,E)}$ are factors of $\mathfrak{A}$, we can choose the parameters in these propositions such that the boundary condition hold for $N_{T_1}=\hat{V}$. In particular, Proposition \ref{geofinite} (5) guarantees that boundary condition (3) (b) can hold for $\hat{V}$, and Proposition \ref{Seifertpartialfiber} (4) guarantees that boundary condition (3) (c) can hold for $\hat{V}$.

  Suppose that we have constructed a finite semi-cover $s_j:N_{T_j}\to M$ and a semi-cover $p_j:M_{T_j}\to N_{T_j}$ satisfying all conditions in this lemma.

  Now we need to construct $N_{T_{j+1}}$. Since $T_{j+1}$ has only one more vertex than $T_j$ and $T_{j+1}$ is a tree, $M_{T_{j+1}}\setminus M_{T_j}$ consists of a single vertex space $V_K\subset M_K$. As submanifolds of $M_T$, $V_K\cap M_{T_j}$ consists of one single edge space $E$. Let $V$ be the vertex space of $M$ covered by $V_K$. If $V_K\to V$ is a finite cover, we simply take $\hat{V}=V_K$. Otherwise, we will apply Proposition \ref{hypfiber}, \ref{geofinite}, \ref{Seifertfiber}, \ref{Seifertpartialfiber} or \ref{S1bundle} to $V_K\to V$, $K\cap V_K\subset V_K$ and all boundary components of $V_K$ that lie in the interior of $M_K$, to get an intermediate finite cover $q:\hat{V}\to V$ of $\pi|_{V_K}:V_K\to V$ as in the following diagram.
  \begin{diagram}
V_K &\rTo^{p} & \hat{V}\\
& \rdTo_{\pi|_{V_K}} & \dTo_{q}\\
 & & V
\end{diagram}

By inductive hypothesis, the boundary condition of $N_{T_{j+1}}$ holds for all components of $\partial N_{T_j}\setminus p_j(E)$. So when we construct $\hat{V}$, we need the following conditions hold.
\begin{enumerate}
\item The boundary condition of $N_{T_{i+1}}$ holds for $\hat{V}$.
\item $p(E)$ and $p_j(E)$ are the same finite cover of $\pi(E)$, i.e. there is a homeomorphism $g:p(E)\to p_j(E)$ such that the following diagram commute.
\begin{diagram}
& & & & E& & & &\\
& & & \ldTo^{p|_{E}}& & \rdTo^{p_j|_{E}} & & & \\
\partial \hat{V}& \supset & p(E) & & \rTo^g  &  & p_j(E)& \subset & \partial N_{T_j}\\
& &  &  \rdTo_{q|_{p(E)}} & & \ldTo_{s_j|_{p_j(E)}} & & \\
&  & &  & \pi(E) & & & &
\end{diagram}
\end{enumerate}

Once we have the above commutative diagram and make sure the boundary condition holds, then we take $N_{T_{j+1}}=\hat{V} \cup_g N_{T_j}$. The finite semi-cover $s_{j+1}:N_{T_{j+1}}\to M$ is defined by $q$ (on $\hat{V}$) and $s_j$ (on $N_{T_j}$), and the semi-cover $p_{j+1}:M_{T_{j+1}}=V_K\cup M_{T_j}\to N_{T_{j+1}}$ is defined by $p$ (on $V_K$) and $p_j$ (on $M_{T_j}$).

\bigskip

{\it Case A. $E$ is a torus.} Then $V_K$ is either a geometrically finite piece or a noncompact $S^1$-bundle. Boundary condition (1) implies that $p_j|_{E}:E\to p_j(E)$ is a homeomorphism, while Proposition \ref{geofinite} (3) and \ref{S1bundle} (3) imply that $p|_E:E\to p(E)$ is a homeomorphism. So the desired homeomorphism $g$ exists. Since the constant $A_{V_K}$ is a factor of $\mathfrak{A}$, and parameters in Proposition \ref{geofinite} and \ref{S1bundle} for boundary components of $V_K$ can be chosen independently, we can construct $\hat{V}$ such that the boundary condition holds.

\bigskip

{\it Case B. $E$ is a cylinder.} Boundary condition (2) implies that $s_j|_{p_j(E)}:p_j(E)\to \pi(E)$ is the finite cover corresponding to $\mathbb{Z}[c]\oplus \mathbb{Z}[\tau \mathfrak{A}^{n-j+1}t]<\pi_1(\pi(E))$ for some slope $t\subset \pi(E)$ and $\tau\in \mathbb{Z}_+$. Here $t$ is determined by the piece of $M_{T_j}$ adjacent to $E$ and $V_K$, according to Table 4.

We first consider the case that $V_K\to V$ corresponds to a geometrically finite subgroup (of hyperbolic $V$) or has an induced $S^1$-bundle structure (from Seifert $V$). By Table 4, we have chosen the same slope $t\subset \pi(E)$ from both $M_{T_j}$ and $V_K$ viewpoint. Then in Proposition \ref{geofinite} and \ref{S1bundle}, we can choose the parameters for boundary components of $V_K$ independently. So we can construct $\hat{V}$ such that the homeomorphism $g$ exists and the boundary condition is satisfied.

It remains to consider the case that $V_K\to V$ corresponds to a virtually fibered or partially fibered subgroup. Note that in Proposition \ref{hypfiber}, \ref{Seifertfiber} and \ref{Seifertpartialfiber}, the parameter $\alpha$ only depends on the vertex space $V_K$, so we can not freely choose parameters for each boundary component of $V_K$ as in the previous case. Moreover, the slope $t\subset \pi(E)$ chosen from $M_{T_j}$ viewpoint and $t'\subset \pi(E)$ chosen from $V_H$ viewpoint may not be same with each other. For the cylinder $E\subset \partial V_K$, we apply Proposition \ref{hypfiber}, \ref{Seifertfiber} or \ref{Seifertpartialfiber} to $V_K$ to get a constant $A_{(V_K,E)}$, and apply Lemma \ref{Z2} to $c,t,t'\in \pi_1(\pi(E))$ to get two constants $B_E,B'_E\in \mathbb{Z}_+$. Then Lemma \ref{Z2} implies that
$$\mathbb{Z}[c]\oplus\mathbb{Z}[\tau \mathfrak{A}^{n-j+1}t]=\mathbb{Z}[c]\oplus\mathbb{Z}[\tau \mathfrak{A}^{n-j}\frac{\mathfrak{A}}{B_E}B_Et]=\mathbb{Z}[c]\oplus \mathbb{Z}[\tau\mathfrak{A}^{n-j}\frac{\mathfrak{A}}{B_E}B'_Et']. $$
By the choice of $\mathfrak{A}$, $\frac{\mathfrak{A}}{B_EA_{(V_K,E)}}$ is an integer. Now we must apply Proposition \ref{hypfiber}, \ref{Seifertfiber} or \ref{Seifertpartialfiber} with parameter $\alpha=\tau\mathfrak{A}^{n-j}B'_E\frac{\mathfrak{A}}{B_EA_{(V_K,E)}}$, so that $q|_{p(E)}:p(E)\to \pi(E)$ is the finite cover corresponding to $$\mathbb{Z}[c]\oplus \mathbb{Z}[\alpha A_{(V_K,E)}t']=\mathbb{Z}[c]\oplus \mathbb{Z}[\tau\mathfrak{A}^{n-j}\frac{\mathfrak{A}}{B_E}B'_Et']<\pi_1(\pi(E)).$$ So in this case, the homeomorphism $g$ exists.

For any other cylinder boundary component $\bar{E}\subset \partial V_K$ lying in the interior of $M_K$, by our choice of parameter $\alpha$, $q|_{p(\bar{E})}:p(\bar{E})\to \pi(\bar{E})$ is the finite cover corresponding to
$$\mathbb{Z}[\bar{c}]\oplus \mathbb{Z}[\alpha A_{(V_K,\bar{E})}\bar{t}]=\mathbb{Z}[\bar{c}]\oplus \mathbb{Z}[\big(\tau B'_EA_{(V_K,\bar{E})}\frac{\mathfrak{A}}{B_EA_{(V_K,E)}}\big)\mathfrak{A}^{n-(j+1)+1}\bar{t}]<\pi_1(\pi(\bar{E})),$$
which satisfies boundary condition (2) of $N_{T_{j+1}}$. If $V_K$ has a plane boundary component $\bar{E}\subset \partial V_K$ that lies in the interior of $M_K$, then $V_K$ must be a partially fibered piece, so it need to satisfy boundary condition (3) (a) or (3) (c). The above computation implies that boundary condition (3) (a) and (3) (c) hold for the $v$-component. We can make sure boundary condition (3) (a) and (3) (c) also hold for the $u$-component, since $A_{V_K}$ is a factor of $\mathfrak{A}$ and we can choose arbitrary parameter $\beta_{\bar{E}}$ in Proposition \ref{Seifertpartialfiber}.

\bigskip

{\it Case C. $E$ is a plane.} Since $E$ lies in the interior of $M_T$, boundary condition (3) (a) implies that $s_j|_{p_j(E)}:p_j(E)\to \pi(E)$ is the finite cover corresponding to $\mathbb{Z}[\mu \mathfrak{A}^{n-j+1} u]\oplus \mathbb{Z}[\nu \mathfrak{A}^{n-j+1}v]<\pi_1(\pi(E))$ for chosen nonparallel slopes $u,v\subset \pi(E)$ and $\mu,\nu \in \mathbb{Z}_+$. Recall that when we use Table 5 to choose two pairs of nonparallel slopes on $\pi(E)$ from viewpoints of $M_{T_j}$ and $V_K$ respectively, we have $v'=u$ and $u'=v$. If $V$ is a hyperbolic $3$-manifold, then $V_K$ corresponds to a geometrically finite subgroup of $\pi_1(V)$. Since we can choose arbitrary parameters in Proposition \ref{geofinite} for different  boundary components of $V_K$, we can make sure the homeomorphism $g:p(E)\to p_j(E)$ exists and boundary conditions are satisfied.

If $V$ is a Seifert manifold, then $V_K$ corresponds to a partially fibered subgroup of $\pi_1(V)$, and $v'=u$ is the regular fiber of $V$ (by Table 5). So we have to take the parameter of $V_K$ to be $\alpha=\mu \mathfrak{A}^{n-j}\frac{\mathfrak{A}}{A_{V_K}}$ (for the whole manifold $V_K$) and $\beta_E=\nu\mathfrak{A}^{n-j}\frac{\mathfrak{A}}{A_{V_K}}$ (for the boundary component $E\subset \partial V_K$), such that the finite cover $q|_{p(E)}:p(E)\to \pi(E)$ corresponds to
$$\mathbb{Z}[(\nu \mathfrak{A}^{n-j}\frac{\mathfrak{A}}{A_{V_K}})A_{V_K} u']\oplus \mathbb{Z}[(\mu \mathfrak{A}^{n-j}\frac{\mathfrak{A}}{A_{V_K}})A_{V_K}v']=\mathbb{Z}[\nu \mathfrak{A}^{n-j+1}v]\oplus \mathbb{Z}[\mu \mathfrak{A}^{n-j+1}u]<\pi_1(\pi(E)).$$
Then the homeomorphism $g:p(E)\to p_j(E)$ exists. For the boundary condition, since the regular fiber of $\hat{V}$ covers the regular fiber of $V$ by a degree $\mu \mathfrak{A}^{n-j+1}$ map, boundary condition (2) holds, and boundary conditions (3) (a) and (3) (c) hold on the $v$-component. In Proposition \ref{Seifertpartialfiber}, for different plane boundary components of $V_K$ that lie in the interior of $M_K$, we can choose arbitrary parameters for their $u$-components, so we can make sure boundary condition (3) (a) and (3) (c) hold for the $u$-component.

Now we finish the proof of Lemma \ref{constructionontree}, and Step II is done.
\end{proof}

\textbf{ Step III. After necessary modifications of the $N_{T_n}$ constructed in Step II, paste pairs of boundary components of $N_{T_n}$ corresponding to edges in $G_K\setminus T$, to get the desired finite semi-cover $N\to M$.}

For simplicity, we denote the finite semi-cover $s_n:N_{T_n}\to M$ constructed in step II by $s:N_T\to M$, and denote the semi-cover $p_n:M_{T_n}\to N_{T_n}$ by $p:M_T\to N_T$. Since $M_T$ can be obtained by cutting $M_K$ along those edge spaces in $M_K$ corresponding to edges in $G_K\setminus T$, the first vertical map in the following diagram is a quotient map by identifying pairs of homeomorphic boundary components of $M_T$. Then our goal is to modify $N_T$ so that we can paste homeomorphic boundary components of $N_T$ to get the desired $N$ (via the second vertical map in the following diagram).
\begin{diagram}
& & M_T & \rTo^p & N_T & \rTo & T\\
& & \dTo &        & \dDotsto &       & \cap\\
K& \subset & M_K  & \rDotsto & N & \rDotsto & G_K
\end{diagram}

Let $e_1,\cdots,e_t$ be all edges in $G_K\setminus T$, and $E_1,\cdots,E_t$ be the corresponding edge spaces of $M_K$. For each $E_i$, it corresponds to two boundary components $E_i',E_i''\subset \partial M_T$, and Lemma \ref{constructionontree} (3) implies that $$\{p(E_1'),p(E_1''),p(E_2'),p(E_2''),\cdots,p(E_t'),p(E_t'')\}$$ are distinct boundary components of $N_T$. Here $s|_{p(E_i')}:p(E_i')\to \pi(E_i)$ and $s|_{p(E_i'')}:p(E_i'')\to \pi(E_i)$ might be two distinct finite covers of $\pi(E_j)$ (up to isomorphism).

Our modification process on $N_T$ is given by the following lemma.

\begin{lemma}\label{pasting}
For each edge space $E\in \{E_1,\cdots,E_t\}$, we can modify the construction of $N_T$, such that the following holds.
\begin{enumerate}
\item The first three conditions in Lemma 6.4 still hold for $s:N_T\to M$ and $p:M_T\to N_T$.
\item For any $E_i\ne E$, the covering maps $s|_{p(E_i')}:p(E_i')\to \pi(E)$ and $s|_{p(E_i'')}:p(E_i'')\to \pi(E)$ are not affected by this modification.
\item After the modification, there exists a homeomorphism $g:p(E')\to p(E'')$ such that the following diagram commute.

\begin{diagram}
E'& =&E&=&E'' \\
\dTo_{p|_{E'}}& & & &\dTo_{p|_{E''}} \\
p(E') & & \rTo^g  &  & p(E'')\\
& \rdTo_{s|_{p(E')}} & & \ldTo_{s|_{p(E'')}} &\\
& & \pi(E) & &
\end{diagram}
\end{enumerate}
\end{lemma}

\begin{proof} We prove this lemma by a case-by-case argument.

{\it Case A. $E$ is a torus.} Boundary condition (1) implies that both $p|_{E'}:E'\to p(E')$ and $p|_{E''}:E''\to p(E'')$ are homeomorphisms. Since $E'$ is homeomorphic to $E''$, the desired homeomorphism $g: p(E')\to p(E'')$ exists.

\bigskip

{\it Case B. $E$ is a plane.}

By Table 2, the two adjacent pieces of $E'$ and $E''$ in $M_T$ are both geometrically finite or partially fibered, and we denote them by $V_K'$ and $V_K''$ respectively.

We first suppose that both $E'$ and $E''$ lie in geometrically finite vertex spaces of $M_T$. Then boundary condition (3) (b) implies that $p(E')\to \pi(E)$ corresponds to $$\mathbb{Z}[\mathfrak{A}u'] \oplus \mathbb{Z}[\mathfrak{A}v']<\pi_1(\pi(E))$$ and $p(E'')\to \pi(E)$ corresponds to $$\mathbb{Z}[\mathfrak{A}u''] \oplus \mathbb{Z}[\mathfrak{A}v'']<\pi_1(\pi(E)).$$
Here the slopes $u',v'$ and $u'',v''$ are chosen by Table 5 from $V_K'$ and $V_K''$ viewpoint respectively, so we have $u'=v''$ and $v'=u''$. Then $p(E')\to \pi(E)$ and  $p(E'')\to \pi(E)$ correspond to the same subgroup of $\pi_1(\pi(E))$, and the desired homeomorphism $g:p(E')\to p(E'')$ exists.

If one of $E'$ and $E''$ lies in a geometrically finite piece and the other one lies in a partially fibered piece, we assume that $E'$ lies in a geometrically finite piece $V_K'$, and let $\hat{V}'$ be the image of $V_K'$ in $N_T$. Boundary condition (3) (b) implies that $p(E')\to \pi(E)$ corresponds to subgroup
$$\mathbb{Z}[\mathfrak{A}u'] \oplus \mathbb{Z}[\mathfrak{A}v']<\pi_1(\pi(E)).$$
Since $u'=v''$ and $v'=u''$ hold, boundary condition (3) (c) implies that $p(E'')\to \pi(E)$ corresponds to subgroup
$$\mathbb{Z}[\mathfrak{A}u''] \oplus \mathbb{Z}[\nu'' \mathfrak{A}v'']=\mathbb{Z}[\nu'' \mathfrak{A}u']\oplus \mathbb{Z}[\mathfrak{A}v']<\pi_1(\pi(E))$$
for some $\nu''\in \mathbb{Z}_+$. These two subgroups of $\pi_1(\pi(E))$ may not be equal to each other, so we need to apply Proposition \ref{geofinite} to modify our construction of $\hat{V}'$. For the parameters we used to construct $\hat{V}'$, we multiply the parameter corresponding to slope $u'\subset \pi(E)$ by $\nu''$, and fix all other parameters to construct a new $\hat{V}'$. At first, $V'_K\cap K$ still embeds into this new $\hat{V}'$. Since all other parameters are not changed, this new $\hat{V}'$ can be pasted with other (old) pieces of $N_T$ to get a new $N_T$, and any other boundary components of $N_T$ corresponding to edges in $G_K\setminus T$ are not changed. By our modification, $p(E')$ and $p(E'')$ correspond to the same subgroup of $\pi_1(\pi(E))$, so the desired homeomorphism $g:p(E')\to p(E'')$ exists.

If both $E'$ and $E''$ lie in partially fibered pieces of $M_T$, we denote these two pieces by $V_H',V_H''$ respectively, and let $\hat{V}'$ and $\hat{V}''$ be the images of $V_H'$ and $V_H''$ in $N_T$ respectively. Then boundary condition (3) (c) implies that $p(E')\to \pi(E)$ corresponds to
$$\mathbb{Z}[\mathfrak{A}u'] \oplus \mathbb{Z}[\nu' \mathfrak{A}v']<\pi_1(\pi(E)),$$
and  $p(E'')\to \pi(E)$ corresponds to
$$\mathbb{Z}[\mathfrak{A}u''] \oplus \mathbb{Z}[\nu'' \mathfrak{A}v'']=\mathbb{Z}[\nu'' \mathfrak{A}u']\oplus \mathbb{Z}[\mathfrak{A}v']<\pi_1(\pi(E))$$
for some $\nu',\nu'' \in \mathbb{Z}_+$. These two subgroups of $\pi_1(\pi(E))$ may not be equal to each other. We apply Proposition \ref{Seifertpartialfiber} to modify our construction of $\hat{V}'$ and $\hat{V}''$ as in the previous case. In Proposition \ref{Seifertpartialfiber}, the parameters not corresponding to regular fibers ($u'$ and $u''$ in our case) can be chosen arbitrary, and it does not affect other boundary components. For the parameters we used to construct $\hat{V}'$ and $\hat{V}''$, we multiply the parameter corresponding to $u'$ by $\nu''$ and multiply the parameter corresponding to $u''$ by $\nu'$, while we fix all other parameters to construct our new $\hat{V}'$ and $\hat{V}''$. Similarly to the previous case, after our modification, $K\cap V'_K$ and $K\cap V''_K$ still embed into the new $\hat{V}'$ and $\hat{V}''$, and the desired homeomorphism $g:p(E')\to p(E'')$ exists. Moreover, this modification does not affect other pieces of $N_T$ and other boundary components of $N_T$ corresponding to edges in $G_K\setminus T$.

\bigskip

{\it Case C. $E$ is a cylinder.} Let $V_K'$ and $V_K''$ be the pieces of $M_T$ containing $E'$ and $E''$ respectively.

Case C.1. At least one of $V_K'$ and $V_K''$ does not correspond to virtually fibered or partially fibered subgroups. We assume that $V_K'$ is either geometrically finite, or has an induced $S^1$-bundle structure. By boundary condition (2),  $p(E')\to \pi(E)$ corresponds to subgroup
$$\mathbb{Z}[c] \oplus \mathbb{Z}[\tau'\mathfrak{A}t']<\pi_1(\pi(E)),$$
and  $p(E'')\to \pi(E)$ corresponds to subgroup
$$\mathbb{Z}[c] \oplus \mathbb{Z}[\tau'' \mathfrak{A}t'']<\pi_1(\pi(E))$$
for some $\tau',\tau''\in \mathbb{Z}_+$. Since $V_K'$ is neither virtually fibered nor partially fibered, Table 4 implies $t'=t''$, but these two subgroups of $\pi_1(\pi(E))$ may not be equal to each other.

Now we apply Proposition \ref{geofinite} or \ref{S1bundle} to modify our construction of $\hat{V}'$. In these two propositions, the parameter of $t'$ can be chosen arbitrarily and it does not affect other boundary components of $\hat{V}'$. For the parameters we used to construct $\hat{V}'$, we change the parameter $\alpha_{E'}$ (corresponding to $t'$) from $\tau'\frac{\mathfrak{A}}{A_{V'_K}}$ to $\tau''\frac{\mathfrak{A}}{A_{V'_K}}$ and fix all other parameters to construct our new $\hat{V}'$. Similarly to Case B, after doing this modification, there exists the desired homeomorphism $g:p(E')\to p(E'')$, and this modification does not affect other boundary components of $N_T$ corresponding to edges in $G_K\setminus T$.

\bigskip

Case C.2. If we are not in case C.1, both $V_K'$ and $V_K''$ correspond to virtually fibered or partially fibered subgroups. Then by edge condition (3), the edge $e\subset G_K\setminus T$ is contained in a simple cycle $c\subset G_{\Phi(H)}\subset G_K$. This is the only case that we need to use the aspirality of $H$ in this proof, and we will prove that $p(E')$ and $p(E'')$ are the same finite cover of $\pi(E)$, without doing any modification.

For the inclusion $i:c\to G_{\Phi(H)}$ and the projection $r:\Phi(H)\to G_{\Phi(H)}$, we take a lifting $j:c\to \Phi(H)$ such that $r\circ j=i$ and $j:c\to \Phi(H)$ transverses with the decomposition circles of $\Phi(H)$. We endow $c$ with an arbitrary orientation, then the decomposition circles of $\Phi(H)$ decompose $c$ into a concatenation of oriented paths $\delta_1,\delta_2,\cdots \delta_n$, and each $\delta_i$ is an oriented proper path in a vertex space $\Sigma_i\subset \Phi(H)$. Moreover, we can suppose the terminal point of $\delta_n$ equals the initial point of $\delta_1$, and this point lies in the decomposition circle corresponding to the edge $e\subset G_K\setminus T$.

For each $\delta_i$, $s_{\delta_i}\in \mathbb{Q}_+$ is defined in Section \ref{arcs}, and the value of generalized spirality character $s:H_1(\Phi(H);\mathbb{Z})\to \mathbb{Q}_+^{\times}$ on $[c]\in H_1(\Phi(H);\mathbb{Z})$ is defined by $s([c])=\Pi_{i=1}^ns_{\delta_i}.$ Then the aspirality of $H$ implies $$\prod_{i=1}^ns_{\delta_i}=1.$$

By the assumption of this proposition, each $\Sigma_i$ is a compact orientable hyperbolic surface. Let $c_i^{\text{ini}}$ be the boundary component of $\Sigma_i$ containing the initial point of $\delta_i$, and let $c_i^{\text{ter}}$ be the boundary component of $\Sigma_i$ containing the terminal point of $\delta_i$. As decomposition circles of $\Phi(H)$, we have $c_i^{\text{ter}}=c_{i+1}^{\text{ini}}$ and $c_n^{\text{ter}}=c_1^{\text{ini}}$. Let $\Sigma$ be the surface obtained from $\{\Sigma_i\}_{i=1}^n$ by identifying $c_i^{\text{ter}}$ with $c_{i+1}^{\text{ini}}$ for $i=1,2,\cdots,n-1$ (without identifying $c_n^{\text{ter}}$ with $c_{1}^{\text{ini}}$). Then $\Sigma$ is a subsurface of $\Phi(H)$ and it is also a subsurface of $M_T$. In $M_H$, the edge space $E$ containing $c_1^{\text{ini}}=c_n^{\text{ter}}$ corresponds to two distinct boundary components $E',E''\subset \partial M_T$, with $c_1^{\text{ini}}\subset E'$ and $c_n^{\text{ter}}\subset E''$. Then the semi-cover $s:M_T\to N_T$ embeds $\Sigma\subset M_T$ into $N_T$.

Let $\hat{V}_i$ be the piece of $N_T$ containing $\Sigma_i$, and let $\hat{T}_i^{\text{ini}}$ and $\hat{T}_i^{\text{ter}}$  be the boundary components of $\hat{V}_i$ containing $c_i^{\text{ini}}$ and $c_i^{\text{ter}}$ respectively. Then in $N_T$, we have  $\hat{T}_i^{\text{ter}}=\hat{T}_{i+1}^{\text{ini}}$ for $i=1,2,\cdots,n-1$, while $p(E')=\hat{T}_1^{\text{ini}}$ and $p(E'')=\hat{T}_n^{\text{ter}}$.

Let $V_i$ be the image of $\hat{V}_i$ in $M$, and let $T_i^{\text{ini}}$ and $T_i^{\text{ter}}$ be the images of $\hat{T}_i^{\text{ini}}$ and $\hat{T}_i^{\text{ter}}$ in $M$ respectively. In $M$, we have $T_i^{\text{ter}}=T_{i+1}^{\text{ini}}$ for $i=1,2,\cdots,n-1$ and $\pi(E)=T_1^{\text{ini}}=T_n^{\text{ter}}$.

If $V_i$ is hyperbolic, then the piece of $M_T$ containing $\Sigma_i$ corresponds to a virtually fibered subgroup of $\pi_1(V_i)$. By our construction in Proposition \ref{hypfiber}, $\hat{V}_i$ is a $\Sigma_i$-bundle over $S^1$. By Table 4, the slope $t$ chosen in Proposition \ref{hypfiber} is the degeneracy slope. Then Proposition \ref{hypfiber} (3) implies that, for each boundary component $T\subset \partial \hat{V}_i$, the algebraic intersection number between the degeneracy slope on $T$ and $\Sigma_i\cap T$ is $1$. So the pseudo-Anosov monodromy $\phi_i:\Sigma_i\to \Sigma_i$ of $\hat{V}_i$ restricts to identity on $\partial \Sigma_i$, and it satisfies the conditions in Case II of Section \ref{arcs} (in particular condition (3)). So we have $$s_{\delta_i}=\frac{[\hat{T}_i^{\text{ini}}:T_i^{\text{ini}}]}{[\hat{T}_i^{\text{ter}}:T_i^{\text{ter}}]}.$$

If $V_i$ is a Seifert manifold, then the piece of $M_T$ containing $\Sigma_i$ corresponds to a virtually fibered or partially fibered subgroup of $\pi_1(V_i)$. Then we use Case I of Section \ref{arcs} to compute $s_{\delta_i}$. By Proposition \ref{Seifertfiber} (3) and \ref{Seifertpartialfiber} (3), the covering map $\hat{T}_i^{\text{ini}}\to T_i^{\text{ini}}$ corresponds to subgroup $\mathbb{Z}[c_i^{\text{ini}}]\oplus \mathbb{Z}[\alpha A t_i^{\text{ini}}]<\pi_1(T_i^{\text{ini}})$ for some $\alpha\in \mathbb{Z}_+$. By Table 4, $t_i^{\text{ini}}$ is the regular fiber $h$ of $V_i$, so we have
$$[\hat{T}_i^{\text{ini}}:T_i^{\text{ini}}]=\alpha A\langle c_i^{\text{ini}}, h\rangle.$$
Since the induced covering degree of $\hat{V}_i\to V_i$ on regular fiber is independent of which boundary component we take, we also have
$$[\hat{T}_i^{\text{ter}}:T_i^{\text{ter}}]=\alpha A\langle c_i^{\text{ter}}, h\rangle,$$
So we get
$$s_{\delta_i}=\frac{\langle c_i^{\text{ini}},h\rangle}{\langle c_i^{\text{ter}},h\rangle}=\frac{\alpha A \langle c_i^{\text{ini}},h\rangle}{\alpha A \langle c_i^{\text{ter}},h\rangle}=\frac{[\hat{T}_i^{\text{ini}}:T_i^{\text{ini}}]}{[\hat{T}_i^{\text{ter}}:T_i^{\text{ter}}]}.$$

Since $s([c])=1$, while $\hat{T}_i^{\text{ter}}=\hat{T}_{i+1}^{\text{ini}}$ and $T_i^{\text{ter}}=T_{i+1}^{\text{ini}}$ hold for $i=1,2,\cdots,n-1$, we have
$$1=s([c])=\prod_{i=1}^ns_{\delta_i}=\prod_{i=1}^n\frac{[\hat{T}_i^{\text{ini}}:T_i^{\text{ini}}]}{[\hat{T}_i^{\text{ter}}:T_i^{\text{ter}}]}=\Big(\prod_{i=1}^{n-1}\frac{[\hat{T}_i^{\text{ini}}:T_i^{\text{ini}}]}{[\hat{T}_{i+1}^{\text{ini}}:T_{i+1}^{\text{ini}}]}\Big)\cdot \frac{[\hat{T}_n^{\text{ini}}:T_n^{\text{ini}}]}{[\hat{T}_n^{\text{ter}}:T_n^{\text{ter}}]}=\frac{[\hat{T}_1^{\text{ini}}:T_1^{\text{ini}}]}{[\hat{T}_n^{\text{ter}}:T_n^{\text{ter}}]}.$$
Since $\pi(E)=T_n^{\text{ter}}=T_1^{\text{ini}}$, while $p(E')=\hat{T}_1^{\text{ini}}$ and $p(E'')=\hat{T}_n^{\text{ter}}$ hold, we get
 $$[p(E'):\pi(E)]=[p(E''):\pi(E)].$$

By boundary condition (2), $p(E')\to \pi(E)$ corresponds to subgroup $\mathbb{Z}[c]\oplus \mathbb{Z}[\tau' \mathfrak{A}t']<\pi_1(\pi(E))$ and $p(E')\to \pi(E)$ corresponds to subgroup $\mathbb{Z}[c]\oplus \mathbb{Z}[\tau'' \mathfrak{A}t'']<\pi_1(\pi(E))$ for some $\tau',\tau''\in \mathbb{Z}_+$. By Table 4, $t'$ and $t''$ may not be the same slope on $\pi(E)$. For $c,t',t''\in \pi_1(\pi(E))\cong \mathbb{Z}^2$, Lemma \ref{Z2} gives us two constants $B'$ and $B''$ and they are both factors of $\mathfrak{A}$. Then we have
$$[p(E'):\pi(E)]=\Big[\pi_1(\pi(E)):\mathbb{Z}[c]\oplus \mathbb{Z}[\tau' \mathfrak{A}t']\Big]=\Big[\pi_1(\pi(E)):\mathbb{Z}[c]\oplus \mathbb{Z}[B't']\Big]\cdot\frac{\tau'\mathfrak{A}}{B'}.$$
Similarly, we have
$$[p(E''):\pi(E)]=\Big[\pi_1(\pi(E)):\mathbb{Z}[c]\oplus \mathbb{Z}[\tau'' \mathfrak{A}t'']\Big]=\Big[\pi_1(\pi(E)):\mathbb{Z}[c]\oplus \mathbb{Z}[B''t'']\Big]\cdot\frac{\tau''\mathfrak{A}}{B''}.$$
Since $[p(E'):\pi(E)]=[p(E''):\pi(E)]$ and $\mathbb{Z}[c]\oplus \mathbb{Z}[B't']=\mathbb{Z}[c]\oplus \mathbb{Z}[B''t'']$, the above two equalities imply that $\frac{\tau'\mathfrak{A}}{B'}=\frac{\tau''\mathfrak{A}}{B''}$.

Then Lemma \ref{Z2} implies that $$\pi_1(p(E'))=\mathbb{Z}[c]\oplus \mathbb{Z}[\frac{\tau' \mathfrak{A}}{B'}B't']=\mathbb{Z}[c]\oplus \mathbb{Z}[\frac{\tau' \mathfrak{A}}{B'}B''t'']=\mathbb{Z}[c]\oplus \mathbb{Z}[\frac{\tau'' \mathfrak{A}}{B''}B''t'']=\pi_1(p(E'')).$$ So $p(E')\to \pi(E)$ and $p(E'')\to \pi(E)$ correspond to the same subgroup of $\pi_1(\pi(E))$, and the desired homeomorphism $g:p(E')\to p(E'')$ exists.

\end{proof}

For each $E_i$, the above modification on $p(E_i')$ and $p(E_i'')$ does not affect $p(E_j')$ and $p(E_j'')$ for any $j\ne i$. So we can paste all pairs $\{p(E'_i),p(E''_i)\}$ for $i=1,2,\cdots,t$ by homeomorphisms given by Proposition \ref{pasting}, to get a $3$-manifold $N$ with dual graph isomorphic to $G_K$. This finishes Step III.

\bigskip

In step III, we modified $N_T$ such that we can paste pairs of its boundary components corresponding to edges in $G_K\setminus T$, to get a finite semi-cover $N\to M$ whose dual graph is isomorphic to $G_K$. Then we have the following diagram.

\begin{diagram}
K\cap M_T &\subset &M_T & \rTo^{p}  & N_T  &            & \\
\cap      &        & \dTo &          & \dTo &  \rdTo^{s} & \\
K         & \subset & M_K  & \rTo     & N    & \rTo^{f}  & M\\
\end{diagram}

Lemma \ref{constructionontree} (2) implies that the restriction of $p:M_T\to N_T$ embeds $K\setminus (\cup_{i=1}^t E_i)=K\cap M_T$ into $N_T$, and the modification of $N_T$ in Step III does not change this property. So when we paste $N_T$ to get $N$, $K\subset M_K$ embeds into $N$. The finite semi-cover $f:N \to M$ is induced by the finite semi-cover $s:N_T\to M$ constructed in Step II. This finishes the proof of Proposition \ref{embedsemicover}.

\bigskip

Now we are ready to prove Proposition \ref{toriboundaryseparability}.

\begin{proof} We suppose that $H<\pi_1(M)$ is aspiral, and will prove it is separable in $\pi_1(M)$.

  Let $\chi:H_1(\Phi(H);\mathbb{Z})\to \mathbb{Z}_2$ be the orientability character of $\Phi(H)$. Since the restriction of $\chi$ on $\partial \Phi(H)$ is trivial, and $\pi_1(\Phi(H))$ corresponds to a subgraph of the dual graph of $H$, $\chi$ can be extended to a homomorphism $\chi':H\to \mathbb{Z}_2$. Let $H'$ be the kernel of $\chi'$, then it is an index-$1$ or index-$2$ subgroup of $H$. Then for $H'<\pi_1(M)$, $\Phi(H')$ is an orientable surface.

  By Lemma 3.1 of \cite{PW}, $M$ has a finite cover $M'$ such that for each Seifert piece $V\subset M'$, the base orbifold of $V$ is an orientable surface (with no singular points) with positive genus. Since $H$ is aspiral in $M$, Lemma \ref{subgroupspirality} and \ref{manifoldcoverspirality} imply that $H''=H'\cap \pi_1(M')$ is aspiral in $M'$. In the following, we will prove that $H''$ is separable in $\pi_1(M')$, then Lemma \ref{basicLERF} (2) and (3) imply that $H$ is separable in $\pi_1(M)$.

  In the following, we abuse notation and denote the above $H''<\pi_1(M')$ by $H<\pi_1(M)$. So we can assume that $\Phi(H)$ is an orientable surface, and the base orbifold of any Seifert piece in $M$ is an orientable surface (with no singular points) with positive genus, thus the assumption of Proposition \ref{embedsemicover} holds.

  Let $\pi: M_H\to M$ be the covering space of $M$ corresponds to $H<\pi_1(M)$. For any compact subset $K\subset M_H$, we enlarge $K$ such that it is path-connected, while it contains a Scott core of $M_H$ and $\Phi(H)$. Then Proposition \ref{embedsemicover} gives us a finite semi-cover $f:N\to M$ and an embedding $i:K\hookrightarrow N$, such that $f\circ i=\pi|_{K}$.

  For the finite semi-cover $f: N\to M$, let $M_N\to M$ be the covering space of $M$ correspond to $\pi_1(N)<\pi_1(M)$. Then $f:N\to M$ lifts to an embedding $j_N:N\hookrightarrow M_N$. Since $N$ is compact, Lemma \ref{semicover} implies that $M_N\to M$ has an intermediate finite cover $q: M'\to M$ (with covering map $r:M_N\to M'$) such that $j=r\circ j_N:N\to M'$ is an embedding. The maps we have constructed are summarized in the following diagram, while both horizontal maps are embeddings.
\begin{diagram}
& & & & M_N\\
& & & \ruTo^{j_N} & \dTo_{r}\\
K &\rTo^i & N & \rTo^j & M'\\
&\rdTo_{\pi|_K} &\dTo_{f} & \ldTo_{q} &\\
& &M & &
\end{diagram}

Now we have a finite cover $q:M'\to M$ and an embedding $k=j\circ i:K\hookrightarrow M'$ such that $q\circ k=\pi|_K$, so we have $(\pi|_K)_*(\pi_1(K))<q_*(\pi_1(M'))$. Since $K$ contains a compact core of $M_H$, $\pi_*(\pi_1(M_H))=(\pi|_K)_*(\pi_K)$ is contained in $q_*(\pi_1(M'))$. So $\pi:M_H\to M$ lifts to $p: M_H\to M'$, which is an extension of $k:K\hookrightarrow M'$.
\begin{diagram}
M_H & \rTo^{p}  & M'\\
\cup &  \ruTo^{k} & \dTo_{q}\\
K  & \rTo^{\pi|_K} & M\\
\end{diagram}

So $q:M'\to M$ is an intermediate finite cover of $\pi:M_H\to M$ such that $p|_K=k:K\to M'$ is an embedding. Since $K$ is an arbitrary compact subset of $M_H$, by Lemma \ref{Scott}, $H$ is a separable subgroup of $\pi_1(M)$.

\end{proof}

\subsection{Proof for the general case}\label{filling}

Now we are are ready to prove "aspirality implies separability"  for a general $3$-manifold as in Theorem \ref{main}.

\begin{proof}
Let $M$ be a compact, orientable, irreducible, $\partial$-irreducible $3$-manifold with nontrivial torus decomposition and does not support the $\text{Sol}$-geometry. Let $H<\pi_1(M)$ be an aspiral finitely generated infinite index subgroup. Then we will prove that $H$ is separable in $\pi_1(M)$.

If $M$ has empty or tori boundary, then the separability of $H<\pi_1(M)$ directly follows from Proposition \ref{toriboundaryseparability}. So we only need to consider the case that $M$ has a boundary component with higher genus (genus at least $2$).

We take the torus decomposition $\mathcal{T}\subset M$, then each component $V$ of $M\setminus \mathcal{T}$ is either a Seifert space or atoroidal. If $V$ is atoroidal and has higher genus boundary, then $V$ is irreducible, $\partial$-irreducible and atoroidal, but may not be acylindrical.

For each $g\geq 2$, there exists a compact hyperbolic $3$-manifold $N_g$, such that $\partial N_g$ is a totally geodesic closed surface of genus $g$. The manifold $N_2$ is explicitly constructed in Section 3.2 of \cite{Thu1}, and $N_g$ can be obtained by taking a $(g-1)$-sheet cyclic cover of $N_2$, since $H_1(\partial N_2;\mathbb{Z})\to H_1(N_2;\mathbb{Z})$ is surjective.

For each component $V\subset M\setminus \mathcal{T}$ with higher genus boundary, and each boundary component $\Sigma_g\subset \partial V$ with $g\geq 2$, we attach $N_g$ to $V$ along $\Sigma_g$ (by any homeomorphism) to get a $3$-manifold $\bar{V}$ with tori boundary. Now $\Sigma_g$ is incompressible in both $V$ and $N_g$, $V$ and $N_g$ are both irreducible and atoroidal, while $N_g$ is also acylindrical. A classical argument in $3$-manifold topology implies that $\bar{V}$ is irreducible and atoroidal. So the hyperbolization theorem of Haken manifolds (\cite{Thu3}, see also \cite{Mor}) implies that $\bar{V}$ supports a complete hyperbolic structure with finite volume. By our construction, the inclusion $V\to \bar{V}$ is $\pi_1$-injective, and $\pi_1(V)<\pi_1(\bar{V})$ is a geometrically finite subgroup (since it is not a virtually fibered subgroup).

After doing the above filling construction for each vertex space $V\subset M\setminus \mathcal{T}$ to get $\bar{V}$, we get a compact, orientable, irreducible $3$-manifold $\bar{M}$ with empty or tori boundary, such that $M$ is a submanifold of $\bar{M}$ and the inclusion $M\hookrightarrow \bar{M}$ is $\pi_1$-injective. Each component of $\bar{M}\setminus \mathcal{T}$ is either a component of $M\setminus \mathcal{T}$ with tori boundary, or a finite volume hyperbolic $3$-manifold. So the collection of tori $\mathcal{T}\subset M\subset \bar{M}$ also gives the torus decomposition of $\bar{M}$. Let the covering spaces of $M$ and $\bar{M}$ corresponding to $H<\pi_1(M)$ and $H<\pi_1(\bar{M})$ be denoted by $M_H$ and $\bar{M}_H$ respectively, then $M_H$ is naturally a submanifold of $\bar{M}_H$. Since $\mathcal{T}$ gives the torus decomposition of both $M$ and $\bar{M}$, the induced graph of space structures on $M_H$ and $\bar{M}_H$ are compatible with each other.

For each vertex space $V_H\subset M_H$, let $\bar{V}_H$ be the vertex space of $\bar{M}_H$ containing $V_H$, then $\pi_1(V_H)=\pi_1(\bar{V}_H)$ holds. We use $V$ and $\bar{V}$ to denote the image of $V_H$ and $\bar{V}_H$ in $M$ and $\bar{M}$ respectively. If $V_H\ne \bar{V}_H$, then $V$ has a boundary component of genus at least $2$, and $\bar{V}$ is a hyperbolic $3$-manifold with finite volume. Since $\pi_1(V)<\pi_1(\bar{V})$ is a geometrically finite subgroup, $\pi_1(\bar{V}_H)=\pi_1(V_H)<\pi_1(V)<\pi_1(\bar{V})$ is also geometrically finite. When $V_H=\bar{V}_H$, one is virtually fibered or partially fibered if and only if the other one is.

Let $\Phi_M(H)$ be the almost fibered surface of $H$ defined by $H<\pi_1(M)$, and let $\Phi_{\bar{M}}(H)$ be the almost fibered surface of $H$ defined by $H<\pi_1(\bar{M})$. Then the above argument implies that $V_H \subset M_H$ contributes to $\Phi_M(H)$ if and only if the corresponding piece $\bar{V}_H\subset \bar{M}_H$ contributes to $\Phi_{\bar{M}}(H)$, so $\Phi_M(H)=\Phi_{\bar{M}}(H)$ holds. Then it is easy to check that they have the same generalized spirality character (defined by $M$ and $\bar{M}$ respectively).

So the aspirality of $H$ in $M$ implies that $H$ is aspiral in $\bar{M}$. Since $\bar{M}$ has empty or tori boundary, Proposition \ref{toriboundaryseparability} implies that $H$ is separable in $\pi_1(\bar{M})$. Since $\pi_1(M)$ is a subgroup of $\pi_1(\bar{M})$ containing $H$, Lemma \ref{basicLERF} (1) implies that $H$ is separable in $\pi_1(M)$.
\end{proof}

\section{LERFness of finitely generated $3$-manifold groups}\label{3manifold}

In this section, we use Theorem \ref{main} to prove Theorem \ref{allmanifold}. Actually, the proof only uses the "aspirality implies separability" part of Theorem \ref{main}.

\begin{proof}
  We first deal with the case that $M$ is compact, orientable, irreducible and $\partial$-irreducible. If the torus decomposition of $M$ is trivial, then it is either Seifert of atoroidal. By \cite{Sco1}, Seifert manifolds have LERF fundamental groups. If $M$ is atoroidal, then $M$ supports a complete hyperbolic structure. The Haken case follows from \cite{Thu3} (see also \cite{Mor}), and the non-Haken case follows from \cite{Per1}, \cite{Per2} (see also \cite{BBMBP},\cite{KL} and \cite{MT}). Then works of Wise and Agol prove that $\pi_1(M)$ is LERF (\cite{Wise}, \cite{Agol2}). So we can assume that $M$ has nontrivial torus decomposition. Note that $3$-manifolds supporting the $\text{Sol}$ geometry have virtually polycyclic fundamental groups, and these groups are also LERF.

Under the torus decomposition of $M$, if there are two adjacent pieces $M_1, M_2\subset M$ such that both of them have tori boundary ($M_1$ might be same with $M_2$), then $M_1\cup M_2$ is a $\pi_1$-injective submanifold of $M$. Since $M$ does not support the $\text{Sol}$ geometry, neither of $M_1$ and $M_2$ is homeomorphic to $T^2\times I$ and at least one of them is not homeomorphic to the twisted $I$-bundle over Klein bottle. By the minimal property of torus decomposition, $M_1\cup M_2$ is a mixed manifold or graph manifold. Then Theorem \ref{toriboundary} implies that $\pi_1(M_1\cup M_2)$ is not LERF. Since $\pi_1(M_1\cup M_2)$ is a subgroup of $\pi_1(M)$, Lemma \ref{basicLERF} (4) implies that $\pi_1(M)$ is not LERF.

  Now we suppose that for any two adjacent pieces of $M$, at least one of them has higher genus boundary. For any finitely generated infinite index subgroup $H<\pi_1(M)$, the covering space $M_H\to M$ corresponding to $H$ has an induced graph of space structure. For any piece of $M_H$ that covers  a piece of $M$ with higher genus boundary, since it corresponds to a geometrically finite subgroup of a hyperbolic $3$-manifold group, it has no contribution to the almost fibered surface $\Phi(H)$. So only those pieces of $M_H$ that cover pieces of $M$ with tori boundary may contribute to $\Phi(H)$. Since any two such pieces of $M_H$ are not adjacent to each other, the dual graph $G_{\Phi(H)}$ of $\Phi(H)$ consists of finitely many vertices and has no edges. Since the generalized spirality character $s:H_1(\Phi(H);\mathbb{Z})\to \mathbb{Q}_+^{\times}$ factors through $H_1(G_{\Phi(H)};\mathbb{Z})$ (Remark \ref{factorthrough}), it must be the trivial homomorphism. Then Theorem \ref{main} implies that $H$ is separable in $\pi_1(M)$, and $\pi_1(M)$ is LERF.

\bigskip

 If $M$ is only compact and orientable, let $M_1,M_2,\cdots,M_n$ be the pieces of $M$ under the sphere-disc decomposition. Then $\pi_1(M)$ is a free product of $\pi_1(M_1),\pi_1(M_2),\cdots,\pi_1(M_n)$, and Lemma \ref{basicLERF} (5) implies that $\pi_1(M)$ is LERF if and only if all of $\pi_1(M_1),\pi_1(M_2),\cdots,\pi_1(M_n)$ are LERF.
\end{proof}

\end{document}